\documentclass{article}
\usepackage{amsmath,amssymb,mathrsfs,amsthm}
\usepackage{array,enumerate}
\usepackage{latexsym}
\usepackage{mathrsfs}
\usepackage{amscd}
\usepackage[all]{xy}
\usepackage{graphicx}
\usepackage[svgnames]{xcolor}
\theoremstyle{definition}
\newtheorem{thm}{Theorem}[section]
\newtheorem{dfn}[thm]{Definition}
\newtheorem{nota-dfn}[thm]{Notation-Definition}
\newtheorem{nota-lem}[thm]{Notation-Lemma}
\newtheorem{nota}[thm]{Notation}
\newtheorem{exam}[thm]{Example}
\newtheorem{prop}[thm]{Proposition}
\newtheorem{cor}[thm]{Corollary}
\newtheorem{lem}[thm]{Lemma}
\newtheorem{rem}[thm]{Remark}

\newtheorem{sub}[thm]{}
\newtheorem{prop-dfn}[thm]{Proposition-Definition}

\newcommand{\N}{\mathbb{N}}
\newcommand{\Z}{\mathbb{Z}}
\newcommand{\bP}{\mathbb{P}}

\newcommand{\R}{\mathbb{R}}

\newcommand{\bG}{\mathbb{G}}
\newcommand{\cC}{\mathcal{C}}
\newcommand{\cD}{\mathcal{D}}
\newcommand{\cE}{\mathcal{E}}

\newcommand{\cJ}{\mathcal{J}}

\newcommand{\cL}{\mathcal{L}}

\newcommand{\cN}{\mathcal{N}}
\newcommand{\cO}{\mathcal{O}}
\newcommand{\cP}{\mathcal{P}}
\newcommand{\cU}{\mathcal{U}}
\newcommand{\C}{\mathbb{C}}

\newcommand{\fm}{\mathfrak{m}}
\newcommand{\fp}{\mathfrak{p}}

\newcommand{\bpm}{\begin{pmatrix}}
\newcommand{\epm}{\end{pmatrix}}

\newcommand{\Gal}{\mathrm{Gal}}

\newcommand{\gp}{\mathrm{gp}}

\newcommand{\Hom}{\mathrm{Hom}}

\newcommand{\node}{\mathrm{node}}
\newcommand{\lreg}{\mathrm{lreg}}

\newcommand{\ncd}{\mathrm{ncd}}
\newcommand{\Frac}{\mathop{\mathrm{Frac}}\nolimits}

\newcommand{\Pic}{\mathrm{Pic}}
\newcommand{\Proj}{\mathop{\mathrm{Proj}}\nolimits}

\newcommand{\id}{\mathrm{id}}
\newcommand{\rank}{\mathrm{rank}\,}
\newcommand{\red}{\mathrm{red}}
\newcommand{\reg}{\mathrm{reg}}

\newcommand{\Spec}{\mathrm{Spec}\,}
\newcommand{\sep}{\mathrm{sep}}
\newcommand{\sncd}{\mathrm{sncd}}

\newcommand{\Supp}{\mathrm{Supp}\,}
\newcommand{\tor}{\mathrm{tor}}
\newcommand{\Zar}{\mathrm{Zar}}

\title{Minimal log regular models of hyperbolic curves over discrete valuation fields}
\author{Ippei Nagamachi}
\date{}
\begin{document}
\maketitle

\setcounter{section}{-1}

\begin{abstract}
In the famous paper of Deligne and Mumford, they proved that a proper hyperbolic curve over a discrete valuation field has stable reduction if and only if the Jacobian variety of the curve has stable reduction in the case where the residue field of its valuation ring is algebraically closed.
In the proof, the theory of minimal regular models played an important role.
In this paper, we establish a theory of minimal log regular models of curves.
As a key tool for this theory, we give a criterion for $2$-dimensional local schemes to be log regular in terms of their minimal desingularization.
Moreover, as an application of this theory, we prove the above equivalence without the assumption on the residue field.
\end{abstract}

\tableofcontents

\section{Introduction}
\label{introsection}
Let $K$ be a discrete valuation field, $O_{K}$ the valuation ring of $K$, and $k$ the residue field of $O_{K}$.
Let $C$ be a proper smooth curve over $K$ with geometrically connected fibers of genus $g$.
Let $J$ be the Jacobian variety of $C$, $\cJ$ the N\'eron model of $J$ over $O_{K}$, and $\cJ_{k}$ the special fiber of $\cJ$. 
Recall that $J$ is said to have stable reduction if the unipotent radical of $\cJ_{k}$ is trivial (cf.\,\cite[Definition 2.1]{DM}).
In \cite{DM}, the following theorem is proven:

\begin{thm}[{\cite[Theorem (2.4)]{DM}}]
\label{DMtheorem2.4}
Suppose that $g\geq 2$ and $k$ is algebraically closed.
Then $C$ has stable reduction (cf.\,Definition \ref{models}) if and only if $J$ has stable reduction.
\end{thm}

Note that, by using elementary descent theory, we can generalize this theorem to the case where $k$ is perfect (and not necessarily algebraically closed).
In the proof of \cite[Theorem (2.4)]{DM}, the theory of minimal regular models of curves plays the key role.
In this paper, for (not necessarily proper) curves over $K$, we establish a theory of minimal log regular models similar to the theory of minimal regular models and give a generalization of Theorem \ref{DMtheorem2.4} to the case where $k$ is not necessarily perfect (cf.\,Corollary \ref{intromain1}).

Here, we give a brief description of pointed curves and their log regular models (for their precise definitions, see \ref{lognotationa}, \ref{curvedfn}, and Notation-Definition \ref{models}).
Let $D$ be a reduced closed subscheme of $C$ finite over $K$ and $r$ the rank of $D$ over $\Spec K$.
Then $(C,C\setminus D)$ is a toric pair, i.e., there exists a log structure on $C$ such that the obtained log scheme is log regular and the open subscheme $C\setminus D$ of $C$ is the trivial locus of the log structure (cf.\,\ref{lognotationa}).
We write $C^{\log}$ or $(C,D)$ for the resulting log scheme.
We define a log regular model $\cC^{\log}$ of $C^{\log}$ to be a log regular scheme $\cC^{\log}$ whose underlying scheme $\cC$ is a model of $C$ such that $(\cC,C\setminus D)$ is a toric pair.
In the following, for any log scheme $X^{\log}$, we write $X$ for the underlying scheme of $X^{\log}$.

Part of the theory of minimal log regular models established in this paper is described as follows:
\begin{thm}
\label{introdefinition}
Suppose that $2g+r-2>0$ and $D$ is \'etale over $\Spec K$.
\begin{enumerate}
\item
We have the ``smallest" minimal log regular model $\cC_{\lreg}^{\log}$ of $C^{\log}$ in a certain sense (cf.\,Theorem \ref{logreglem}, Notation-Definition \ref{lregnota}, and Proposition \ref{mlregstr}).
\item
If $C^{\log}$ has stable reduction (resp.\,log smooth reduction) (cf.\,Notation-Definition \ref{models}), $\cC^{\log}_{\lreg}$ is a stable model (resp.\,a log smooth model) of $C^{\log}$ (cf.\,Proposition \ref{logsmoothcase} and Proposition \ref{stablecoin}).
\item
Suppose that $C^{\log}$ has log smooth reduction.
Let $O_{K'}\supset O_{K}$ be an extension of discrete valuation rings whose ramification index is $1$ and $K'$ the field of fractions of $O_{K'}$.
Write $(C_{K'},D_{K'}):=(C\times_{\Spec K}\Spec K',D\times_{\Spec K}\Spec K')$ and $\cC'^{\log}_{\lreg}$ for the ``smallest" minimal log regular model``$\cC_{\lreg}^{\log}$" for $(C_{K'},D_{K'})$ (cf.\,Theorem \ref{introdefinition}.1).
Then $\cC_{\lreg,O_{K'}}:=\cC_{\lreg}\times_{\Spec O_{K}}\Spec O_{K'}$ is canonically isomorphic to $\cC'_{\lreg}$ (cf.\,Proposition \ref{basechange}).
\end{enumerate}
\end{thm}

This theory enables us to generalize Theorem \ref{DMtheorem2.4} as follows:

\begin{cor}[cf.\,{\cite[Theorem (2.4)]{DM}}, Theorem \ref{generalmain}]
\label{intromain1}
Suppose $g\geq 2$.
Then $C$ has stable reduction if and only if $J$ has stable reduction.
\end{cor}

\begin{proof}[Proof of ``if" part of Corollary \ref{intromain1}]
It suffices to show that $\cC_{\lreg}$ is a stable curve over $O_{K}$.
Let $O_{K'}$ be a discrete valuation ring such that $O_{K'}$ satisfies the assumption of Theorem \ref{introdefinition}.3 and the residue field of $O_{K'}$ is algebraically closed.
Since $O_{K'}$ is faithfully flat over $O_{K}$, it suffices to show that $\cC_{\lreg,O_{K'}}$ is a stable curve over $O_{K'}$.
By \cite[Theorem 4.2 (3)$\Rightarrow$(1)]{Sai2} and  \cite[Exp.IX, Proposition 3.9]{SGA7}, $C^{\log}$ has a log smooth model (cf.\,Notation-Definition \ref{models}).
From this and Theorem \ref{introdefinition}.3, it follows that $\cC_{\lreg,O_{K'}}\simeq\cC'_{\lreg}$.
Since $C_{K'}$ has stable reduction by \cite[Theorem (2.4)]{DM} and the assumption that $J$ has stable reduction, $\cC'_{\lreg}(=\cC_{\lreg,O_{K}})$ is a stable curve over  $O_{K'}$ by Theorem \ref{introdefinition}.2.
\end{proof}
 
We give the proof of ``only if" part and a generalization of Corollary \ref{intromain1} as part of Theorem \ref{generalmain}.

In this paper, as applications of the theory of minimal log regular models, we give generalizations of previous works on models of curves over discrete valuation fields with algebraically closed residue fields.
Here, we give comparisons of our results and other papers.
\begin{enumerate}
\item
The same statement as that of Corollary \ref{intromain1} (and Corollary \ref{monodromy}) is given in \cite[Proposition 1.14]{Sai2} without proofs.
However, the author of the present paper could not find proofs of \cite[Proposition 1.14]{Sai2} in the literature.
In the present paper, we give its proof by using only \cite[Theorem 4.2 (3)$\Rightarrow$(1)]{Sai2} (which can be proved without \cite[Proposition 1.14]{Sai2}).
\item
A log smooth version of Corollary \ref{intromain1} is given in \cite[Theorem 1]{Sai2}.
Also, in the case where $k$ is algebraically closed and $D=\emptyset$, similar theorems written in terms of monodromy Galois actions are given in  \cite[THEOREM (3.11)]{Sai1}, \cite{Stix}, and \cite[Theorem 1]{Lo}.
Since regular varieties are not smooth in general if the coefficient fields are not perfect, a naive analogue of \cite[Theorem 1]{Lo} does not hold in general.
In this paper, we give a modified version (cf.\,Corollary \ref{wildmonodromy}).
\item
Suppose that $C$ has log smooth reduction (cf.\,Notation-Definition \ref{models}) and $g\geq 2$.
By \cite[Theorem 4.2]{Sai2}, there exists a tamely ramified extension $L'/K$ such that $C\times_{\Spec K}\Spec L'$ has stable reduction.
Moreover, by \cite[Theorem 1.8]{Sai2}, there exists a positive integer $e$ such that, for any extension of discrete valuation fields $L/K$ whose ramification index is divisible by $e$, $C\times_{\Spec K}\Spec L$ has stable reduction.
(Note that, in \cite{Sai2}, the case where $D\neq\emptyset$ is also treated.)
In the case where $O_{K}$ is complete and $k$ is algebraically closed, it is proved in \cite[Section 7]{H} that there exists a positive integer $e_{\min}$ which divides all such $e$ and satisfies the condition on $e$.
Furthermore, in \cite[Section 7]{H}, a characterization of this $e_{\min}$ in terms of the multiplicities of irreducible components of the special fiber of the minimal regular n.c.d.\,model of $C$ is given.
In this paper, we give a generalization of this result (cf.\,Proposition \ref{genH}).
\end{enumerate}

Next, we explain how to construct (the ``smallest" minimal log regular model) $\cC_{\lreg}^{\log}$.
First, we prove the uniqueness of minimal regular normal crossing divisor models (, which we write $\cC_{\ncd}$ for,) of hyperbolic curves (cf.\,Lemma \ref{uniqueNCD}).
Then, by applying Lipman's contractibility criterion to a suitable divisor whose support is contained in the special fiber of $\cC_{\ncd}$, we obtain a normal model $\cC_{\lreg}$ which has at most rational singularities.
Finally, we prove that $(\cC_{\lreg},C\setminus D)$ is a toric pair.
The most difficult part in the construction is the determination of the divisors contractible to a point in a log regular model.
This issue is treated in Section \ref{logregularsecton}.

In the following, we explain the log regularity criterion given in Section \ref{logregularsecton}.
Let $(R,\fm)$ be a strictly henselian Noetherian local normal domain of dimension $2$ which is \textbf{not} regular.
Write $X$ for the scheme $\Spec R$ and $x$ for the unique closed point of $X$.
Let $D$ be a reduced closed subscheme of $X$ of dimension $1$.
We give a criterion to determine whether the pair $(X,X\setminus D)$ becomes a toric pair.
Let $\pi:\widetilde{X}\to X$ be a desingularization, i.e., a proper birational morphism from a connected regular scheme.
Write $E$ for the exceptional divisor of $\pi$ and $D'$ for the strict transform of $D$ in $\widetilde{X}$.
Note that these are effective reduced Cartier divisors on $\widetilde{X}$.

\begin{dfn}
\label{introdivisorconditions}
We consider the following two conditions:
\begin{itemize}
\item[(Exc)]
$E$ is a normal crossing divisor on $\widetilde{X}$ and $E$ can be written as $E=\sum_{i=1}^{m}E_{i}$ such that the following conditions are satisfied:
(i)  For each $1\leq i\leq m$, $E_{i}$ is isomorphic to $\bP^{1}_{k(x)}$ over $k(x)$.
(ii) For each $1\leq i\leq m$, $\cN_{E_{i}/\widetilde{X}}\simeq O_{\bP^{1}_{k(x)}}(a_{i})$ for some negative integer $a_{i}\leq -2$.
(iii) For $1\leq i< m$, there exists a $k(x)$-rational point $e_{i}$ such that $E_{i}\cap E_{i+1}=\{e_{i}\}$.
(iv) If $|i-j|\geq 2$, $E_{i}\cap E_{j}=\emptyset$.
(Note that $E$ is connected and we have $e_{i}\neq e_{i+1}$.)
\end{itemize}
Suppose condition (Exc) is satisfied.
\begin{itemize}
\item[(Str)]
$D'+E$ is a normal crossing divisor on $\widetilde{X}$ and $D'\cap E$ consists of exactly two $k(x)$-rational points.
Moreover, we have $D'\cap E_{i}=\emptyset\;(1<i<m)$ and $D'\cap E_{1}\neq\emptyset\neq D'\cap E_{m}$.
\end{itemize}
\end{dfn}

If condition (Exc) is satisfied, $\widetilde{X}$ coincides with the minimal desingularization of $X$ (cf.\,Lemma \ref{rationalsingularity}).
If condition (Str) is satisfied, $D'$ can be written as $E_{0}+E_{m+1}$, where $E_{0}$ and $E_{m+1}$ are prime divisors (cf.\,Lemma \ref{notationnormalcrossing}).
These conditions characterize log regularities.

\begin{thm}[cf.\,Theorem \ref{logsummary}]
\label{intromaincharacterization}
The following are equivalent:
\begin{enumerate}
\item
$(X,X\setminus D)$ is a toric pair.
\item
There exists a desingularization of $X$ satisfying conditions (Exc) and (Str) in Definition \ref{introdivisorconditions}.
\end{enumerate}
\end{thm}

Note that, in Theorem \ref{logsummary}, we give a generalized version of this theorem.
In the construction of a chart in the proof of ($2\Rightarrow1$ of) Theorem \ref{intromaincharacterization}, we need to prove the following:
(i) $M/R^{\ast}$ is an f.s.\,monoid, where $M:=\{r\in R\mid r\text{ is invertible outside }D\}$.
(ii) A section of the monoid homomorphism $M\to M/R^{\ast}$ is a desired chart (in particular, the image of $(M/R^{\ast})\setminus\{1\}$ generates $\fm$ as an ideal of $R$).
To prove (i), we first consider the short exact sequence
$$0\to O_{\widetilde{X}}^{\ast}\to (O_{\widetilde{X}}\cap j_{\ast}O_{\widetilde{X}\setminus (D'\cup E)}^{\ast})^{\gp}
\to (O_{\widetilde{X}}\cap j_{\ast}O_{\widetilde{X}\setminus (D'\cup E)}^{\ast})/O_{\widetilde{X}}^{\ast}\to0,$$
where $j$ is the open immersion $j:\widetilde{X}\setminus (D'\cup E)\to \widetilde{X}$.
By using the long exact sequence of the cohomology groups associated to this exact sequence (cf.\,(\ref{picexact}) in the proof of Proposition \ref{definelogstructure}), we obtain an exact sequence of groups
\begin{align*}
0\to(M/R^{\ast})^{\gp}\to\Z^{\oplus m+2}\to\Pic(\widetilde{X}),
\end{align*}
where we can identify $\Z^{\oplus m+2}$ with the subgroup of the Weil divisor group of $\widetilde{X}$ generated by $E_{0},\cdots,E_{m+1}$.
The structure of $\Pic(\widetilde{X})$ can be explained using Lipman's result since $R$ has a rational singularity (cf.\,Lemma \ref{rationalsingularity}).
In this way, we can prove (i).
To show (ii), we use the isomorphism
$$H^{0}(\widetilde{X},O_{\widetilde{X}}(-E)|_{E})\simeq \fm/\fm^{2}$$
which follows from the fact that $R$ has a rational singularity (cf.\,Lemma \ref{rationalsingularity}).
It follows from this isomorphism that the image of $(M/R^{\ast})\setminus\{1\}$ in the linear space $H^{0}(\widetilde{X},O_{\widetilde{X}}(-E)|_{E})$ is a system of generators and hence (ii) holds (cf.\,Proposition \ref{definelogstructure}).

The contents of each section and each subsection are as follows:
In section \ref{logregularsecton},  we study singularities and desingularizations of log regular schemes of dimension $2$.
In subsection \ref{convexsubsection}, we study $2$-dimensional strongly convex rational polyhedral cones.
In subsection \ref{exceptionsdivisorsubsection}, we study the exceptional divisors of the minimal desingularizations of log regular schemes of dimension $2$.
In subsection \ref{logblowupsubsection}, we study log blow-ups of log regular schemes of dimension $2$.
In subsection \ref{summarysubsection}, we give the log regularity criterion.
In section \ref{modelsection},  we study models of log regular curves and establish a theory of minimal log regular curves.
In subsection \ref{defsubsection}, we give the definitions of various models of curves over discrete valuation rings.
In subsection \ref{minimalmodelconstructionsection}, we study fundamental properties of minimal regular n.c.d.\,models.
In subsection \ref{lregsubsection}, we study fundamental properties of log regular models and define $\cC^{\log}_{\lreg}$.
In section \ref{thirdsection}, we study properties of $\cC_{\lreg}^{\log}$.
In subsection \ref{lsmsubsection}, we discuss the fundamental structures of log smooth models of log regular curves.
In subsection \ref{basechangesubsection}, we discuss the base changes of $\cC^{\log}_{\ncd}$ and $\cC^{\log}_{\lreg}$.
In subsection \ref{stablesubsection}, we study stable models and generalize \cite[Theorem (2.4)]{DM}.
In subsection \ref{psubsection}, we discuss prime divisors whose supports are contained in the special fibers of log smooth models whose multiplicities are divisible by $p$ and give a generalization of \cite[THEOREM (3.11)]{Sai1}, \cite[Theorem 4.2]{Sai2}.

\noindent \textbf{Acknowledgement}:
The author of this paper thanks Yuichiro Hoshi, Yuya Matsumoto, Takeshi Saito, Teppei Takamatsu, Akio Tamagawa, and Naganori Yamaguchi for helpful comments.
In particular, the auther thanks Akio Tamagawa and Naganori Yamaguchi for giving him the opportunity of considering this problem.
He is supported by JSPS KAKENHI Grant Number 21J00489.

\subsection{Notations and Conventions}

In this paper, for a scheme $X$, we shall write $X_{\red}$ for the maximal reduced closed subscheme of $X$.
For morphisms of schemes $X\to Y$ and $Z\to Y$, we shall write $X_{Z}$ for the scheme $X\times_{Y}Z$ over $Z$.
For any finite set $S$, we shall write $\sharp S$ for the cardinality of $S$.

\begin{sub}[Log schemes]
\label{lognotationa}
Let $(X,M_{X},\alpha)$ be a triple such that $X$ is a scheme, $M_{X}$ is a sheaf of monoids on the \'etale site of $X$, and $\alpha:M_{X}\to O_{X}$ is a monoid homomorphism (to the multiplicative monoid sheaf $O_{X}$) satisfying $\alpha^{-1}O_{X}^{\ast}\simeq O_{X}^{\ast}$.
We shall refer to such a triple $(X,M_{X},\alpha)$ as a {\em log scheme} (cf.\,\cite{Ka1}).
We often write $X^{\log}$ or $(X,M_{X})$ for such a log scheme $(X,M_{X},\alpha)$.
In this paper, any log scheme $(X,M)$ is said to be {\em fine saturated} (or, an {\em f.s.\,log scheme}) if the following condition (S) is satisfied (cf.\,\cite[(1.5)]{Ka2}):\\
(S): $X$ is locally Noetherian and there exist an \'etale covering $\{U_{\lambda}\}_{\lambda}$ of $X$, f.s.\,monoids $\{P_{\lambda}\}_{\lambda}$, and homomorphisms $\{h_{\lambda}:P_{\lambda}\to \cO_{U_{\lambda}}\}_{\lambda}$ such that $M|_{U_{\lambda}}$ is isomorphic to the log structure associated to the pre-log structure $(P_{\lambda},h_{\lambda})$ for each $\lambda$ (, which we write $P_{\lambda}^{a}$ for).
We write $\times^{\log}$ for the fiber product in the category of f.s.\,log schemes.

Let $(X,M_{X})$ be a log scheme.
For $x\in X$ and a geomertic point $\overline{x}\to X$ over $x$, we shall write $I(\overline{x},M_{X})$ for the ideal of $O_{X,\overline{x}}$ generated by the image of $M_{X,\overline{x}}\setminus O_{X,\overline{x}}^{\ast}$ (cf.\,\cite[(2.1) Definition]{Ka2} and \cite[Definition 2.2]{Niz}).
Recall that, for an f.s.\,log scheme $X^{\log}=(X,M_{X})$ and a point $x\in X$, $X^{\log}$ is said to be {\em log regular} at $x$ if, for a geometric point $\overline{x}$ over $x$, the following two conditions hold:
\begin{itemize}
\item[(i)]
$O_{X,\overline{x}}/I(\overline{x},M)$ is a regular local ring,
\item[(ii)]
$\dim(O_{X,\overline{x}})=\dim(O_{X,\overline{x}}/I(\overline{x},M))+\rank(M_{X,\overline{x}})^{\gp}/\cO_{X,\overline{x}}^{\ast}$.
\end{itemize}
Note that \cite{Ka2} treats Zariski log regular schemes (cf.\,\cite[(2.1) Definition]{Ka2}).
In the case where there exists a chart $P\to \Gamma(X,M_{X})$ such that $P$ is a fine saturated monoid, $(X,M_{X})$ is log regular if and only if $(X,M_{X}^{\Zar})$ is Zariski log regular by \cite[Lemma II.4.6]{Ts}.
Here, $M_{X}^{\Zar}$ is the Zariski sheaf associated to $M_{X}$.
Suppose that $(X,M_{X})$ is a log regular log scheme.
We shall write $F(X):=\{x\in X\mid I(\overline{x},M_{X})=\fm_{\overline{x}}\}$ and $F(X)^{(i)}:=\{x\in F(X)\mid \dim O_{X,x}=i\}$ for $i\in\N$.
By the discussion in \cite[(10.4)]{Ka2}, $F(X)$ is a finite set if $X$ is quasi-compact.
Let $x$ be a point of $X$.
Take a geometric point $\overline{y}$ of $\Spec O_{X,\overline{x}}$ whose image corresponds to the prime ideal $I(\overline{x},M_{X})$ of $O_{X,\overline{x}}$ and write $y$ for the image of $\overline{y}$ in $X$.
We have a cospecialization homomorphism $M_{X,\overline{x}}\to M_{X,\overline{y}}$, which is an isomorphism by \cite[Lemma 2.12(4)]{Niz}.

We shall refer to a pair of schemes $(Y,U)$ as a {\em toric pair} if there exists a log structure $M_{Y}\to O_{Y}$ such that the log scheme $(Y,M_{Y})$ is log regular and $U=\{y\in Y\mid M_{Y,\overline{y}}= O_{Y,\overline{y}}^{\ast}\}$.
For a toric pair $(Y,U)$, we shall write $j_{U}$ for the open immersion $U\to Y$.
Recall that, if $(Y,M_{Y})$ is a log regular log scheme, then $Y$ is normal, $Z:=\{y\in Y\mid M_{Y,\overline{y}}\neq O_{Y,\overline{y}}^{\ast}\}$ is a closed subset of $Y$ of pure of codimension $1$, and $M_{Y}\simeq j_{Y\setminus Z,\ast}O_{Y\setminus Z}^{\ast}\cap O_{Y}$.
(See the paragraph before \cite[Definition 1.2]{M} and the paragraph before \cite[Proposition 1.1]{Sai2}.)
Hence, the log structure in the definition of a toric pair is uniquely determined from the pair.
For a given toric pair $(Y,U)$, we write $Y^{\log}$ for the log scheme.
\end{sub}

Here, we give an elementary lemma.

\begin{lem}
\label{stalkwiselogregular}
Let $X$ be a locally Noetherian scheme, $D$ a closed subscheme of $X$ of pure of codimension $1$, and $U$ the complement $X\setminus D$.
\begin{enumerate}
\item
Suppose that $(X,U)$ is a toric pair such that $X^{\log}$ is Zariski log regular.
Then $F(X)$ coincides with the subset of $X$ consisting of the generic points of the intersections of finitely many irreducible components of $D$.
\item
The following are equivalent:
(i) $(X,U)$ is a toric pair.
(ii) For any $x\in X$, $(\Spec O_{X,x}, U\times_{X}\Spec O_{X,x})$ is a toric pair.
\end{enumerate}
\end{lem}

\begin{proof}
First, we show assertion 1.
Let $x$ be a point of $X$.
Write $F'(X)$ (resp.\,$F'(\Spec O_{X,x})$) for the subset of $X$ consisting of the generic points of the intersections of finitely many irreducible components of $D$ (resp.\,$D\times_{X}\Spec O_{X,x}$).
Then $x\in F'(X)$ if and only if $x\in F'(\Spec O_{X,x})$.
Hence, we may assume that $X=\Spec O_{X,x}$.
It suffices to show $x\in F(X)\Leftrightarrow x\in F'(X)$.
Then assertion 1 follows from \cite[Corollary 2.3.8]{Ogus}, \cite[(7.3) COROLLARY]{Ka2}, \cite[(10.1) PROPOSITION]{Ka2} and its proof, and \cite[(10.2)]{Ka2}.

Next, we show assertion 2.
(i) $\Rightarrow$(ii) follows from the definition of the log regularity.
Suppose that (ii) holds.
Then $X$ is normal and $M_{X}:=j_{U,\ast}O_{U}^{\ast}\cap O_{X}\to O_{X}$ is a log structure on $X$.
We need to show that the log scheme $X^{\log}$ is log regular.
Let $x$ be a point of $X$ and $\overline{x}$ a geometric point of $X$ over $x$.
By (ii), $P_{\overline{x}}:=M_{X,\overline{x}}/O_{X,\overline{x}}^{\ast}$ is an f.s.\,sharp monoid.
We will show that, there exist an \'etale neighborhood $V\to X$ of $\overline{x}$ and a homomorphism $P_{\overline{x}}\to\Gamma(V,M_{X})$ which is a chart of $(V,M_{X}|_{V})$.
By the discussion after \cite[Definition 1.3]{M}, we have a chart $P_{\overline{x}}\to O_{X,\overline{x}}$ of $(\Spec O_{X,\overline{x}})^{\log}$ which induces the identity of $P_{\overline{x}}$.
Since $P_{\overline{x}}$ is finitely generated, we may assume that $X(=\Spec A)$ is affine and we have a monoid homomorphism $P_{\overline{x}}\to\Gamma(X,M_{X})$ inducing the identity of $P_{\overline{x}}$ by replacing $X$ with a suitable \'etale neighborhood of $\overline{x}$.
Furthermore, since $F'(X)$ is a finite set, we may assume that $x\in\overline{\{z\}}$ for any point $z\in F'(X)$.
Hence, we have $F'(X)=F'(\Spec O_{X,x})=F(\Spec O_{X,x})$ by assertion 1.

It suffices to show that $P^{a}_{\overline{x}}\to M_{X}$ is an isomorphism.
The injectivity follows from \cite[Proposition 1.2.2]{Ogus} and the fact that homomorphism $P_{\overline{x}}\to\Gamma(X,M_{X})$ is injective and the stalks of $M_{X}$ are integral.
To see the surjectivity of $P^{a}_{\overline{x}}\to M_{X}$, it suffices to show that, for any geometric point $\overline{y}$ of $X$, $P_{\overline{x}}\to M_{X,\overline{y}}/O_{X,\overline{y}}^{\ast}$ is surjective.
Take a geometric point $\overline{z}$ of $\Spec O_{X,\overline{y}}$ whose image corresponds to the prime ideal $I(\overline{y},M_{X})$ of $O_{X,\overline{y}}$ and write $z$ for the image of $\overline{z}$ in $X$.
Then the cosepcialization homomorphism $M_{X,\overline{y}}/O_{X,\overline{y}}^{\ast}\to M_{X,\overline{z}}/O_{X,\overline{z}}^{\ast}$ is an isomorphism by (ii).
Since $P_{\overline{x}}\to O_{X,x}$ is a chart of the log regular scheme $(\Spec O_{X,x})^{\log}$, the desired surjectivity follows from \cite[(10.2)]{Ka2}.
\end{proof}

\begin{sub}[A ``node"-like local ring]
\label{node-like}
We discuss $1$-dimensional local rings with ``node"-like structures.
Let $D$ be a $1$-dimensional Noetherian reduced scheme, $d$ a closed point of $D$, $\fm_{d}$ for the maximal ideal of $O_{D,d}$, and $\pi:\widetilde{D}\to D$ the normalization of $D$.
Write $\Spec R$ for the scheme $\widetilde{D}\times_{D}\Spec O_{D,d}$.
Suppose that $\mathrm{length}_{O_{D,d}}R/O_{D,d}=1$ and $\widetilde{D}\times_{D}\Spec k(d)\to\Spec k(d)$ is finite \'etale of rank $2$.
Then $O_{D,d}$ coincides with the inverse image of the image of $k(d)$ in $R/\fm_{d}R$ via the natural surjection $R\to R/\fm_{d}R$.
Moreover, if $\widetilde{D}\to D$ is isomorphic over $D\setminus\{d\}$, the underlying topological space of $D$ is a quotient topological space of that of $\widetilde{D}$.
These show that $D$ is determined uniquely by $\widetilde{D}$ and the \'etale morphism $\pi^{-1}(d)\to \Spec k(d)$ of rank $2$.
\end{sub}

\begin{sub}[Curves (cf.\,{\cite{DM}}, {\cite{Kn}})]
\label{curvedfn}
Let $T$ be a scheme, $C$ a scheme over $T$, and $D$ a closed subscheme of $C$.
For natural numbers $g$ and $r$, we shall say that the pair of schemes $(C,D)$ is a {\em stable curve of type $(g,r)$} over $T$ if the following hold:
\begin{itemize}
\item
$C$ is a proper flat scheme over $T$ such that any geometric fiber of $C\to T$ is a connected curves of arithmetic genus $g$ which has only ordinary double points.
(We write $C^{\mathrm{sm}}$ for the smooth locus of $C$ over $T$.)
\item
The structure morphism $D\to T$ is finite \'etale of rank $r$ and $D$ is contained in $C^{\mathrm{sm}}$.
\item
For any geometric point $t$ of $T$ and any nonsingular rational component $E$ of $C_{t}$, we have $\sharp(E\cap D_{t})+\sharp(E\setminus C^{\mathrm{sm}}_{t})\geq3$.
\item
$2g+r-2>0$.
\end{itemize}
We shall say that a stable curve $(C,D)$ over $T$ is a {\em hyperbolic curve} if $C=C^{\mathrm{sm}}$.
In the case where we consider a stable curve $(C,D)$ of type $(g,0)$, we often simply write $C$ for the stable curve.

Suppose that $T=\Spec K$ for a field $K$.
We shall say that $(C,D)$ is a {\em log regular curve} over $K$ if $C$ is smooth proper over $K$ of relative dimension $1$ with geometrically connected fibers and $D$ is a reduced effective divisor of $C$.
Note that a hyperbolic curve over $K$ is a log regular curve.
Since $(C,C\setminus D)$ is a toric pair, we write $C^{\log}$ for the log scheme.

In this paper, for a $1$-dimensional Noetherian integral scheme $F$, we shall write $k(F)$ for the ring $H^{0}(F,O_{F})$ in the case where $k(F)$ is a field and $F$ is proper over $k(F)$.
We shall say that such a scheme $F$ is {\em of type $(g,0)$ over $k(F)$} if the arithmetic genus of $F$ over $k(F)$ is $g$.
\end{sub}

\begin{sub}[Normal crossing divisors]
\label{ncddef}
Let $D$ be an effective Cartier divisor on a scheme $X$.
In this paper, we also treat $D$ as a closed subscheme of $X$.
We shall refer to an integral effective Cartier divisor as a {\em prime divisor}.
We shall say that $D$ is a {\em simple normal crossing divisor} (or $D$ has {\em simple normal crossings}) if $X$ is regular around $D$ and, for any point $d\in D$, there exists a subset of regular system of parameters $f_{1},\cdots,f_{r}$ such that we have $D|_{\Spec O_{X,d}}\simeq V(f_{1}\cdots f_{r})(\subset \Spec O_{X,d})$.
We shall say that $D$ is a {\em normal crossing divisor} (or $D$ has {\em normal crossings}) if, for any point $d\in D$, there exist an \'etale neighborhood $U\to X$ of $d$ such that $D|_{U}$ is a simple normal crossing divisor on $U$.
If $D$ is a normal crossing divisor on $X$, we can define a natural log structure on $X$ (cf.\,\cite[(1.5)]{Ka1}).
If, moreover, $X$ is regular, the associated log scheme $X^{\log}$ is log regular and the associated toric pair is $(X,X\setminus D)$.
If the local equation of $D$ at $d$ is described as above \'etale locally, we have $d\in F(X)^{(r)}$.

Suppose that $X$ is a $2$-dimensional local scheme and $D$ has normal crossings.
Write $d$ for the unique closed point of $X$.
Suppose $d\in F(X)^{(2)}$.
One can see that $D$ satisfies the assumptions of ``$D$" in \ref{node-like}.
Indeed, we may assume that $D$ has simple normal crossings.
By taking a regular system of parameters $f_{1},f_{2}$ satisfying $V(f_{1}f_{2})\simeq D$ and considering the normalization $V(f_{1})\coprod V(f_{2})\to D$, we can show that the desired assumptions are satisfied.
We give a generalization of this fact in Theorem \ref{logsummary}.4.
\end{sub}

\begin{sub}[Rational curves which is irreducible and not geometrically irreducible]
\label{badcurves}
Let $k'/k$ be a separable extension of fields of degree $2$.
Let $C$ be a $1$-dimensional integral scheme proper over $k$ whose normalization $C_{n}(\xrightarrow{\pi} C)$ is isomorphic to $\bP^{1}_{k'}$ over $k$.
Suppose that $C$ is normal except at an exactly one ordinary double point $c$ which is $k$-rational.
(Note that, by \ref{ncddef}, $C$ is uniquely determined by $C_{n}$ and the point of $C_{n}$ over $c$.)
By considering the long exact sequence of the cohomology groups arising from the exact sequence of the coherent sheaves $0\to O_{C}\to\pi_{\ast}O_{C_{n}}\to k\to0$, we obtain $\dim_{k}H^{1}(C,O_{C})=0$.
We treat the notion of the degree of line bundles on $C$ or $C_{n}$ explained in \cite[Section 9.1]{BLR}.
For example, for any line bundle $\cL$ on $C$, we have $\deg_{C/k}\cL=\deg_{C_{n}/k}\cL|_{C_{n}}=2\deg_{C_{n}/k'}\cL|_{C_{n}}$.
We also have $\deg_{C_{n}/k}O_{\bP^{1}_{k'}}(a)=2a$ for any integer $a$.
\end{sub}

\begin{sub}[Exceptional curves of the first kind]
\label{firstkind}
Let $Y$ be a scheme projective over a Noetherian ring $A$.
A reduced effective Cartier divisor $E$ on $Y$ satisfying the following two conditions are referred to as an {\em exceptional curve of the first kind} (cf.\,\cite[Tag:0C2I]{Stacks}):
\begin{itemize}
\item[(Fir$_{1}$)]
There exist a field $k(E)$ and an isomorphism of schemes $\bP^{1}_{k(E)}\to E$.
(Note that $k(E)$ is the coefficient field of $E$.)
\item[(Fir$_{2}$)]
The restriction of the normal sheaf $\cN_{E/Y}$ of $E$ in $Y$ to $E$ is isomorphic to $O_{\bP^{1}_{k(E)}}(-1)$.
\end{itemize} 
For any Cartier divisor $E$ on $Y$, $E$ is an exceptional curve of the first kind if and only if there exist a scheme $Y'$ projective over $A$ and a regular closed point $y'\in Y'$ such that $Y$ is isomorphic to the blow-up of $Y'$ at $y'$ and $E$ can be regarded as the exceptional divisor of the blow-up via the isomorphism by \cite[Tag:0AGQ]{Stacks} and \cite[Tag:0C2M]{Stacks}.
\end{sub}

\begin{sub}[Blow-ups at points of normal crossing divisors]
\label{ncd}
Here, we give a generalization of \cite[Lemma 9.2.31]{Liu}.
Note that the excellence of the divisors in the assumption of \cite[Lemma 9.2.31]{Liu} is not used in the proof.

Let $Y'$ be a $2$-dimensional regular local scheme, $y'$ the unique point of $Y'$, $Y$ the blow-up of $Y'$ at $y'$, $E$ the exceptional divisor on $Y$, $F'$ a prime divisor on $Y'$ of dimension $1$, and $F$ the strict transform of $F'$.
$F'$ is regular if and only if $F+E$ has (simple) normal crossings and $F\cap E\simeq\Spec k(y')$.
(Indeed, we may assume $F\cap E\simeq\Spec k(y')$.
Then both conditions are equivalent to the condition that $f'\in\fm_{y'}\setminus\fm^{2}_{y'}$, where $f'$ is a local defining equation of $F'$.)
Let $F'_{i}$ be integral closed subschemes of $\Spec O_{Y',y'}$ of dimension $1$ and $F_{i}$ the strict transform of $F'_{i}$ for $i=1,2$.
Then $E+F_{1}+F_{2}$ has (simple) normal crossings if and only if $F'_{1}+F'_{2}$ does.
(Indeed, we may suppose that $F'_{i}$ is regular for $i=1,2$.
Then $E\cap F_{1}\cap F_{2}=\emptyset$ if and only if $F'_{1}+F'_{2}$ has simple normal crossings.)
Moreover, the following are equivalent:
(i) $F'$ has normal crossings and \textbf{not} regular.
(ii) $F+E$ has normal crossings and $F\cap E\to\Spec k(y')$ is finite \'etale of degree $2$.
(Indeed, by replacing $O_{Y',y'}$ with its strict henselization, we can write $F'$ as $F'_{1}+F'_{2}$, where $F'_{1}$ and $F'_{2}$ are distinct prime divisors.
Then this equivalence follows from the above discussion.)
In particular, $F'$ has normal crossings if and only if $F+E$ has (simple) normal crossings.
\end{sub}

\section{$2$-dimensional log regular log schemes}
\label{logregularsecton}
In this section, we study singularities and desingularizations of log regular schemes of dimension $2$.

\subsection{$2$-dimensional strongly convex rational polyhedral cones}
\label{convexsubsection}
In this subsection, we study $2$-dimensional strongly convex rational polyhedral cones.
We start by fixing notation which is compatible with that used in \cite{Niz}.

\begin{sub}
\label{conesetting}
Let $L$ (resp.\,$L_{\R}$; $L^{\vee}$) be a free abelian group of rank $2$ (resp.\,the $\R$-linear space  $L\otimes_{\Z}\R$;  the abelian group $\Hom_{\Z}(L,\Z)$).
As in \cite[Section 1.1]{Oda}, a subset $\sigma$ of $L_{\R}$ is said to be a {\em strongly convex rational polyhedral cone} if there exists $l_{1},\cdots,l_{s}\in L$ such that $\sigma=\sum_{1\leq i\leq s} \R_{\geq0}l_{i}$ and $\sigma\cap(-\sigma)=\{0\}$.
Note that $P:=\{a\in L^{\vee}\mid a(\sigma)\subset\R_{\geq0}\}$ is an f.s.\,sharp monoid by \cite[Proposition 1.1]{Oda}.

As in \cite[Section 1.1]{Oda}, we define a \textit{fan in $L$} to be a nonempty collection $\Delta$ of strongly convex rational polyhedral cones in $L_{\R}$ which satisfies the following two conditions:
\begin{itemize}
\item[(i)]
Each face of any element of $\Delta$ is also contained in $\Delta$.
\item[(ii)]
For any $\sigma_{1},\sigma_{2}\in\Delta$, $\sigma_{1}\cap\sigma_{2}$ is a face of $\sigma_{1}$ and $\sigma_{2}$.
\end{itemize}
We shall say that a fan $\Delta$ in $L$ is \textit{nonsingular} if, for any $\tau\in\Delta$, there exists a subset of a $\Z$-basis of $L$ which generates $\tau\cap L$ as a monoid (cf.\,\cite[Theorem 1.10]{Oda}). 
Fix two fans $\Delta$ and $ \Delta'$ in $L$.
We shall say that $\Delta'$ is a \textit{proper subdivision of $\Delta$} if the following two conditions are satisfied (cf.\,\cite[Theorem 1.15]{Oda}, \cite[(9.5), (9.6), (9.7)]{Ka2}):
\begin{itemize}
\item[I]
For each $\tau'\in\Delta'$, there exists an element $\tau\in\Delta$ such that $\tau'\subset\tau$.
\item[II]
For each $\tau\in\Delta$, the set $\Delta'_{\tau}:=\{\tau'\in\Delta'\mid \tau'\subset \tau\}$ is a finite set and $\bigcup_{\tau'\in\Delta'_{\tau}}\tau'=\tau$ holds.
\end{itemize}
\label{fundnota}
If conditions I and II are satisfied and $\Delta'$ is nonsingular, we shall say that $\Delta'$ is a \textit{nonsingular proper subdivision} of $\Delta$.

Let $\sigma$ be a strongly convex rational polyhedral cone and $f:\sigma\to\R$ a function.
We shall say that $f$ satisfies property ($\ast$) (cf.\,the discussion before \cite[I.Theorem 9]{KKMS}, the discussion before \cite[Lemma 3.2]{Niz}, the discussion after \cite[Lemma 2.3]{Sai2}) if $f$ satisfies the following:\\
($\ast$): homogeneous, continuous, piecewise-linear, convex, integral on $\sigma\cap L$ (i.e., $f(\sigma\cap L)\subset \Z$).\\
Here, $f$ is said to be convex if, for any $s,t\in\sigma$ and any $\lambda\in [0,1]$, $\lambda f(s)+(1-\lambda)f(t)\leq f(\lambda s+(1-\lambda)t)$ holds.
If $f$ satisfies property ($\ast$), the proper subdivision which consists of all maximal (strongly) convex rational polyhedral cones contained in $\sigma$ where $f$ is linear and all the faces of these convex rational polyhedral cones is referred to as the \textit{proper subdivision of $\sigma$ associated to} $f$ (cf.\,\cite[Theorem 4.7]{Niz}, \cite[Theorem 1.10]{KKMS}).

Suppose $\sigma+(-\sigma)=L_{\R}$.
As in \cite[Proposition 1.19]{Oda}, write $\Theta$ for the convex hull in $L_{\R}$ of the set $(\sigma\cap L)\setminus \{0\}$, $\partial\Theta$ for the boundary polygon of $\Theta$, and $l_{0},l_{1},\cdots,l_{m},l_{m+1}$ for the points of the intersection of $L$ and the compact edges of $\partial\Theta$.
Then the {\em coarsest nonsingular proper subdivision} of the fan $\{\tau\subset\sigma\mid\tau\text{ is a face of }\sigma\}$ is the fan consisting of $\R_{\geq0}l_{i}+\R_{\geq0}l_{i+1}$ ($0\leq i\leq m$) and their faces by \cite[Proposition 1.19]{Oda}.
Note that there exists an integer $a_{i}\leq -2$ such that $l_{i-1}+a_{i}l_{i}+l_{i+1}=0$ for each $1\leq i\leq m$ by \cite[Proposition 1.19]{Oda}.
\end{sub}

\begin{lem}
\label{propertiesoffans}
We use the notation in \ref{conesetting}.
\begin{enumerate}
\item
For any $1\leq i\leq m$, write $l_{i,i}^{\ast},l_{i,i+1}^{\ast}$ for the dual basis of $l_{i},l_{i+1}$ and $P(i)$ for the set $\{p\in P\mid p(l_{i})=1\}$.
Then we have
\begin{align}
\label{P(i)}
P(i)&=\{l_{i,i}^{\ast}+jl_{i,i+1}^{\ast}\in L^{\vee}\mid 1-\delta_{m,i}\leq j\leq -a_{i}-1+\delta_{1,i}\},\\
\label{sumP(i)}
\sharp(\bigcup_{1\leq i\leq m}P(i))&=3-\sum_{1\leq i\leq m}(a_{i}+2),
\end{align}
where $\delta_{-,-}$ denotes Kronecker's delta.
Moreover, for $1\leq i<i'\leq m$, $\sharp(P(i)\cap P(i'))\leq 1$ holds, and the equality holds if and only if $a_{k}=-2$ for every $i<k<i'$.
If this equality holds, the (unique) element of $P(i)\cap P(i')$ is $l_{i,i}^{\ast}+l_{i,i+1}^{\ast}(=l_{i',i'}^{\ast}+(-a_{i'}-1)l_{i',i'+1}^{\ast})$.
\item
Let $\Delta$ be a proper subdivision of the fan $\{\tau\subset\sigma\mid\tau\text{ is a face of }\sigma\}$.
Then there exists a function $f:\sigma\to\R$ satisfying ($\ast$) such that the proper subdivision associated to $f$ coincides with $\Delta$.
\end{enumerate}
\end{lem}

\begin{proof}
First, we show assertion 1.
Any element of $P(i)$ can be written as $l_{i,i}^{\ast}+\beta l_{i,i+1}^{\ast}$ for some $\beta\in\N$.
Moreover, by using the relation $l_{i-1}+a_{i}l_{i}+l_{i+1}=0$ and the inequality $a_{i}\leq-2$, we obtain $0\leq\beta\leq -a_{i}$.
To show that (\ref{P(i)}) holds, we need to know the condition that $(l_{i,i}^{\ast}+\beta l_{i,i+1}^{\ast})(l_{j})\geq0$ holds for every $0\leq j\leq m+1$.
Thus, by using the relation $l_{j-1}+a_{j}l_{j}+l_{j+1}=0$, the inequality $a_{j}\leq-2$, and (usual) induction and descending induction on $j$, we can show that the desired equation holds.
The second and third assertions follow from similar arguments.
Since we have $\sharp P(i)=-a_{i}-1+\delta_{1,i}+\delta_{m,i}$ by (\ref{P(i)}), the equation (\ref{sumP(i)}) follows from this equation and the second and third assertions.
This completes the proof of assertion 1.

Next, we show assertion 2.
Fix an isomorphism $L\simeq\Z^{\oplus 2}$ and write $\Psi$ for its scalar extension $L\otimes_{\Z}\R\simeq \R^{\oplus 2}$.
Let $0\neq l'_{0},\cdots,l'_{t+1}\in L$ be elements of $L\cap\sigma$ such that $l'_{0}=l_{0}$, $l'_{t+1}=l_{m+1}$, $(\R_{\geq 0}l'_{j})\cap L=\N l'_{j}\in \Delta$, $\R_{\geq 0}l'_{j}+\R_{\geq 0}l'_{j+1}\in\Delta$, and $l'_{j}\neq l'_{j+1}$.
Write $\bpm p'_{j}\\ q'_{j}\epm$ for the image $\Psi(l'_{j})$. 
We consider functions $g_{j}:\R^{2}\to \R:\bpm x\\y\epm\mapsto -q'_{i}x+p'_{i}y$ and $f_{j}:=\sum_{k=0}^{j}g_{k}\circ\Psi|_{\sigma}$ for each $0\leq j\leq t$.
Then the function $f:=\min_{0\leq j\leq t}f_{j}$ on $\sigma$ satisfies the desired properties.
\end{proof}

\begin{lem}
\label{fordualmonoid}
Let $c_{1},\cdots,c_{m}$ be elements of $\Z_{\leq-2}$.
Write $v_{i}$ for the projection $\Z^{\oplus m+2}\to\Z: \bpm x_{i}\epm_{0\leq i\leq m+1}\mapsto x_{i}$ for $0\leq i\leq m+1$.
Then the kernel $K$ of the homomorphism $\bpm v_{i-1}+c_{i}v_{i}+v_{i+1}\epm_{1\leq i\leq m}:\Z^{\oplus m+2}\to\Z^{\oplus m}$ is a free abelian group of rank $2$.
Moreover, $L:=K^{\vee}$ and $\sigma:=\R_{\geq 0}(v_{1}|_{K}\otimes\R)+\cdots+\R_{\geq 0}(v_{1}|_{K}\otimes\R)$ satisfies the assumptions in \ref{conesetting} and, for any $i$, $v_{i}$ and $c_{i}$ work as $l_{i}$ and $a_{i}$ in \ref{conesetting}, respectively.
\end{lem}

\begin{proof}
By the definition of $K$, we have an isomorphism $\bpm v_{0}\\ v_{1}\epm:K\to\Z^{\oplus 2}$.
Write $\Phi:K^{\vee}\simeq\Z^{\oplus2}$ for the dual isomorphism of this isomorphism and $\bpm p_{i}\\ q_{i}\epm$ for the image of $v_{i}|_{K}$.
Then we have $p_{0}=q_{1}=1$ and $p_{1}=q_{0}=0$.
By induction on $i$ and using the equation $v_{i-1}+c_{i}v_{i}+v_{i+1}=0$ on $K$, we obtain $p_{i+1}<p_{i}<0<q_{i}<q_{i+1}$ and $q_{i}/p_{i}<q_{i+1}/p_{i+1}$ for each $i\geq 1$.
Hence, the constructed $\sigma$ is a strongly convex rational polyhedral cone satisfying $\sigma+(-\sigma)=L_{\R}$ and the rays $\R_{\geq0}(v_{0}|_{K}\otimes\R)$\;($0\leq i\leq m+1$) give a proper subdivision of $\sigma$.
Since $\bpm v_{i}\\v_{i+1}\epm:K\to\Z^{\oplus 2}$ is an isomorphism by the definition of $K$, this subdivision is nonsingular.
By using the equation $v_{i-1}+c_{i}v_{i}+v_{i+1}=0$ on $K$, we can see that this nonsingular proper subdivision is the coarsest one.
\end{proof}

\subsection{Exceptional divisors}
\label{exceptionsdivisorsubsection}

In this subsection, we study the exceptional divisors of the minimal desingularizations of log regular schemes of dimension $2$.

Let $(R,\fm)$ be a Noetherian local normal domain of dimension $2$ which is \textbf{not} regular.
Write $X$ for the scheme $\Spec R$ and $x$ for the (unique) closed point of $X$.
Let $D$ be a reduced closed subscheme of $X$ of dimension $1$, $\pi:\widetilde{X}\to X$ a desingularization (i.e., a proper birational morphism from a connected regular scheme), $E$ the exceptional divisor of $\pi$, and $D'$ the strict transform of $D$.
(Note that, in Section \ref{introsection}, we assume that $R$ is strictly henselian.)

We explain the ``strict henselization" of this situation.
Write $R^{sh}$ for the strict henselization of $R$ and $\pi^{sh}:\widetilde{X}^{sh}\to X^{sh}$ (resp.\,$E^{sh}$; $x^{sh}$) for the base change of $\pi:\widetilde{X}\to X$ (resp.\,$E$; $x$) by $\Spec R^{sh}\to\Spec R$.
Then $\pi^{sh}$ is a desingularization and $E^{sh}$ is the exceptional divisor of $\pi^{sh}$.

\begin{dfn}[cf.\,Definition \ref{introdivisorconditions}]
We shall say that $E$ satisfies condition (Exc) if $E^{sh}$ satisfies the condition (Exc) in Definition \ref{introdivisorconditions}.
Moreover, in the case where condition (Exc) is satisfied, we shall say that $E$ and $D$ satisfy condition (Str) if $E^{sh}$ and $D^{sh}$ satisfy the condition (Str) in Definition \ref{introdivisorconditions}.
\end{dfn}

We define apparently weaker conditions than condition (Exc).

\begin{dfn}
\label{divisorconditions}
We define conditions on $E$.
\begin{itemize}
\item[(Exc$_{x}$)]
$E$ is a normal crossing divisor on $\widetilde{X}$ and $E$ can be written as $E=\sum_{i=1}^{m}E_{i}$ such that conditions (i), (ii), (iii), and (iv) in Definition \ref{introdivisorconditions} (Exc) are literally satisfied.
\item[(Exc$_{\mathrm{o}}$)]
$E$ is a normal crossing divisor on $\widetilde{X}$ and there exists a separable extension field $k''$ of $k(x)$ of degree $2$ such that $E$ can be written as $E=\sum_{i=1}^{m}E_{i}$ and the following are satisfed:
(i) For each $1\leq i< m$, $E_{i}\simeq\bP^{1}_{k''}$ over $k(x)$.
(i$_{m}$) $E_{m}\times_{\Spec k(x)}\Spec k(x^{sh})$ is isomorphic to $\bP^{1}_{k(x^{sh})}$ over $k(x^{sh})$.
(ii) For each $1\leq i< m$, $a_{i}:=\deg_{E_{i}/k''} \cN_{E_{i}/\widetilde{X}}\leq -2$.
(ii$_{m}$) $a_{m}:=\deg_{E_{m}/k(x)}\cN_{E_{m}/\widetilde{X}}\leq -2$.
(iii) For $1\leq i< m$, there exists a $k''$-rational point $e_{i}$ such that $E_{i}\cap E_{i+1}=\{e_{i}\}$.
(iv) If $|i-j|\geq 2$, $E_{i}\cap E_{j}=\emptyset$.
\item[(Exc$_{\mathrm{e}}$)]
$E$ is a normal crossing divisor on $\widetilde{X}$ and there exists a separable extension field $k''$ of $k(x)$ of degree $2$ such that $E$ can be written as $E=\sum_{i=1}^{m}E_{i}$ and the following are satisfied:
(i) For each $1\leq i< m$, $E_{i}\simeq\bP^{1}_{k''}$ over $k$.
(i$_{m}$) The normalization of $E_{m}$ is isomorphic to $\bP^{1}_{k''}$ and the geometric fiber of the curve $E_{m}$ over $k(x)$ has exactly one ordinary double point (cf.\,\ref{badcurves}).
(ii) For each $1\leq i< m$, $a_{i}:=\deg_{E_{i}/k''}\cN_{E_{i}/\widetilde{X}}\leq-2$.
(ii$_{m}$) $a_{m}:=\deg_{E_{m}/k(x)}\cN_{E_{i}/\widetilde{X}}\leq -2$ (cf.\,\ref{badcurves}).
(iii) For $1\leq i< m$, there exists a $k''$-rational point $e_{i}$ such that $E_{i}\cap E_{i+1}=\{e_{i}\}$.
(iv) If $|i-j|\geq 2$, $E_{i}\cap E_{j}=\emptyset$.
\end{itemize}
Moreover, we define condition (Exc$_{x}'$) (resp.\,(Exc$_{\mathrm{o}}'$); (Exc$_{\mathrm{e}}'$)) to be the condition obtained by replacing ``$k(x)$" in condition (Exc$_{x}$) (resp.\,(Exc$_{\mathrm{o}}$); (Exc$_{\mathrm{e}}$)) with $k':=H^{0}(E,O_{E})$.
By the definitions, each condition (Exc$_{-}'$) is satisfied if condition (Exc$_{-}$) is satisfied.
Note that, as explained in Lemma \ref{rationalsingularity}, each condition (Exc$_{-}$) is equivalent to condition (Exc$_{-}'$).

Suppose that one of conditions (Exc$_{x}'$), (Exc$_{\mathrm{o}}'$), or (Exc$_{\mathrm{e}}'$) is satisfied.
Since the matrix $\bpm \deg_{E_{j}/k'}O_{\widetilde{X}}(E_{i})|_{E_{j}}\epm_{ij}$ is negative definite, we can define the {\em fundamental curve} $E_{f}$ associated to $E$ (cf.\,\cite[Definition 9.4.13]{Liu}, the discussion before \cite[THEOREM 3]{A}), i.e.,\,the minimal effective Cartier divisor $E_{f}=\sum_{1\leq i\leq m} c_{i}E_{i}$ such that $\deg_{E_{i}/k'}O_{\widetilde{X}}(E_{f})|_{E_{j}}\leq 0$ and $E_{f}\geq E_{j}$ for all $j$.
(Note that the condition that $E_{f}\geq E_{j}$ for all $j$ follows from the other conditions and the assumption that $E$ is connected.)
\end{dfn}

\begin{lem}
\label{rationalsingularity}
We suppose that $E$ satisfies one of conditions (Exc$_{x}'$), (Exc$_{\mathrm{o}}'$), or (Exc$_{\mathrm{e}}'$) and use the notation in Definition \ref{divisorconditions}.
\begin{enumerate}
\item
$\widetilde{X}$ is projective over $R$.
\item
We have $E_{f}=\sum_{1\leq i\leq m}E_{i}(=E)$.
\item
$R$ has a rational singularity (cf.\,\cite[Definition (1.1)]{Lip}).
\item
We have $\pi^{\ast}\widetilde{\fm}=O_{\widetilde{X}}(-E)$, where $\widetilde{\fm}$ is the coherent sheaf on $X$ associated to $\fm$ (cf.\,\cite[Theorem 4]{A} and Remark \ref{unnec}).
Moreover, we have $H^{1}(\widetilde{X},O_{\widetilde{X}}(-lE))=0$ and $H^{1}(E,O_{\widetilde{X}}(-lE)|_{E})=0$ for any $l\in\N$.
\item
We have a natural isomorphism $H^{0}(\widetilde{X},O_{\widetilde{X}}(-lE)|_{E})\simeq\fm^{l}/\fm^{l+1}$ for any $l\in\N$.
In particular, we have $k'=k(x)$ and $(H:=)H^{0}(\widetilde{X},O_{\widetilde{X}}(-E)|_{E})\simeq \fm/\fm^{2}$.
\item
$\widetilde{X}\to X$ is the minimal desingularization of $R$ (cf.\,\cite[Theorem (4.1)]{Lip} and \cite[Corollary (27.3)]{Lip}).
\end{enumerate}
In particular, condition (Exc$_{x}$) (resp.\,(Exc$_{\mathrm{o}}$); (Exc$_{\mathrm{e}}$)) is satisfied if and only if (Exc$_{x}'$) (resp.\,(Exc$_{\mathrm{o}}'$); (Exc$_{\mathrm{e}}'$)) is.
Next, suppose that $E$ satisfies condition (Exc$_{x}'$).
\begin{enumerate}
\setcounter{enumi}{6}
\item
The canonical homomorphisms $\Pic(E)\to\bigoplus_{1\leq i \leq m}\Pic(E_{i}): \cL\mapsto (\cL|_{E_{i}})$ and $\bigoplus_{1\leq i \leq m}\Pic(E_{i})\to\Z^{\oplus m}:(\cL_{i})\mapsto (\deg_{E_{i}/k(x)}\cL_{i})$ are isomorphisms.
\item
The natural homomorphism $\Pic(\widetilde{X})\to \Pic(E):\cL\mapsto \cL|_{E}$ is injective.
\item
Write $H_{i}$ for the $k(x)$-linear space $H^{0}(\widetilde{X},O_{\widetilde{X}}(-E)|_{E_{i}})$ for each $1\leq i\leq m$.
We have an exact sequence of $k(x)$-linear spaces
\begin{align}
\label{exactPic}
0\to H\to\bigoplus_{1\leq i\leq m}H_{i}\to k(x)^{\oplus m-1}\to0.
\end{align}
\item
$\dim_{k(x)}\fm/\fm^{2}(=\dim_{k(x)} H)=3-\sum_{1\leq i\leq m}(a_{i}+2)$.
\end{enumerate}
\end{lem}

\begin{proof}
Assertion 1 follows from \cite[Corollary (27.2)]{Lip}.
Since we have
$$\deg_{E_{i}/k'}O_{\widetilde{X}}(\sum_{|i-j|\leq 1}E_{j})|_{E_{i}}\leq [k(E_{i}):k'](a_{i}+2)\leq 0,$$
assertion 2 holds.
Since we have
$$\chi_{k'}(O_{E})=\dim_{k'}(H^{0}(E,O_{E})-H^{1}(E,O_{E}))=1>0,$$
assertion 3 follows from \cite[Theorem (27.1)]{Lip} and assertions 1 and 2.
Next, we show assertion 4.
The first assertion follows from the proof of \cite[Theorem 4]{A} (cf.\,Remark \ref{unnec}.1) and assertions 2 and 3.
The desired vanishing of the cohomology groups follows from the first assertion, \cite[Lemma (12.2)]{Lip}, and \cite[Theorem (12.1)(ii)]{Lip}.
Assertion 5 follows from \cite[Tag:0B4Y]{Stacks} (cf.\,Remark \ref{unnec}.2) and assertions 3 and 4.
Since we have $\deg_{E_{i}/k'}O_{\widetilde{X}}(E_{i})|_{E_{i}}\leq-2$, $E_{i}$ cannot be an exceptional curve of the first kind.
Therefore, by the discussion in \ref{firstkind}, assertion 6 holds.
Assertion 7 follows from the theory of the Picard groups of curves.
Assertion 8 follows from assertion 3, \cite[Proposition (1.2)]{Lip}, and \cite[Theorem (12.1)(i)]{Lip}.
Write $\varphi:\coprod_{1\leq i\leq m}E_{i}\to E$ for the normalization of $E$.
We have an injective homomorphism $O_{\widetilde{X}}(E)|_{E}\to\bigoplus_{1\leq i\leq m}\varphi_{\ast}O_{\widetilde{X}}(E)|_{E_{i}}$ whose quotient sheaf is isomorphic to $\bigoplus_{1\leq i\leq m-1}k(x)_{e_{i}}$, where $k(x)_{e_{i}}$ is the skyscraper sheaf at $e_{i}$ with value $k(x)$.
Considering the long exact sequence of the cohomology groups of these sheaves, we obtain the exact sequence (\ref{exactPic}) by assertion 5.
Thus, assertion 9 holds.
Since we have $O_{\widetilde{X}}(E)|_{E_{i}}\simeq O_{\bP^{1}_{k(x)}}(-a_{i}-2+\delta_{i,1}+\delta_{i,m})$, we have $\dim H_{i}=-a_{i}-1+\delta_{i,1}+\delta_{i,m}$.
Assertion 10 follows from assertion 9.
\end{proof}

\begin{rem}
\label{unnec}
\begin{enumerate}
\item
In \cite[Theorem 4]{A}, $k(x)$ is assumed to be algebraically closed.
Here, we give a brief proof without the assumption that $k(x)$ is algebraically closed.
By \cite[Theorem (12.1)(ii)]{Lip}, the fundamental curve $E_{f}$ is the maximal nontrivial effective divisor $F$ such that $O_{\widetilde{X}}(-F)$ is globally generated and $\Supp F\subset E$.
Then, by \cite[Tag:0B4Y]{Stacks} (cf.\,Remark \ref{unnec}.2) and \cite[Theorem (4.1)]{Lip}, we have $\pi^{-1}(x)=E_{f}$.
\item
In \cite[Tag:0B4Y]{Stacks}, $R$ is assumed to be Nagata (cf.\,\cite[Tag:0B4W]{Stacks}).
In the proof, this assumption is not used.
\end{enumerate}
\end{rem}

\begin{lem}
\label{notationnormalcrossing}
\begin{enumerate}
\item
Condition (Exc) is satisfied if and only if one of conditions (Exc$_{x}$), (Exc$_{\mathrm{o}}$), or (Exc$_{\mathrm{e}}$) is satisfied.
\item
Suppose that condition (Exc) is satisfied.
(We use the notation in Definition \ref{introdivisorconditions} (resp.\,Definition \ref{divisorconditions}) if condition (Exc$_{x}$) (resp.\,one of conditions (Exc$_{\mathrm{o}}$) or (Exc$_{\mathrm{e}}$)) is satisfied.)
Then condition (Str) is satisfied if and only if one of the following conditions is satisfied:
\begin{itemize}
\item[(Str$_{0}$)]
$D'$ is a prime divisor on $\widetilde{X}$, condition (Exc$_{x}$) is satisfied, and $D'+E$ is a normal crossing divisor on $\widetilde{X}$.
Moreover, $D'\cap E_{1}$ and $D'\cap E_{m}$ consist of exactly one $k(x)$-rational point and we have $D'\cap E_{i}=\emptyset\;(1<i<m)$.
\item[(Str$_{1,x}$)]
$D'$ is a prime divisor on $\widetilde{X}$, condition (Exc$_{x}$) is satisfied, and $D'+E$ is a normal crossing divisor on $\widetilde{X}$.
Moreover, we have $m=1$ and $E\cap D'=\{e_{0m}\}$ such that $k(e_{0m})/k(x)$ is a separable extension of degree $2$.
\item[(Str$_{1}$)]
$D'$ is a prime divisor on $\widetilde{X}$, one of conditions (Exc$_{\mathrm{o}}$) or (Exc$_{\mathrm{e}}$) is satisfied, and $D'+E$ is a normal crossing divisor on $\widetilde{X}$.
Moreover, we have $D'\cap E_{i}=\emptyset\;(1<i\leq m)$ and $E\cap D'=\{e_{0m}\}$ such that $k(e_{0m})/k(x)$ is a separable extension of degree $2$.
\item[(Str$_{2}$)]
$D'$ is the sum of distinct prime divisors $E_{0}$ and $E_{m+1}$, condition (Exc$_{x}$) is satisfied, and $D'+E$ is a normal crossing divisor on $\widetilde{X}$.
Moreover, we have $D'\cap E_{i}=\emptyset\;(1<i<m)$, $D'\cap E_{1}=E_{0}\cap E_{1}=\{e_{0}\}$, and $D'\cap E_{m}=E_{m+1}\cap E_{m}=\{e_{m}\}$ such that $e_{0}$ and $e_{m}$ are $k(x)$-rational points.
(In this case, we write $\zeta_{0}$ and $\zeta_{m+1}$ for the image of the generic point of $E_{0}$ and $E_{m+1}$ in $X$, respectively.)
\end{itemize}
\end{enumerate}
\end{lem}

\begin{proof}
First, we show assertion 1.
We suppose that condition (Exc$_{\mathrm{e}}$) in Definition \ref{divisorconditions} is satisfied and use the notation of condition (Exc$_{\mathrm{e}}$).
Write $\widetilde{E}_{m}\cup\widetilde{E}_{m+1}$ for the irreducible decomposition of $E_{m,k(x^{sh})}$.
By using the calculations in \ref{badcurves} and the equation $\deg_{\widetilde{E}_{m}/k(x)^{sh}}\cN_{E_{m}/\widetilde{X}}|_{\widetilde{E}_{m}}=\deg_{\widetilde{E}_{m+1}/k(x)^{sh}} \cN_{E_{m}/\widetilde{X}}|_{\widetilde{E}_{m+1}}$, we obtain
$$\deg_{\widetilde{E}_{i}/k(x)^{sh}}\cN_{\widetilde{E}_{i}/\widetilde{X}}|_{\widetilde{E}_{i}}=\frac{a_{m}-2}{2}\leq -2\quad (i\in\{m,m+1\}).$$
By using this observation and elementary algebraic geometry, we can show that (Exc) is satisfied if one of conditions (Exc$_{x}$), (Exc$_{\mathrm{o}}$), or (Exc$_{\mathrm{e}}$) is satisfied.

Next, we suppose condition (Exc) is satisfied.
In this paragraph, $``E_{i}"$ (resp.\,$``e_{1}"$; $``m"$) denotes the object $E_{i}$ (resp.\,$e_{1}$; $m$) in Definition \ref{introdivisorconditions} for $R^{sh}$.
Let $F$ be an integral closed subscheme of $E$ and $F_{n}$ for the normalization of $F$.
Suppose that the base change $F^{sh}$ contains $``E_{1}"$.
Then we have a natural immersion $``E_{1}"\hookrightarrow F^{sh}_{n}$.
If $``m"=1$, we have $F^{sh}=``E_{1}"(\simeq\bP^{1}_{k(x^{sh})})$ and condition (Exc$_{\mathrm{o}}$) is satisfied.
Suppose $``m"\geq2$.
Then the image $e_{1}$ of $``e_{1}"$ in $F_{n}$ is a $k(F_{n})$-rational point.
Moreover, since $``E_{1}"$ intersects other irreducible components of $E^{sh}$ at only $``e_{1}"$, $F_{n}^{sh}=``E_{1}"$ or $F_{n}^{sh}\simeq ``E_{1}"\coprod``E_{m}"$ holds.
If $F_{n}^{sh}=``E_{1}"$, we have $F(=F_{n})\simeq\bP^{1}_{k(x)}$ and we can show that condition (Exc$_{x}$) is satisfied.
Suppose $F_{n}^{sh}\simeq ``E_{1}"\coprod``E_{m}"$.
Then we have $F\simeq\bP^{1}_{k''}$ for some separable extension field $k''$ over $k$ of degree $2$.
If $``m"=2$, condition (Exc$_{\mathrm{e}}$) is satisfied.
If $``m"=3$, condition (Exc$_{\mathrm{o}}$) is satisfied.
If $``m"\geq 4$, then there exists an integral closed subscheme $F_{2}$ of $E$ such that the base change of $F_{2}$ contains $``E"_{2}$ and $``E"_{m-1}$.
In this case, by induction on $``m"$, it holds that condition (Exc$_{\mathrm{o}}$)(resp.\,condition (Exc$_{\mathrm{e}}$)) is satisfied if $m$ is odd (resp.\,if $m$ is even).

Finally, we suppose that condition (Exc) is satisfied and show assertion 2.
We may assume that $D'+E$ is a normal crossing divisor on $\widetilde{X}$, $D'$ is normal, and the morphism $D'\cap E\to\Spec k(x)$ is finite \'etale of rank $2$.
Then $D'\to D$ is the normalization morphism.
Since we have $D'\times_{D}\Spec k(x)=D'\cap E$ by Lemma \ref{rationalsingularity}.4, $D'\cap E=\coprod_{x'\in D',x'\mapsto x}\Spec k(x')$ holds.
Then it follows that condition (Str$_{0}$) (resp.\,one of conditions (Str$_{1,x}$) or (Str$_{1}$); condition (Str$_{2}$)) is satisfied if and only if $D'$ is irreducible and $\sharp(D'\cap E)=2$ (resp.\,$D'$ is irreducible and $\sharp(D'\cap E)=1$; $D'$ is not irreducible) and condition (Str) is satisfied.
\end{proof}

\begin{prop}
\label{definelogstructure}
Suppose that conditions (Exc$_{x}$) and (Str$_{2}$) are satisfied.
We use the notations in Definition \ref{divisorconditions} and Lemma \ref{notationnormalcrossing}.
\begin{enumerate}
\item
The restriction of the homomorphism
\begin{align*}
\Z^{\oplus m+2}\to\Pic(\widetilde{X}):(c_{i})_{0\leq i\leq m+1}\mapsto O_{\widetilde{X}}(-\sum_{0\leq i\leq m+1} c_{i}E_{i})
\end{align*}
to $0\times \Z^{\oplus m}\times 0$ is injective.
Moreover, the composite homomorphism of this homomorphism and the homomorphisms in Lemmas \ref{rationalsingularity}.7 and \ref{rationalsingularity}.8 sends $(c_{i})_{0\leq i\leq m+1}$ to $-(c_{i-1}+a_{i}c_{i}+c_{i+1})_{1\leq i\leq m}$.
\item
$(X,X\setminus D)$ is a toric pair and the log scheme $X^{\log}$ is Zariski log regular.
\end{enumerate}
\end{prop}

\begin{proof}
First, we show assertion 1.
The desired injectivetiy follows from the discussion after \cite[Lemma (14.1)]{Lip}.
The second assertion follows from a calculation.

Next, we show assertion 2.
Write $M_{X}$ (resp.\,$M_{\widetilde{X}}$) for Zariski monoid sheaf $j_{X\setminus D,\ast}O_{X\setminus D}^{\ast}$ (resp.\,the monoid sheaf determined by the log structure on $\widetilde{X}$ defined by the normal crossing divisor $E+D'$).
Then we have $\pi_{\ast}M_{\widetilde{X}}=M_{X}$ since we have $\pi_{\ast}O_{\widetilde{X}}=O_{X}$.
We write $M:=\Gamma(\widetilde{X},M_{\widetilde{X}})=\Gamma(X,M_{X})$.
Since $R$ is normal, $(X,M_{X})$ is a Zariski log scheme.
We need to show that $(X,M_{X})$ is Zariski log regular.

To prove the desired Zariski log regularity, we start by showing that $M/R^{\ast}$ is an f.s.\,sharp monoid satisfying $\rank(M/R^{\ast})^{\gp}=2$.
Write $K(X)$ for the function field of $X$, $K(X)_{\widetilde{X}}$ for the constant sheaf on $\widetilde{X}$ whose stalks are isomorphic to $K(X)$, and $v_{i}$ for the valuation on $K(X)$ associated to the generic point $\xi_{i}$ of $E_{i}$ for $0\leq i\leq m+1$.
The sheaves $O_{\widetilde{X}}$, $O_{\widetilde{X}}^{\ast}$, $M_{\widetilde{X}}$, and $M_{\widetilde{X}}^{\gp}$ are subsheaves of $K(X)_{\widetilde{X}}$.
Note that we have an exact sequence of sheaves
$$0\to O_{\widetilde{X}}^{\ast}\to M_{\widetilde{X}}^{\gp}\to M_{\widetilde{X}}^{\gp}/O_{\widetilde{X}}^{\ast}\to0,$$
which induces the exact sequence
\begin{align}
\label{picexact}
0\to R^{\ast}\to \Gamma(\widetilde{X},M_{\widetilde{X}}^{\gp})
\to\Gamma(\widetilde{X},M_{\widetilde{X}}^{\gp}/O_{\widetilde{X}}^{\ast})\to \Pic(\widetilde{X}).
\end{align}
For each $0\leq i\leq m$, we write $\iota_{i}$ for the inclusion $\{\xi_{i}\}\hookrightarrow\widetilde{X}$ and  $\Z_{\xi_{i}}$ for the constant sheaf on $\{\xi_{i}\}$.
Then we have a homomorphism $M_{\widetilde{X}}^{\gp}/O_{\widetilde{X}}^{\ast}\to\bigoplus_{0\leq i\leq m+1}\iota_{i,\ast}\Z_{\xi_{i}}$ induced by using the isomorphism $v_{i}:(M_{\widetilde{X}}^{\gp}/O_{\widetilde{X}}^{\ast})_{\xi_{i}}\simeq \Z$.
Thus, we have an isomorphism $(v_{i})_{0\leq i\leq m+1}:\Gamma(\widetilde{X},M_{\widetilde{X}}^{\gp}/O_{\widetilde{X}}^{\ast})\simeq\Z^{\oplus (m+2)}$.
By using this identification, we can see that the last homomorphism in (\ref{picexact}) coincides with the homomorphism in assertion 1.
Hence, we have
\begin{align*}
\label{identification}
&\Gamma(\widetilde{X},M_{\widetilde{X}}^{\gp})/R^{\ast}\\
\simeq&\{(c_{i})_{0\leq i\leq m+1}\in\Z^{\oplus m+2}\mid 1\leq\forall i\leq m, c_{i-1}+a_{i}c_{i}+c_{i+1}=0\}.\tag{$\star$}
\end{align*}
Since we have $M_{\widetilde{X}}=M_{\widetilde{X}}^{\gp}\cap O_{\widetilde{X}}$ (which follows from $M_{\widetilde{X}}^{\gp}=j_{\widetilde{X}\setminus(D'\cup E),\ast}O_{\widetilde{X}\setminus(D'\cup E)}^{\ast}$ (cf.\,\cite[(11.6) Theorem]{Ka2})), we can identify
$$M/R^{\ast}=\Gamma(\widetilde{X},M_{\widetilde{X}})/R^{\ast}
(\subset\Gamma(\widetilde{X},M_{\widetilde{X}}^{\gp})/R^{\ast}\simeq\Z^{\oplus 2})$$
with the submonoid of the right hand side of ($\star$) consisting of the elements whose images by $v_{i}$ is contained in $\N$ for any $0\leq i\leq m+1$.
Then, by Lemma \ref{fordualmonoid} and the first paragraph of \ref{conesetting}, it holds that $M/R^{\ast}$ is an f.s.\,sharp monoid.

Since $M/R^{\ast}$ is an f.s.\,sharp monoid, we have a section $s$ of the group homomorphism $M^{\gp}\to M^{\gp}/R^{\ast}$.
Write $P$ for the image of $M/R^{\ast}$ in $M^{\gp}$, which is contained in $M$ and satisfies $PR^{\ast}=M$.
In the following, we show that the associated log structure $P^{a}(\to M_{X})\to O_{X}$ coincides with $M_{X}\to O_{X}$ and $(X,M_{X})$ is Zariski log regular.
Over $X\setminus \Supp D$, since $P^{a}\to O_{X}$ factors through $M_{X}\to O_{X}$, we have $P^{a}|_{X\setminus D}\simeq O_{X\setminus D}^{\ast}$.
At $\zeta_{0}$ (resp.\,$\zeta_{m+1}$), since there exists a (unique) element $p_{0}\in P$ (resp.\,$p_{m}\in P$) satisfying $v_{0}(p_{0})=v_{1}(p_{0})=1$ (resp.\,$v_{m}(p_{m})=v_{m+1}(p_{m})=1$) by ($\star$) and Lemma \ref{propertiesoffans}.1, it holds that $v_{0}(P)=\N$ (resp.\,$v_{m+1}(P)=\N$).
Since we have $\Supp D=\{\zeta_{0},\zeta_{m+1},x\}$, it suffices to show that $(P\setminus\{1\})R=\fm$, or equivalently, $\fm/\fm^{2}$ is generated by the image of $P\setminus\{1\}$ as an $R$-module.

We use the notation of Lemma \ref{rationalsingularity}.
By Lemma \ref{rationalsingularity}.5, it suffices to show that the image of $P\setminus\{1\}$ in $H$ is a system of generators of the $k(x)$-linear space.
Note that, for any $1\leq i\leq m$, the image of an element $p\in P\setminus\{1\}$ in $H_{i}$ is nontrivial if and only if $v_{i}(p)=1$.
Write $P(i):=\{p\in P\setminus\{1\}\mid v_{i}(p)=1\}$.
Note that $O_{\widetilde{X}}(-E)|_{E_{i}}$ is isomorphic to $O_{\bP^{1}_{k(x)}}(-a_{i}-1+\delta_{i,1}+\delta_{i,m})$.
For an element $p\in P(i)$, the image of $p$ in $H_{i}$ defines the effective Cartier divisor $v_{i-1}(p)e_{i-1}+v_{i+1}(p)e_{i}$ on $E_{i}$.
Therefore, by Lemmas \ref{propertiesoffans}.1 and \ref{fordualmonoid}, the image of $P(i)$ in $H_{i}$ is a $k(x)$-basis.
By using this fact, Lemmas \ref{propertiesoffans}.1, and induction on $i$, we can see that the image of $\bigcup_{1\leq j\leq i}P(j)$ in $H$ is linearly independent (cf.\,Lemma \ref{rationalsingularity}.9).
Thus, the image of $\bigcup_{1\leq i\leq s}P(i)$ in $H$ is linearly independent subset of cardinality $3-\sum_{1\leq i\leq m}(a_{i}+2)$ by Lemma \ref{propertiesoffans}.1, which shows this subset is a $k(x)$-basis of $H$ by Lemma \ref{rationalsingularity}.10.
\end{proof}

\subsection{Log blow-ups}
\label{logblowupsubsection}

In this subsection, we study log blow-ups of log regular schemes of dimension $2$.

\begin{nota-dfn}[cf.\,\cite{Niz}]
\label{fundnota}
Let $X^{\log}=(X,M_{X})$ be a connected Noetherian log regular log scheme of dimension $2$.
Fix a point $x\in X$ and take a geometric point $\overline{x}$ over $x$.
We shall write $P_{\overline{x}}$ (resp.\,$L_{\overline{x}}$; $L_{\overline{x},\R}$; $L^{\vee}_{\overline{x}}$; $\sigma_{\overline{x}}$) for the monoid $M_{X,\overline{x}}/O_{X,\overline{x}}^{\ast}$ (resp.\,the abelian group $\Hom_{\Z}(P_{\overline{x}}^{\gp},\Z)$; the $\R$-linear space $L_{\overline{x}}\otimes_{\Z}\R$; the abelian group $P_{\overline{x}}^{\gp}(\simeq \Hom_{\Z}(L_{\overline{x}},\Z))$; the strongly convex rational polyhedral cone $\{v\in L_{\overline{x},\R}\mid \forall u\in P_{\overline{x}},u(v)\geq0\}$) (cf.\,\cite[Section 1]{Oda}, \cite[Section 9]{Ka2}, \cite{Niz}).
Note that $P_{\overline{x}}$ is an f.s.\,sharp monoid.

For any coherent fractional ideal $J\subset M_{X,\overline{x}}^{\gp}$ (cf.\,\cite[(5.7) Definition]{Ka2}), write $f_{J}:\sigma_{\overline{x}}\to\R$ for the function satisfying $f_{J}(a)=\min\{\mu(a)\mid \mu\in J\}$.
This function $f_{J}$ satisfies property ($\ast$).
Conversely, for a function $f:\sigma_{\overline{x}}\to\R$ satisfying ($\ast$), we define a coherent fractional ideal $J_{f,\overline{x}}$ of $P_{\overline{x}}$ to be $\{a\in L_{\overline{x}}^{\vee}(=P_{\overline{x}}^{\gp})\mid a(z)\geq f(z),\forall z\in \sigma_{\overline{x}}\}$ (cf.\,the discussion before \cite[Theorem 9]{KKMS}).
(See also the discussion after \cite[Definition 3.11]{Niz}.)

Suppose $x\in F(X)^{(2)}$.
We write $U(x)$ for the open subset
$$(U(x):=)X\setminus\bigl((F(X)^{(2)}\setminus\{x\})\cup \bigcup_{\substack{z\in F(X)^{(1)}\\ x\notin\overline{\{z\}}}}\overline{\{z\}}\bigr)$$
of $X$.
Note that we have $F(X)^{(2)}\cap U(x)=F(U(x))^{(2)}=\{x\}$.
Let $y$ be a point of $F(X)\cap U(x)$.
Then we have $x\in\overline{\{y\}}$ and, for a suitbale geometric point $\overline{y}$ over $y$, we have a cospecialization homomorphism $P_{\overline{x}}\to P_{\overline{y}}$, which induces injection $\sigma_{\overline{y}}\to \sigma_{\overline{x}}$ by \cite[Lemma 2.12]{Niz}.
Let $z$ be a point of $U(x)$.
Then the prime ideal $I(\overline{z},M_{X})$ of $O_{X,\overline{z}}$ defines a point $y_{z}\in F(X)\cap U(x)$.
We have a natural isomorphism $\sigma_{\overline{z}}\simeq\sigma_{\overline{y_{z}}}$.
For a function $f:\sigma_{\overline{x}}\to\R$ satisfying ($\ast$) and $z\in U(x)$, we define a function
$$f_{\overline{z}}:\sigma_{\overline{z}}\to \sigma_{\overline{y_{z}}}\to \sigma_{\overline{x}}
\xrightarrow{f}\R.$$
As in the discussion before \cite[Lemma 3.12]{Niz}, we can define a coherent fractional ideal $J_{f}\subset M_{U(x)}^{\gp}$ (cf.\,\cite[Definition 3.1(2)]{Niz}) satisfying that, for any \'etale morphism $V\to U(x)$,
$$J_{f}(V)=\{a\in M_{U(x)}^{\gp}(V)\mid \forall\text{ geometric point }\overline{z}\to U(x), \forall s\in\sigma_{\overline{z}}, a_{\overline{z}}(s)\geq f_{\overline{z}}(s)\}$$
by \cite[Lemma 3.12]{Niz}.
Note that the pull-back of the fractional ideal $J_{f,\overline{x}}$ by canonical surjection $M_{X,\overline{x}}^{\gp}\to M_{X,\overline{x}}^{\gp}/O_{X,\overline{x}}^{\ast}$ coincides with the stalk $(J_{f})_{\overline{x}}$ by the definitions.
\end{nota-dfn}

In the next proposition, we use the notion of log blow-ups.
For a detailed explanation on log blow-ups, see \cite[2.2]{FK}, \cite[1.6]{I1}, \cite{I2}, and \cite[section 4]{Niz}.

\begin{prop}
\label{minlogdes}
Let $X^{\log}=(X,M_{X})$ be a connected Noetherian log regular log scheme of dimension $2$.
For any $x\in F(X)^{(2)}$, let $\Delta(x)$ be a proper subdivision of $\{\tau\subset\sigma_{\overline{x}}\mid\tau\text{ is a face of }\sigma_{\overline{x}}\}$.
\begin{enumerate}
\item
There exists a log regular scheme $\pi^{\log}:X_{\Delta}^{\log}\to X^{\log}$ over $X^{\log}$ satisfying the following:
\begin{itemize}
\item
\'Etale locally on $X$, $\pi^{\log}$ is obtained by the log blow-up of a fractional coherent ideal of $M_{X}$.
In particular, $\pi^{\log}$ is log \'etale.
\item
For any $x\in F(X)^{(2)}$ and any geometric point $\overline{x}\to X$ over $x$, the log scheme $X^{\log}_{\Delta}\times^{\log}_{X^{\log}}(\Spec O_{X,\overline{x}})^{\log}$ coincides with the ``base change" (cf.\,\cite[(9.10) Definition]{Ka2}) of $(\Spec O_{X,\overline{x}})^{\log}$ associated to $\Delta(x)$.
\end{itemize}
\item
Let $U\to X$ be an \'etale morphism from a connected Noetherian scheme.
Then we have
$$U_{\Delta_{U}}^{\log}\simeq U^{\log}\times^{\log}_{X^{\log}}X_{\Delta}^{\log},$$
where, for $x\in F(U)^{(2)}$, $\Delta_{U}(x)=\Delta(x)$.
Moreover, if $U^{\log}$ is Zariski log regular, $U_{\Delta_{U}}^{\log}\to U^{\log}$  coincides with the ``base change" (cf.\,\cite[(9.10) Definition]{Ka2}).
\end{enumerate}
\end{prop}

\begin{proof}
We use the notation in \ref{fundnota}.
Let $x$ be a point of $F(X)^{(2)}$.
By using Lemma \ref{propertiesoffans}.2, take a function $f:\sigma_{\overline{x}}\to\R$ satisfying ($\ast$) such that the proper subdivision associated to $f$ coincides with $\Delta(x)$.
As in Notation-Definition \ref{fundnota}, we have a coherent fractional ideal $J_{f}\subset M_{U(x)}$.
Let $(B_{J_{f}}U(x))^{\log}$ be the log blow-up of $U(x)^{\log}$ defined by $J_{f}$.
Then the morphism $(B_{J_{f}}U(x))^{\log}\to U(x)^{\log}$ is isomorphic over $(U(x)\setminus\{x\})^{\log}$.
Since $(B_{J_{f}}U(x))^{\log}$ is an f.s.\,log scheme log \'etale over the log regular scheme $U(x)^{\log}$, $(B_{J_{f}}U(x))^{\log}$ is also log regular by \cite[(8.2) THEOREM]{Ka2}.
Hence, $B_{J_{f}}U(x)$ is a normal scheme birational to $U(x)$.
Moreover, $(B_{J_{f}}U(x))^{\log}\times^{\log}_{U(x)^{\log}}(\Spec O_{X,\overline{x}})^{\log}$ is canonically isomorphic to the log blow-up of $(\Spec O_{X,\overline{x}})^{\log}$ with respect to $J_{f}|_{\Spec O_{X,\overline{x}}}$.
By \cite[Proposition 4.5]{Niz}, \cite[Theorem 4.7]{Niz}, and \cite[Lemma II.4.6]{Ts} (cf.\,\cite[Lemma 2.3]{Niz}), the log blow-up of $(\Spec O_{X,\overline{x}})^{\log}$ coincides with the ``base change" (cf.\,\cite[(9.10) Definition]{Ka2}) of $(\Spec O_{X,\overline{x}})^{\log}$ associated to the subdivision of $\sigma_{\overline{x}}$ defined by $f$.
Hence, the log blow-up of $\Spec O_{X,\overline{x}}$ does not depend on $f$.
By elementary descent theory, the normal scheme $B_{J_{f}}U(x)$ and also $(B_{J_{f}}U(x))^{\log}$ do not depend on $f$.
We write $U(x)_{\Delta_{U(x)}}^{\log}$ for the log scheme $(B_{J_{f}}U(x))^{\log}$ over $U(x)^{\log}$.

By gluing the log scheme $U(x)_{\Delta_{U(x)}}\to U(x)\;(x\in F(X)^{(2)})$ and $\id:X\setminus F(X)^{(2)}\to X\setminus F(X)^{(2)}$, we obtain a log scheme over $X^{\log}$.
Then the resulting log scheme works as the desired log scheme $X_{\Delta}^{\log}$.
Other assertions of Proposition \ref{minlogdes} follow from the construction of $X_{\Delta}^{\log}$.
\end{proof}

\begin{nota-dfn}
\label{minimaldef}
We use the setting of Proposition \ref{minlogdes}.
Suppose that $\Delta(x)$ is the coarsest subdivision of the fan $\{\tau\subset\sigma\mid\tau\text{ is a face of }\sigma_{\overline{x}}\}$ for any $x\in F(X)^{(2)}$.
In this case, we shall write $\widetilde{X}^{\log}$ for the log scheme $X_{\Delta}^{\log}$ and refer to $\widetilde{X}^{\log}$ as the \textit{minimal log desingularization} of $X^{\log}$.
\end{nota-dfn}

\begin{exam}
\label{maximallogblowup}
Let $X^{\log}=(X,M_{X})$ be a connected Noetherian log regular log scheme of dimension $2$, $x$ an element of $F(X)^{(2)}$, and $\overline{x}$ a geometric point of $X$ over $x$.
By considering the coherent fractional ideal $P_{\overline{x}}\setminus\{1\}$ in $P_{\overline{x}}$, the blow-up of $X$ with respect to the maximal ideal of $x$ can be regarded as (the morphism between the underlying schemes of) a log blow-up.
\end{exam}

\subsection{A characterization of log regular local schemes of dimension $2$}
\label{summarysubsection}

In this subsection, we give a characterization of log regular local schemes of dimension $2$.

\begin{prop}[cf.\,{\cite[Proposition 1.19]{Oda}}]
\label{2dimfan}
Let $X^{\log}$ be a log regular log scheme.
Suppose that $X$ is the spectrum of a 2-dimensional Noetherian local ring $R$ which is \textbf{not} regular.
Write $x$ for the (unique) closed point of $X$.
Suppose $x\in F(X)^{(2)}$.
Write $\pi^{\log}:\widetilde{X}^{\log}\to X^{\log}$ for the minimal log desingularization, $D$ for the reduced closed subscheme of $X$ where the log structure is nontrivial, and $D'$ (resp.\,$E$) for the strict transform of $D$ (resp.\,the exceptional divisor) of $\widetilde{X}\to X$.
Then conditions (Exc) and (Str) are satisfied in this situation.
Moreover, if $X^{\log}$ is Zariski log regular, conditions (Exc$_{x}$) and (Str$_{2}$) are satisfied.
In this case, in particular, $E$ satisfies the properties stated in Lemma \ref{rationalsingularity}.
\end{prop}

\begin{proof}
If $X^{\log}$ is Zariski log regular and conditions (Exc) and (Str) are satisfied, we have $\sharp F(X)^{(2)}=2$ and condition (Str$_{2}$) is automatically satisfied by Lemma \ref{notationnormalcrossing}.
Hence, the second assertion follows from the first assertion.

To show the first assertion, we may assume that $X$ is strictly local by Proposition \ref{minlogdes}.1.
Again by Proposition \ref{minlogdes}.1, $\widetilde{X}^{\log}\to X^{\log}$ coincides with the desingularization defined in \cite{Ka2}.
Hence, $\widetilde{X}$ is a regular scheme and $D'\cup E$ has normal crossings by \cite[(10.4)]{Ka2} and \cite[(11.6)]{Ka2} (see also Lemma \cite[Lemma 5.2]{Niz}).
In particular, the setting of subsection \ref{exceptionsdivisorsubsection} is satisfied.
Moreover, it follows from  the second and third paragraph of \cite[\S2]{M} and \cite[Theorem 2.1]{M} that conditions (i), (iii), and (iv) in condition (Exc) and condition (Str) are satisfied.
It suffices to show that condition (ii) in condition (Exc) are satisfied.
In the following, we use the notation of Lemma \ref{conesetting}.1 and Notation-Definition \ref{fundnota}.
Then each $l_{i}$ corresponds to the valuation defined by the generic point of $E_{i}$.
For $1\leq i\leq m$, take an element $p_{i}\in P(i)$.
Then we have
\begin{align*}
&\deg_{E_{i}/k(x)}\cN_{E_{i}/\widetilde{X}}
=-\deg_{E_{i}/k(x)}(O_{\widetilde{X}}(-E_{i})|_{E_{i}})\\
=&-(\deg_{E_{i}/k(x)}(p_{i}O_{\widetilde{X}}|_{E_{i}})+l_{i-1}(p_{i})+l_{i+1}(p_{i}))\\
=&a_{i}l_{i}(p_{i})=a_{i}.
\end{align*}
This completes the proof of Proposition \ref{2dimfan}.
\end{proof}

\begin{thm}[cf.\,Theorem \ref{intromaincharacterization}]
\label{logsummary}
Let $(R,\fm)$ be a Noetherian normal local domain of dimension $2$.
Write $X=\Spec R$.
Let $D$ be a $1$-dimensional reduced closed subscheme of $X$.
\begin{enumerate}
\item
The following are equivalent:
\begin{itemize}
\item[(I)]
$(X,X\setminus D)$ is a toric pair and $X^{\log}$ is Zariski log regular.
\item[(II)]
One of the following holds:
\begin{itemize}
\item[(i)]
$X$ is regular and $D$ is a simple normal crossing divisor on $X$.
\item[(ii)]
$X$ is \textbf{not} regular and there exists a desingularization $\pi:\widetilde{X}\to X$ such that $\pi$ and $D$ satisfy conditions (Exc$_{x}$) and (Str$_{2}$).
\end{itemize}
\end{itemize}
\item
The following are equivalent:
\begin{itemize}
\item[(A)]
$(X,X\setminus D)$ is a toric pair.
\item[(B)]
One of the following holds:
\begin{itemize}
\item[(i)]
$X$ is regular and $D$ is a normal crossing divisor on $X$.
\item[(ii)]
$X$ is \textbf{not} regular and there exists a desingularization $\pi:\widetilde{X}\to X$ such that $\pi$ and $D$ satisfy conditions (Exc) and (Str).
\end{itemize}
\item[(C)]
The condition obtained by replacing ``conditions (Exc) and (Str)" in condition (B) with ``one of conditions (Exc$_{x}$), (Exc$_{\mathrm{o}}$), or (Exc$_{\mathrm{e}}$) and one of conditions (Str$_{0}$), (Str$_{1,x}$), (Str$_{1}$), or (Str$_{2}$)" is satisfied.
\end{itemize}
\item
Suppose that $(X,X\setminus D)$ is a toric pair and $x\in F(X)^{(2)}$.
Then $X^{\log}$ is Zariski log regular if and only if $\sharp F(X)^{(1)}=2$.
\item
Suppose that $(X,X\setminus D)$ is a toric pair and $x\in F(X)^{(2)}$.
Write $\widetilde{D}=\Spec A$ for the normalization of $D$.
Then $\widetilde{D}\times_{D}\Spec k(x)\to\Spec k(x)$ is finite \'etale of rank $2$ and $\mathrm{length}_{O_{D,x}}A/O_{D,x}=1$ (cf.\,\ref{node-like}).
Moreover, suppose that $X^{\log}$ is Zariski log regular (and use the notation of Lemma \ref{notationnormalcrossing}).
The integral closed subschemes $\overline{\{\zeta_{0}\}}$ and $\overline{\{\zeta_{m+1}\}}$ of $X$ are normal (cf.\,\cite[(7.2) PROPOSITION]{Ka2}).
\end{enumerate}
\end{thm}

\begin{proof}
First, we show assertion 1.
If $R$ is regular, (I)$\Leftrightarrow$(II)(i) follows from Lemma \cite[Lemma 5.2]{Niz}.
If $R$ is not regular, (I)$\Leftrightarrow$(II)(ii) follows from Proposition \ref{definelogstructure}.2 and Proposition \ref{2dimfan}.

Next, we show assertion 2.
(B) $\Leftrightarrow$ (C) follows from Lemma \ref{notationnormalcrossing}.
If $R$ is regular, (A) $\Leftrightarrow$ (B)(i) follows from Lemma \cite[Lemma 5.2]{Niz}.
If $R$ is not regular, (A) $\Leftrightarrow$ (B)(ii) follows from assertion 1, Proposition \ref{2dimfan}, and Lemma \ref{stalkwiselogregular}.2.

Assertion 3 follows from assertion 1, \cite[(7.3) COROLLARY]{Ka2}, \cite[(10.2)]{Ka2}, and the fact that the number of the prime ideal of height $1$ in $P_{\overline{x}}$ is $2$ (cf.\,\ref{conesetting}).

The second assertion of assertion 4 follows from \cite[(7.2) PROPOSITION]{Ka2}.
To show assertion 4, we may assume that $R$ is strictly henselian and hence $X^{\log}$ is Zariski log regular.
By these observation, we have $\widetilde{D}\simeq\overline{\{\zeta_{0}\}}\coprod\overline{\{\zeta_{m+1}\}}$.
Fix a chart $P_{\overline{x}}\to M_{X,\overline{x}}$ inducing a section of $M_{X,\overline{x}}\to P_{\overline{x}}$ and write $l'_{0},\cdots,l'_{m+1}$ for the minimal system of generators of $P_{\overline{x}}$ (as in \ref{conesetting}).
Then the ring $R_{0}:=R/(l'_{1},\cdots,l'_{m+1})$ (resp.\,$R_{m+1}:=R/(l'_{0},\cdots,l'_{m})$) is a discrete valuation ring whose maximal ideal $\fm_{0}:=\fm/(l'_{1},\cdots,l'_{m+1})$ (resp.\,$\fm_{m+1}:=\fm/(l'_{0},\cdots,l'_{m})$) is generated by the image of $l'_{0}$ (resp.\,$l'_{m+1}$) and we may assume that the integral closed subschemes $\overline{\{\zeta_{0}\}}$ (resp,\,$\overline{\{\zeta_{m+1}\}}$) is isomorphic to $\Spec R/(l'_{1},\cdots,l'_{m+1})$ (resp.\,$\Spec R/(l'_{0},\cdots,l'_{m})$) again by \cite[(7.2) PROPOSITION]{Ka2}.
To show the desired assertion, we may assume that $R$ is complete since the composite (monoid) homomorphism $P_{\overline{x}}\to R\to\widehat{R}$ defines a log regular log scheme $(\Spec\widehat{R})^{\log}$.
It suffices to show that (i) we have $D\simeq\Spec R/(l'_{1},\cdots,l'_{m})$ and (ii) the maximal ideal $\overline{\fm}:=\fm/(l'_{1},\cdots,l'_{m})\subset R/(l'_{1},\cdots,l'_{m})$ is canonically isomorphic to $\fm_{0}\times \fm_{m+1}(\subset R_{0}\times R_{m+1})$.

(i) follows from (ii).
Indeed, since $R/(l'_{1},\cdots,l'_{m})$ is a subring of $R_{0}\times R_{m+1}$ by (ii), $R/(l'_{1},\cdots,l'_{m})$ is also reduced.
Moreover, since $R/(l'_{1},\cdots,l'_{m})\hookrightarrow R_{0}\times R_{m+1}$ is finite, the underlying topological space of $R/(l'_{1},\cdots,l'_{m})$ is $D$.
To show (ii), since the quotient rings of $R$ are complete, it suffices to show the homomorphism
$$\theta_{r}:\overline{\fm}^{r}/\overline{\fm}^{r+1}\to\fm_{0}^{r}/\fm_{0}^{r+1}\times \fm_{0}^{r}/\fm_{0}^{r+1}$$
is an isomorphism for each $r\in\N_{\geq1}$.
The surjectivity of $\theta_{r}$ follows from the fact that $\fm_{0}^{r}/\fm_{0}^{r+1}$ (resp.\,$\fm_{0}^{r}/\fm_{0}^{r+1}$) is generated by the image of $l'^{r}_{0}$ (resp.\,$l'^{r}_{m+1}$).
Since $\fm^{r}/\fm^{r+1}$ is generated by $(P_{\overline{x}}\setminus\{1\})^{r}$ by \cite[(6.1) THEOREM]{Ka2}, $\overline{\fm}^{r}/\overline{\fm}^{r+1}$ is generated by the images of $l'^{r}_{0}$ and $l'^{r}_{m+1}$.
Hence, the injectivity of  $\theta_{r}$ follows from the surjectivity of $\theta_{r}$.
\end{proof}

\section{Models of curves}
\label{modelsection}
In this section, we study models of log regular curves and establish a theory of minimal log regular curves.

Throughout this section, we work under the following notations:
Let $K$ be a discrete valuation field, $O_{K}$ the valuation ring of $K$, $k$ the residue field of $O_{K}$, $\varpi$ a uniformizer of $O_{K}$, $p$ the characteristic of $k$, $O_{K}^{h}$ the henselization of $O_{K}$, $K^{h}$ the field of fractions of $O_{K}^{h}$, $O_{K}^{sh}$ the strict henselization of $O_{K}^{h}$, and $K^{sh}$ the field of fractions of $O_{K}^{sh}$.
$(\Spec O_{K})^{\log}$ always denotes the log regular log scheme defined by the toric pair $(\Spec O_{K},\Spec K)$.
Let $C^{\log}=(C,D)$ be a log regular curve over $K$ (cf.\,\ref{curvedfn}), $g$ the genus of $C$, and $r$ the rank of $D$ over $K$.

Sometimes, we assume that the following condition is satisfied:\\
(F): The normalization of $O_{K}$ in $D$ is finite over $O_{K}$ (cf.\,Remark \ref{why(F)}).\\
Note that condition (F) is satisfied, for example, if one of the following holds:
\begin{itemize}
\item
$D$ is finite \'etale over $K$.
(This is the case if $C^{\log}$ is a hyperbolic curve.)
\item
$O_{K}$ is excellent.
(See the following remark.)
\end{itemize}

Note that, for any discrete valuation ring $A$, the following are equivalent:
(i) $A$ is N-2 (i.e., Japanese).
(ii) $A$ is Nagata (i.e., universally Japanese) (cf.\,\cite[Tag:0334]{Stacks}).
(iii) $A$ is excellent.
Indeed, by \cite[Tag:09E1]{Stacks} and \cite[Tag:07QV]{Stacks}, it holds that (i)$\Leftrightarrow$(ii)$\Leftarrow$(iii).
Since discrete valuation rings are (Cohen-Macaulay and hence) universally catenary and J-2 by \cite[Corollaire 6.12.6]{EGA}, $A$ is excellent if and only if $A$ is a G-ring, or equivalently, the field extension $\Frac A\subset\Frac\widehat{A}$ is geometrically reduced.
If $A$ is N-2, for any finite field extension $\Frac A\subset F$, $\Frac \widehat{A}\otimes_{\Frac A}F$ is isomorphic to the total ring of fractions of the product of finite (complete) discrete valuation rings, which is reduced.
Hence, (i)$\Rightarrow$(iii) holds.

\subsection{Minimal models of curves}
\label{defsubsection}

In this subsection, we give the definitions of various models of $C^{\log}$.

\begin{nota-dfn}
\label{models}
Let $(\cC,\cD)$ be a pair of a scheme $\cC$ over $O_{K}$ and a reduced closed subscheme $\cD\subset\cC$.
Such a pair $(\cC,\cD)$ is said to be a \textit{model} of $C^{\log}=(C,D)$ if the following hold:
\begin{itemize}
\item
$\cC$ is integral and the structure morphism $\cC\to\Spec O_{K}$ is proper.
\item
$\cC_{K}$ is isomorphic to $C$ over $K$.
\item
The scheme theoretic image of $D$ in $\cC$ (via this isomorphism) coincides with $\cD$.
(Note that $\cD$ is uniquely determined by $\cC$.)
\end{itemize}
For any models $(\cC_{1},\cD_{1})$ and $(\cC_{2},\cD_{2})$ of $C^{\log}$, we define a {\em morphism of models} $(\cC_{1},\cD_{1})\to(\cC_{2},\cD_{2})$ to be a $O_{K}$-morphism $\cC_{1}\to\cC_{2}$ inducing the identity of $C$ on their generic fibers.
Note that such a morphism $\cC_{1}\to\cC_{2}$ induces a morphism $\cD_{1}\to\cD_{2}$.
Suppose that $(\cC,\cD)$ is a model of $C^{\log}$.
We shall refer to $(\cC,\cD)$ as a {\em regular n.c.d.\,model} (resp.\,a {\em regular s.n.c.d.\,model}) of $C$ if $\cC$ is regular and $\cC_{k}\cup\cD$ is the support of a normal crossing divisor (resp.\,a simple normal crossing divisor) on $\cC$ (cf.\,\ref{ncddef}).
In the case of $D=\emptyset$ (and hence $\cD=\emptyset$), we shall refer to $\cC$ as a {\em regular model} of $C^{\log}$ if $\cC$ is regular.
\begin{enumerate}
\item
A model $(\cC,\cD)$ of $(C,D)$ is said to be a {\em log regular model} (resp.\,{\em log smooth model}) if the pair $(\cC,\cC\setminus(\cC_{k}\cup\cD))$ is a toric pair (resp.\,the pair $(\cC,\cC\setminus(\cC_{k}\cup\cD))$ is a toric pair and the resulting log regular scheme $\cC^{\log}$ is log smooth over $(\Spec O_{K})^{\log}$) (cf.\,Remark \ref{why(F)}.3).
We shall say that $(C,D)$ has {\em log smooth reduction} if $(C,D)$ admits a log smooth model.
\item
We shall say that $(C,D)$ has {\em stable reduction} if there exists a stable curve $(\cC,\cD)\to\Spec O_{K}$ whose generic fiber is isomorphic to $(C,D)$ over $K$ (cf.\,\cite[Definition (2.2)]{DM}).
Such a model $(\cC,\cD)$ is referred to as a {\em stable model}.
Note that there exists a canonical log structure on a stable model which is log smooth over $(\Spec O_{K})^{\log}$.
(cf.\,\cite{K}, \cite[Lemma 4.2]{M}, and \cite[Lemma 1.12]{Sai2}.
See also the discussion in \cite[Curves in Section 0]{M2}.)
Thus, if $(C,D)$ has stable reduction, we also say that $C^{\log}$ has {\em stable reduction} and a log smooth model whose underlying scheme is a stable model is referred to as a {\em stable model} of $C^{\log}$.
\item
Let $\cP$ be an element of $\{\text{regular}, \text{regular n.c.d.}, \text{regular s.n.c.d.}\}$ and $(\cC,\cD)$ a $\cP$ model of $C^{\log}$.
We shall refer to $(\cC,\cD)$ as a {\em minimal $\cP$ model} of $C^{\log}$ if any morphism of models from $(\cC,\cD)$ to another $\cP$ model is an isomorphism and an {\em absolutely minimal $\cP$ model} if any other $\cP$ model admits a morphism of models to $(\cC,\cD)$.
\end{enumerate}
\end{nota-dfn}

\begin{rem}
\label{why(F)}
\begin{enumerate}
\item
In \cite[Proposition 2.3]{DM}, it is proved that $C$ has stable reduction if and only if the special fiber of the (unique) minimal regular model of $C$ is reduced and has only ordinary double points in the case where $D=\emptyset$, $g\geq 2$, and $k$ is algebraically closed.
This equivalence holds for general $k$ by \cite[Tag:0E8D, 0CDH, and 0CDG]{Stacks}.
(Note that, to show this equivalence, we may assume that $O_{K}$ is strictly henselian.
In this case, the proof of \cite[Proposition 2.3]{DM} works even if $k$ is not algebraically closed.)
\item
If $C^{\log}$ admits a regular n.c.d.\,model $(\cC,\cD)$, $\cD\to\Spec O_{K}$ is the normalization morphism in $D$ and of finite type.
Hence, by \ref{existenceNCD}, $C^{\log}$ admits a regular n.c.d.\,model if and only if condition (F) is satisfied.
If $C^{\log}$ admits a log regular model, $C^{\log}$ admits a regular n.c.d.\,model by Lemma \ref{fundlog} and hence condition (F) is satisfied.
In subsection \ref{minimalmodelconstructionsection}, we show the existence of minimal regular n.c.d.\,models and minimal log regular models of $C^{\log}$ in the case where condition (F) is satisfied.
\item
Suppose that $C^{\log}$ admits a log smooth model $(\cC,\cD)$.
$\cD$ is normal and at most tamely ramified over $\Spec O_{K}$ by Lemma \ref{fundlogsm}.2.
Hence, $D$ is \'etale over $K$ and hence $C^{\log}$ satisfies condition (F).
\end{enumerate}
\end{rem}

\begin{nota}
\label{contractibledivisor}
Let $\cC^{\log}=(\cC,\cD)$ be a log regular model of $C^{\log}$.
\begin{enumerate}
\item
We shall write $\cE_{\bP,2}(\cC)$ (resp.\,$\cE_{\bP,1}(\cC)$; $\cE_{\bP,1'}(\cC)$) for the set of $1$-dimensional closed subschemes $F$ of $\cC_{k}$ satisfying the following two properties:
(i) $F$ is isomorphic to $\bP^{1}_{k(F)}$ over $k(F)$.
(ii) $F$ intersects other irreducible components of $\cC_{k}\cup\cD$ at exactly two $k(F)$-rational points (resp.\,exactly one $k(F)$-rational point; exactly one closed point whose residue field is separable over $k(F)$ of degree $2$).
\item
We shall write $\cE_{0,1'}(\cC)$ for the set of $1$-dimensional closed subschemes $F$ of $\cC_{k}$ satisfying the following two properties:
(i) $F$ is a proper smooth curve over $k(F)$ with connected geometric fibers of genus $0$.
(ii) $F$ intersects other irreducible components of $\cC_{k}\cup\cD$ at exactly one closed point whose residue field is separable over $k(F)$ of degree $2$.
Note that we have $\cE_{\bP,1'}(\cC)\subset\cE_{0,1'}(\cC)$.
\item
We shall write $\cE_{\node}(\cC)$ for the set of $1$-dimensional closed subschemes $F$ of $\cC_{k}$ satisfying the following two properties:
(i) There exists a separable extension field $k''$ of $k(F)$ of degree $2$ such that the normalization of $F$ is isomorphic to $\bP^{1}_{k''}$ over $k(F)$.
Moreover, the geometric fiber of $F$ over $k(F)$ has exactly one ordinary double point (cf.\,\ref{badcurves}).
(ii) $F$ intersects other irreducible components of $\cC_{k}\cup\cD$ at exactly one $k''$-rational point.
\item
We define $\cE(\cC):=\cE_{\bP,2}(\cC)\cup\cE_{\bP,1}(\cC)\cup\cE_{0,1'}(\cC)\cup\cE_{\node}(\cC)$.
We also define $\cE_{=-1}(\cC)$ (resp.\,$\cE_{\leq -2}(\cC)$) to be the subset of $\cE_{\bP,2}(\cC)\cup\cE_{\bP,1}(\cC)\cup\cE_{\bP,1'}(\cC)$ (resp.\,$\cE_{\bP,2}(\cC)\cup\cE_{0,1'}(\cC)\cup\cE_{\node}(\cC)$) consisting of elements $F$ which satisfy the following two properties:
(i) $\cC_{k}\cup\cD$ is a normal crossing divisor around $F$.
(ii) $\deg_{k(F)}O_{\cC}(F)|_{F}\leq-1$ (resp.\,$\deg_{k(F)}O_{\cC}(F)|_{F}\leq-2$).
\end{enumerate}
\end{nota}

\begin{rem}
\label{sepcase}
Let $\cC^{\log}=(\cC,\cD)$ be a regular n.c.d.\,model of $C^{\log}$.
\begin{enumerate}
\item
If $k$ is separably closed, we have $\cE_{0,1'}(\cC)=\emptyset=\cE_{\node}(\cC)$.
In particular, in this case, every $F\in\cE_{\leq-2}(\cC)$ is isomorphic to $\bP^{1}_{k(F)}$ over $k(F)$.
\item
$\cC_{O_{K}^{h}}$ is also a regular n.c.d.\,model of $(C_{K^{h}},D_{K^{h}})$ whose special fiber is isomorphic to $\cC_{k}$.
Let $F$ be a prime divisor on $\cC$ whose support is contained in $\cC_{k}$ and $F^{h}$ the base change of $F$ to $\cC_{O_{K}^{h}}$.
Then $F\in\cE_{=-1}(\cC)$ (resp.\,$F\in\cE_{\leq-2}(\cC)$) if and only if $F\in\cE_{=-1}(\cC_{O_{K}^{h}})$ (resp.\,$F\in\cE_{\leq-2}(\cC_{O_{K}^{h}})$).
\item
For any $F\in\cE_{\node}(\cC)$, we have $F\in\cE_{\leq-2}(\cC)$ if $\Supp F\neq\cC_{k}$ by \ref{badcurves}.
\end{enumerate}
\end{rem}

We give a stronger statement in Lemma \ref{forbasechange} than the following lemma in the case where $2g+r-2>0$ holds.

\begin{lem}[cf.\,Lemma \ref{forbasechange}]
\label{sthencase}
Let $\cC^{\log}=(\cC,\cD)$ be a regular n.c.d.\,model of $C^{\log}$.
Note that $\cC_{O_{K}^{sh}}$ is also a regular n.c.d.\,model of $(C_{K^{sh}},D_{K^{sh}})$.
Let $F'$ be a prime divisor on $\cC_{O_{K}^{sh}}$ whose support is contained in the special fiber and $F$ the prime divisor on $\cC$ whose support is the image of $F'$ in $\cC$.
If $F\in\cE_{=-1}(\cC)$ (resp.\,$F\in\cE_{\leq-2}(\cC)$), then we have $F'\in\cE_{=-1}(\cC_{O_{K}^{sh}})$ (resp.\,$F'\in\cE_{\leq-2}(\cC_{O_{K}^{sh}})$).
\end{lem}

\begin{proof}
We only treat the case where $F\in\cE_{\node}(\cC)\cap\cE_{\leq-2}(\cC)$.
In this case, Lemma \ref{sthencase} follows from a similar argument to that given in the first paragraph of the proof of Lemma \ref{notationnormalcrossing}.
\end{proof}

\subsection{Regular n.c.d.\,models}
\label{minimalmodelconstructionsection}
In this subsection, we study fundamental properties of minimal regular n.c.d.\,models.

In this subsection, we suppose that the log regular curve $C^{\log}$ satisfies (F) (cf.\,Remark \ref{why(F)}.2).
First, we review the fundamental theory of regular models, for example, the existence of minimal regular models and minimal regular n.c.d.\,models.

\begin{sub}[The intersection theory on special fibers of models]
\label{intersectiontheory}
In this paper, we mainly use the terminology of \cite[Tag:0C5Y]{Stacks}.
Let $\cC$ be a regular model of $C$.
Note that $\cC$ is projective over $O_{K}$ by \cite[Tag:0C5P]{Stacks}.
For any divisor $E$ on $\cC$ whose support is contained in $\cC_{k}$, $E$ is an exceptional curve of the first kind if and only if there exist a regular model $\cC'$ of $C$ and a closed point $c'\in\cC'$ such that $\cC$ is isomorphic to the blow-up of $\cC'$ at $c'$ and  $E$ can be identified with the exceptional divisor of the blow-up via the isomorphism by \cite[Tag:0AGQ]{Stacks} and \cite[Tag:0C2M]{Stacks} (cf.\,\ref{firstkind}).
Note that this equivalence holds if $\cC$ is a (not necessarily proper) model and $E$ is contained in the regular locus of $\cC$.

For an integral effective divisor $E$ on $\cC$ whose support is contained in $\cC_{k}$ and an invertible sheaf $\cL$ on $\cC$, we define
$$(E\cdot\cL):=[k(E),k]\deg \cL|_{E}=\chi(\cL|_{E})-\chi(O_{E}),$$
where $\chi(-):=\dim_{k}H^{0}(X,-)-\dim_{k}H^{1}(X,-)$ (cf.\,\cite[Tag:0C64]{Stacks}).
Note that we can define a symmetric bilinear form on the subsgroup of the Weil divisor group of $\cC$ generated by divisors whose supports are contained in $\cC_{k}$ by using this $(-\cdot-)$ (cf.\,\cite[Tag:0C65]{Stacks}).
Let $E$ be an effective Cartier divisor on $\cC$ whose support is contained in $\cC_{k}$ satisfying condition (Fir$_{1}$).
Then $E$ satisfies condition (Fir$_{2}$) if and only if $(E\cdot E)=-[k(E):k]$ is satisfied.
Suppose that $E$ is an exceptional curve of the first kind on $\cC$ whose support is contained in $\cC_{k}$ and write $\pi:\cC\to\cC'$ for the contraction of $E$.
Let $F'_{1}$ and $F'_{2}$ be prime effective divisors on $\cC'$ one of whose support is contained in $\cC'_{k}$.
Write $F_{1}$ and $F_{2}$ for the strict transforms of $F'_{1}$ and $F'_{2}$, respectively.
By the proof of \cite[Tag:0C6C]{Stacks}, we have
\begin{align}
\label{contractionnumber}
(F'_{1}\cdot F'_{2})=(F_{1}\cdot F_{2})-(F_{1}\cdot E)(F_{2}\cdot E)/(E\cdot E).
\end{align}
(Note that, in \cite[Tag:0C6C]{Stacks}, both $F'_{1}$ and $F'_{2}$ are assumed to be contained in $\cC_{k}$.
Even if we drop the assumption that ``$C_{i}$" in \cite[Tag:0C6C]{Stacks} is contained in $\cC_{k}$, the same proof works.
Indeed, we have
\begin{align}
\label{onF}
(F'_{1}\cdot F'_{2})&=\deg_{F'_{1}/k}O_{\cC'}(F'_{2})|_{F'_{1}}\notag\\
&=\deg_{F_{1}/k}\pi^{\ast}O_{\cC'}(F'_{2})|_{F_{1}}\notag\\
&=\deg_{F_{1}/k}(O_{\cC}(F_{2})+\mu O_{\cC}(E))|_{F_{1}}\notag\\
&=(F_{1}\cdot F_{2})+\mu (F_{1}\cdot E),
\end{align}
where $\mu$ is the maximal natural number satisfying that $\fm_{e}^{\mu}$ contains the local equation of $F_{1}$ at $e:=\pi(E)$ (cf.\,\cite[Proposition 9.2.23]{Liu}).
On the other hand, $\mu$ can be calculated as follows:
\begin{align}
\label{onE}
0=\deg_{E/k}O_{\cC'}(F'_{2})|_{E}=(E\cdot F_{2})+\mu (E\cdot E).
\end{align}
By combining (\ref{onF}) and (\ref{onE}), we obtain (\ref{contractionnumber}).)
In particular, we have an inequality
\begin{align}
\label{contractionself}
(F'_{1}\cdot F'_{1})=(F_{1}\cdot F_{1})+[k(E):k]^{-1}(F_{1}\cdot E)^{2}\geq(F_{1}\cdot F_{1}).
\end{align}
Note that the following are equivalent:
(i)${}_{0}$ The inequality (\ref{contractionself}) is an equality (i.e., $(E\cdot F_{1})=0$).
(ii)${}_{0}$ $E\cap F_{1}=\emptyset$.
(iii)${}_{0}$ $e\notin F'_{1}$.
\end{sub}

\begin{sub}[Minimal regular n.c.d.\,models of curves over discrete valuation rings]
\label{existenceNCD}
We review the theory of regular models of curves over discrete valuation rings.
Fix a minimal regular model $\cC_{\reg}$ of $C$ (cf.\,\cite[Tag:0C2W]{Stacks}).
(Note that, if $g\geq1$, $C$ has unique minimal regular model by \cite[Tag:0C6B]{Stacks}.)
Write $\cD$ for the scheme theoretic closure of $D$ in $\cC_{\reg}$.
Since $C^{\log}$ satisfies (F), the normalization of $\Spec O_{K}$ in the fields of fractions of $D$ is finite over $\Spec O_{K}$.
By using this fact and the proofs of \cite[Theorem 9.2.26]{Liu} and \cite[Lemma 9.2.32]{Liu},  we obtain a regular n.c.d.\,model by iterating a blow-up at closed points of $\cC_{\reg}$.
Moreover, by iterating a contraction of an exceptional curve of the first kind such that the resulting regular model is a regular n.c.d.\,model, we obtain a minimal regular n.c.d.\,model (cf.\,\cite[Remark (3.2)]{Sai2}).

We also explain blow-ups and contractions of regular n.c.d.\,models.
Let $(\cC,\cD)$ be a model of $C^{\log}$ and $c\in\cC$ a closed point.
If $\cC_{k}\cup\cD$ has normal crossings in $\cC$ around $c$, the exceptional divisor of the blow-up $\cC'\to\cC$ at $c$ is contained in $\cE_{=-1}(\cC')$ by \ref{ncd} and \ref{intersectiontheory}.
On the other hand, if $E\in\cE_{=-1}(\cC)$ and $\cC$ is projective over $O_{K}$, we have a contraction $\cC\to\cC''$ of $E$ such that $\cC''_{k}\cup\cD''$ has normal crossings in $\cC''$ around the image of $E$ again by \ref{ncd} and \ref{intersectiontheory}.
Here, $\cD''$ is the scheme theoretic closure of $D$ in $\cC''$.
\end{sub}

\begin{lem}
\label{ramificationindex}
Let $(\cC,\cD)$ be a regular n.c.d.\,model of $C^{\log}$, $s\in\cD$ a closed point, and $e$ the ramification index of $O_{K}\subset O_{\cD,s}$.
Then $s$ is contained in exactly $1$ irreducible component $F$ of $\cC_{k}$ and the multiplicity $m$ of $F$ in $\cC_{k}$ is $e$.
\end{lem}

\begin{proof}
The first assertion follows from the fact that $\dim O_{\cC,s}=2$.
Then we can write $\varpi=f^{m}u(\in O_{\cC,s})$ by using $u\in O_{\cC,s}^{\ast}$ and a local defining equation $f\in O_{\cC,s}$ of $F$.
Since $(\cC,\cD)$ is a regular n.c.d.\,model, the image of $f$ in $O_{\cD,s}$ is a uniformizer.
Thus, we have $m=e$.
\end{proof}

Here, we prove the uniqueness of minimal regular n.c.d.\,models under a certain condition of $(g,r)$, which is already known in the cases of $D=\emptyset$.

\begin{lem}
\label{uniqueNCD}
Suppose $2g+r-2>0$.
There exists an absolutely minimal regular n.c.d.\,model $(\cC_{\ncd},\cD_{\ncd})$ (resp.\,an absolutely minimal regular s.n.c.d.\,model $(\cC_{\sncd},\cD_{\sncd})$) of $C^{\log}$ unique up to canonical isomorphism.
Moreover, if $C$ has an absolutely minimal regular model $\cC_{\reg}$, the canonical birational morphisms $\cC_{\sncd}\to\cC_{\ncd}\to\cC_{\reg}$ can be obtained by the iteration of blow-up at a closed point where the complement of $C\setminus D$ in a model does not have (simple) normal crossings.
\end{lem}

\begin{proof}
(cf.\,The proof of \cite[Tag:0CDA]{Stacks}.)
Since the assertions for regular s.n.c.d.\,models follow from those for regular n.c.d.\,models, we only treat the assertions for regular n.c.d.\,models.
The desired existence follows from the discussion in \ref{existenceNCD}.
Let $(\cC,\cD)$ and $(\cC',\cD')$ be minimal regular n.c.d.\,models of $C^{\log}$.

Suppose that $C$ has an absolutely minimal regular model $\cC_{\reg}$ and write $\cD_{\reg}$ for the scheme theoretic closure of $D$ in $\cC_{\reg}$.
Note that we have a canonical morphism of models $(\cC,\cD)\to(\cC_{\reg},\cD_{\reg})$ and $(\cC',\cD')\to(\cC_{\reg},\cD_{\reg})$.
If $(\cC_{\reg},\cD_{\reg})$ is a regular n.c.d.\,model, both of the morphism of models are isomorphic.
Suppose that $\cC_{\reg,k}\cup\cD_{\reg}$ does not have normal crossings at a closed point $c$.
Then, by the proof of \cite[Tag:0C5R ]{Stacks}, both of $\cC\to\cC_{\reg}$ and $\cC'\to\cC_{\reg}$ factor through the blow-up of $\cC_{\reg}$ at $c$.
By iterating this operation at most for the number of the irreducible components of $\cC_{1,k}$, we can see that the last assertion of Lemma \ref{uniqueNCD} holds.

Suppose that $C$ does \textbf{not} have a unique minimal regular model.
(In this case, by \cite[Tag:0CDA]{Stacks}, the special fiber of every minimal regular model of $C$ is isomorphic to $\bP^{1}_{k}$.)
Since $\cD$ and $\cD'$ is the scheme theoretic closure of $D$ in $\cC$ and $\cC'$, respectively, it suffices to show that we have a canonical isomorphism $\cC\simeq\cC'$.
By \cite[Tag:0C5S]{Stacks}, there exists a sequence of schemes
$$\cC=\cC_{0}\leftarrow\cdots\leftarrow\cC_{m}=\cC'_{n}\to\cdots\to\cC'_{0}=\cC'$$
such that each morphism is a blow-up at a closed point.
Write $\cD_{i}$ (resp.\,$\cD'_{j}$) for the scheme theoretic closure  of $D$ in $\cC_{i}$ (resp.\,$\cC'_{j}$) for each $1\leq i\leq m$ (resp.\,$1\leq j\leq n$).
By the argument in \ref{existenceNCD}, $(\cC_{i},\cD_{i})$ and $(\cC'_{j},\cD'_{j})$ are regular n.c.d.\,models of $(C,D)$.
We prove Lemma \ref{uniqueNCD} by induction on $n$.

Suppose $n=0$.
Since $(\cC',\cD')$ and $(\cC,\cD)$ are minimal regular n.c.d.\,models, the morphism $(\cC_{m},\cD_{m})\to (\cC_{0},\cD_{0})$ is an isomorphism.

Suppose $n>0$.
Write $E_{m}$ for the exceptional divisor of $\cC'_{n}\to\cC'_{n-1}$ and write $E_{i}$ for the image of $E_{m}$ in $\cC_{i}$ for $0\leq i< m$.
It suffices to show that $E_{0}$ is a point.
Indeed, if this holds, the morphism $(\cC'_{n}=)\cC_{m}\to\cC$ factors through $\cC'_{n}\to\cC'_{n-1}$ by \cite[Tag:0C5J]{Stacks} and the resulting morphism $\cC'_{n-1}\to\cC$ is an iterate of a blow-up at a closed point by \cite[Tag:0C5R]{Stacks}, from which the desired assertion follows by induction hypothesis.
Suppose that $E_{0}$ is not a point.
By (\ref{contractionself}), we have
$$-[k(E_{m}):k]=(E_{m}\cdot E_{m})\leq (E_{m-1}\cdot E_{m-1})\leq\cdots\leq (E_{0}\cdot E_{0})\leq 0.$$

Suppose that $(E_{i}\cdot E_{i})=(E_{i-1}\cdot E_{i-1})$ holds for each $1\leq i\leq m$.
Then the center of $\cC_{i}\to\cC_{i-1}$ is not in $E_{i-1}$ by the equivalence of (i)${}_{0}$ and (iii)${}_{0}$ in \ref{intersectiontheory}, which shows that $\cC_{i}\to\cC_{i-1}$ is isomorphic over some open neighborhood of $E_{i}$.
It follows from this and $E\in\cE_{-1}(\cC'_{n})$ that $E_{0}\in\cE_{=-1}(\cC)$, which contradicts the assumption that $(\cC,\cD)$ is a minimal regular n.c.d.\,model.
Let $i'$ be the maximal integer satisfying $(E_{i'}\cdot E_{i'})<(E_{i'-1}\cdot E_{i'-1})$.
Then $E_{i'}$ is an exceptional curve of the first kind and $E_{i'}$ intersects another exceptional curve of the first kind.
By \cite[Tag:0C6A]{Stacks}, it holds that $i'=1$, $E_{0}=\cC_{k}\simeq\bP^{1}_{k}$, and $\cC_{1}\to\cC_{0}$ is a blow-up at a $k$-rational point.
Since $i'=1$, the center of the blow-up $\cC_{i}\to \cC_{i-1}$ is not contained in $E_{i-1}$ for any $2\leq i\leq m$.
Since $\cC_{k}$ is reduced by the fact that $\cC_{k}\simeq\bP^{1}_{k}$, there are no closed points $s$ of $\cD$ such that the ramification index of the extension $O_{K}\subset O_{\cD,s}$ is more than $1$ by Lemma \ref{ramificationindex}.
By using this fact and the fact that $\cC_{1}\to\cC_{0}$ is a blow-up at a $k$-rational point, $(E_{0}\setminus\{e_{0}\})\cap \cD_{1,k}$ is finite (flat) over $k$ of rank $\geq2$.
From these observations, $E_{m}$ cannot be contained in $\cE_{=-1}(\cC_{m})$.
Therefore, the image of $E$ in $\cC$ is a point and this completes the proof of Lemma \ref{uniqueNCD}.
\end{proof}

\begin{rem}
We consider the case where $2g+r-2\leq 0$.
Note that a similar argument to that in the proof of Lemma \ref{uniqueNCD} also works for the following case:
\begin{itemize}
\item
$(g,r)=(1,0)$.
\item
$(g,r)=(0,2)$ and the normalization of $\Spec O_{K}$ in $D$ has exactly $1$ closed point.
\item
$C$ admits no regular models $\cC$ such that $\cC_{k}$ is isomorphic to $\bP^{1}_{k}$ over $k$.
(cf.\,Lemma \cite[Tag:0CDA]{Stacks}.)
\end{itemize}
On the other hand, in the remaining case, the uniqueness of Lemma \ref{uniqueNCD} always does \textbf{not} hold.
We see this in the following:
\begin{itemize}
\item
Suppose that $(g,r)=(0,0)$ and $C$ has a regular model $\cC$ such that $\cC_{k}$ is isomorphic to $\bP^{1}_{k}$ over $k$.
Let $\cC'\to\cC$ be a blow-up of $\cC$ at a $k$-rational point and $\cC'\to\cC''$ the contraction of the strict transform of $\cC_{k}$ in $\cC'$.
Then $\cC'$ and $\cC''$ are minimal regular n.c.d.\,models of $C$ which are not isomorphic.
\item
Suppose $C\setminus D\simeq \bG_{a,K}$ (resp.\,$C\setminus D\simeq \bG_{m,K}$).
Then the natural compactification of $\bG_{a,O_{K}}$ (resp.\,$\bG_{m,O_{K}}$) defines a minimal regular n.c.d.\,model.
By considering the automorphism of $C$ over $K$ determined by $\varpi\times$, we can see that the curve $(C,D)$ has more than two minimal regular n.c.d.\,models.
\item
Suppose that $(g,r)=(0,2)$ and the normalization of $\Spec O_{K}$ in $D$ has (exactly) two closed points.
By Galois descent and the discussion for the case of $C\setminus D\simeq\bG_{m,K}$, we can see that $\cC$ has more than two minimal regular n.c.d.\,models.
\end{itemize}
\end{rem}

\begin{nota}[cf.\,Remark \ref{notationamb}]
Suppose $2g+r-2>0$.
We shall write $\cC^{\log}_{\ncd}$ (resp.\,$\cC^{\log}_{\sncd}$) for the minimal regular n.c.d.\,model (resp.\,the minimal regular s.n.c.d.\,model) of $C^{\log}$ which exists and is unique up to canonical isomorphism by Lemma \ref{uniqueNCD}.
\end{nota}

\begin{rem}
\label{notationamb}
The models $\cC^{\log}_{\ncd}$, $\cC^{\log}_{\sncd}$, and $\cC^{\log}_{\lreg}$ (cf.\,Notation \ref{lregnota}) depend on $D$ (and is not determined only by $C$). 
\end{rem}

\subsection{Log regular models}
\label{lregsubsection}

In this subsection, we study fundamental properties of log regular models of $C^{\log}$.
We characterize effective divisors on log regular models contractible to points in log regular models.
Finally, we introduce the ``smallest" minimal log regular models $\cC^{\log}_{\lreg}$.

\begin{lem}
\label{fundlog}
Let $\cC^{\log}=(\cC,\cD)$ be a log regular model of $C^{\log}$ and $\widetilde{\cC}^{\log}$ the minimal desingularization of $\cC^{\log}$ (cf.\,Notation-Definition \ref{minimaldef}).
\begin{enumerate}
\item
$\widetilde{\cC}^{\log}$ is a regular n.c.d.\,model of $C^{\log}$.
\item
$\cC$ is projective over $O_{K}$.
\item
$\cD$ is normal.
\end{enumerate}
\end{lem}

\begin{proof}
Assertion 1 follows from \cite[Lemma 5.2]{Niz}.
Assertion 2 follows from Lemma \ref{rationalsingularity}.3, Proposition \ref{2dimfan}, and \cite[Corollary (27.2)]{Lip}.
Assertion 3 follows from Theorem \ref{logsummary}.4.
\end{proof}

\begin{dfn}
\label{conditiondivisorcurve}
Let $\cC^{\log}=(\cC,\cD)$ be a log regular model of $C^{\log}$ and $E$ an effective reduced Cartier divisor on $\cC$ whose support is contained in $\cC_{k}$.
We consider the following conditions:
\begin{itemize}
\item[(cExc)]
$\cC_{k}\cup \cD$ is a normal crossing divisor on $\cC$ around $\Supp E$ and $E$ satisfies the condition obtained by replacing $\widetilde{X}$ in one of conditions (Exc$_{x}'$), (Exc$_{\mathrm{o}}'$), or (Exc$_{\mathrm{e}}'$) with $\cC$ is satisfied.
\end{itemize}
Suppose that condition (cExc) is satisfied and write $k'$ for the field $H^{0}(E,O_{E})$.
Let $D'(E)$ be the sum of prime divisors $P$ such that $\Supp P\subset \cC_{k}\cup\cD$ and $\Supp P\not\subset\Supp E$.
We consider the following condition:
\begin{itemize}
\item[(cStr)]
There exists a regular open subset $\cU\subset\cC$ containing $\Supp E$ such that $D':=D'(E)\cap\cU$ and $E$ satisfies the condition obtained by replacing $k(x)$ and $\widetilde{X}$ in one of conditions (Str$_{0}$), (Str$_{1,x}$), (Str$_{1}$), and (Str$_{2}$) with $k'$ and $\cU$, respectively.
\end{itemize}
\end{dfn}

\begin{prop}
\label{logexcpetionaldivisor}
Let $\cC^{\log}=(\cC,\cD)$ be a log regular model of $C^{\log}$ and $E$ an effective reduced divisor whose support is contained in $\cC_{k}$.
The following are equivalent:
\begin{enumerate}
\item
$E$ satisfies conditions (cExc) and (cStr).
\item
There exist a log regular model $\cC'^{\log}$ of $C^{\log}$ and a non-regular point $c'\in F(\cC')^{(2)}$ satisfying the following:
For $c\in F(\cC)^{(2)}$, define $\Delta(c)$ to be the trivial proper subdivision (resp.\,the coarsest subdivision) of the fan $\{\tau\subset\sigma\mid\tau\text{ is a face of }\sigma_{\overline{c}}\}$ if $c\neq c'$ (resp.\,$c=c'$).
Then the log regular model $\cC'^{\log}_{\Delta}$ of $C^{\log}$ is isomorphic to $\cC^{\log}$ and $E$ corresponds to the exceptional divisor of the log blow-up $\cC'^{\log}_{\Delta}\to\cC'^{\log}$ via this isomorphism (cf.\,Proposition \ref{minlogdes}).
\end{enumerate}
\end{prop}

\begin{proof}
The implication 2$\Rightarrow$1 follows from Theorem \ref{logsummary}.2.
Suppose that $E$ satisfies conditions (cExc) and (cStr).
To construct the underlying scheme of the desired log regular model $\cC'^{\log}$, we need to check the following assumptions of the contractibility criterion \cite[Theorem (27.1)]{Lip}:
(i) $\cC$ is projective over $O_{K}$.
(ii) $E$ is an effective reduced Cartier divisor.
(iii) $E$ is connected.
(iv) The matrix $\bpm(E_{i}\cdot E_{j})\epm_{ij}$ is negative-definite.
(v) The fundamental curve satisfies $\chi(E_{f})>0$.
These conditions follow from condition (cExc), Definition \ref{divisorconditions}, Lemma \ref{rationalsingularity}.2, and Lemma \ref{fundlog}.2.
By applying \cite[Theorem (27.1)]{Lip} to $\cC$ and $E$, we obtain a normal model $\cC'$.
Write $\cD'$ for the scheme theoretic closure of $D$ in $\cC'$ and $c'$ for the image of $E$ in $\cC'$.
It suffices to show that $(\cC',\cC'\setminus(\cC'_{k}\cup\cD'))$ is a toric pair.
This follows from Lemma \ref{stalkwiselogregular}, Theorem \ref{logsummary}.2, and conditions (cExc) and (cStr).
\end{proof}

\begin{lem}
\label{genus01}
Let $\cC^{\log}=(\cC,\cD)$ be a regular n.c.d.\,model of $C^{\log}$.
Write $E$ for the reduced effective divisor $\cC_{k,\red}$ on $\cC$.
\begin{enumerate}
\item
If $E$ is an element of $\cE_{\bP,1}(\cC)$ (resp.\,$\cE_{\bP,2}(\cC)\cup\cE_{0,1'}(\cC)\cup\cE_{\node}(\cC)$), then $\cC_{k}$ is integral (and hence $\cC_{k}=E$) and $(g,r)$ is equal to $(0,1)$ (resp.\,$(0,2)$).
\item
If $E$ is the sum of an element $F$ of $\cE_{\bP,1}(\cC)$ and some elements of $\cE_{\leq-2}(\cC)$, then we have $F=\cC_{k}$ (in particular, we have $\cE_{\leq-2}(\cC)=\emptyset$).
\item
If $E$ is the sum of an element of $(\cE_{0,1'}(\cC)\cup\cE_{\bP,2}(\cC))\cap\cE_{=-1}(\cC)$ and some elements of $\cE_{\leq-2}(\cC)$, then we have $(g,r)=(0,2)$ or $(g,r)=(1,0)$.
\item
If $E$ is the sum of some elements of $\cE_{\leq -2}(\cC)$, then we have $(g,r)=(1,0)$.
\end{enumerate}
\end{lem}

\begin{proof}
We discuss all assertions simultaneously.
Write $k'$ for the field $H^{0}(E,O_{E})$ and $n$ for the number of irreducible components of $\cC_{k}$.
If $n=1$, we have $(E\cdot E)=0$ and hence $E\notin \cE_{=-1}(\cC)\cup\cE_{\leq -2}(\cC)$.
(This does not occur in the situation of assertion 3 or 4.)
In this case, it follows that
\begin{align}
\label{irrcase}
g=1+m\bigl([k(E):k](0-1)-\frac{1}{2}(E\cdot E)\bigr)=1-m[k(E):k]
\end{align}
from \cite[Tag:0CA3]{Stacks}, where $m$ is the multiplicity of $E$ in $\cC_{k}$.
Thus, we have $m=1$, $k(E)=k$, and $g=0$.
Assertion 1 follows from these observations.

In the rest of the proof, we assume $n\geq 2$ and show assertions 2,3, and 4.
By Lemma \ref{sthencase}, we may assume $O_{K}=O_{K}^{sh}$.
(Note that we have $\cE_{\node}(\cC)=\emptyset=\cE_{0,1'}(\cC)$ by Remark \ref{sepcase}.1.)
We discuss three cases (I), (II), and (III) separately.

(I): We suppose that there exists a prime divisor $F_{1}\in\cE_{\bP,1}(\cC)$.
Then we are in the situation of assertion 2, and thus we can write $E$ as $\sum_{i=1}^{n}F_{i}$ such that $F_{i}\in\cE_{\leq-2}(\cC)(\subset\cE_{\bP,2}(\cC))$ for $1< i\leq n$ and $F_{i}\cap F_{i+1}\neq\emptyset$.
Write $m_{i}$ for the multiplicity of $F_{i}$ in $\cC_{k}$.
By calculating $(F_{i}\cdot \varpi O_{\cC})$, we obtain
\begin{align}
\label{chainineq1}
m_{1}(F_{1}\cdot F_{1})+m_{2}(F_{1}\cdot F_{2})&=0,\\
\label{chainineq2}
m_{i-1}(F_{i}\cdot F_{i-1})+m_{i}(F_{i}\cdot F_{i})+m_{i+1}(F_{i}\cdot F_{i+1})&=0\quad \text{if }1<i<n,\text{ and}\\
\label{chainineq3}
m_{n-1}(F_{n}\cdot F_{n-1})+m_{n}(F_{n}\cdot F_{n})&=0.
\end{align}
By (\ref{chainineq1}) and (\ref{chainineq2}), we have $m_{1}\leq m_{2}\leq\cdots\leq m_{n}$, which contradicts (\ref{chainineq3}).
In (II) and (III), we assume that every prime divisor $F$ satisfying $F\leq E$ is contained in  $\cE_{\bP,2}(\cC)$.

(II): We suppose that there exists a prime divisor $F_{1}\in\cE_{\bP,2}(\cC)$ satisfying $F_{1}\cap\cD\neq\emptyset$.
Then we can write $E$ as $\sum_{i=1}^{n}F_{i}$ such that $F_{i}\in\cE_{\bP,2}(\cC)$ for $1< i\leq n$, $F_{i}\cap F_{i+1}\neq\emptyset$, and $F_{n}\cap\cD\neq\emptyset$.
Write $m_{i}$ for the multiplicity of $F_{i}$ in $\cC_{k}$.
By the same calculation as that in case (I), we have $F_{1},F_{n}\notin\cE_{=-1}(\cC)$ and $F_{i}\in\cE_{=-1}(\cC)$ for some $1<i<n$.
Write the regular n.c.d.\,model $\cC'^{\log}$ of $C^{\log}$ obtained by contracting $F_{i}$ and $F'_{j}$ for the prime divisor on $\cC'$ whose support is the image of $F_{j}$ in $\cC'$ for each $1\leq j(\neq i)\leq n$.
Again by the same calculation, at least one of $F'_{i-1}$ or $F'_{i+1}$ is contained in $\cE_{=-1}(\cC)$.
If both $F'_{i-1}$ and $F'_{i+1}$ are contained in $\cE_{=-1}(\cC)$, we have $(g,r)=(0,2)$ by \cite[Tag:0C6A]{Stacks}.
Also, if exactly one of $F'_{i-1}$ or $F'_{i+1}$ is contained in $\cE_{=-1}(\cC)$, we have $(g,r)=(0,2)$ by induction on $n$.

(III): We suppose that $E$ is the sum of some elements of $\cE_{\bP,2}(\cC)$ and $D\neq\emptyset$ (or equivalently, $r=0$).
Then we can write $E$ as $\sum_{i\in\Z/n\Z}F_{i}$ such that $F_{i}\in\cE_{\bP,2}(\cC)$ and $F_{i}\cap F_{i+1}\neq\emptyset$ for each $i\in\Z/n\Z$.
Write $m_{i}$ for the multiplicity of $F_{i}$ in $\cC_{k}$.
By \cite[Tag:0CA3]{Stacks}, we have
\begin{align*}
g
=&1+\sum_{i\in\Z/n\Z}m_{i}\bigl([k':k](0-1)-\frac{1}{2}(F_{i}\cdot F_{i})\bigr)=1.
\end{align*}
Here, the last equality follows from the equation
\begin{align*}
\label{chainintersection}
0=m_{i-1}(F_{i}\cdot F_{i-1})+m_{i}(F_{i}\cdot F_{i})+m_{i+1}(F_{i}\cdot F_{i+1})
\end{align*}
for $i\in\Z/n\Z$.
Thus, we have $(g,r)=(1,0)$.
\end{proof}

\begin{thm}
\label{logreglem}
Suppose $2g+r-2>0$.
Let $\cC^{\log}=(\cC,\cD)$ be a regular n.c.d.\,model of $C^{\log}$.
We have a log regular model $\cC'^{\log}$ whose minimal log desingularization $\widetilde{\cC}'^{\log}$ is isomorphic to  $\cC^{\log}$ and the set of exceptional divisors of $\cC\to\cC'$ coincides with $\cE_{\leq-2}(\cC)$.
Moreover, any log regular model $\cC''^{\log}$ whose minimal log desingularization $\widetilde{\cC}''^{\log}$ is isomorphic to  $\cC^{\log}$ dominates $\cC'^{\log}$.
\end{thm}

\begin{proof}
For any log regular model $\cC''^{\log}$ whose minimal log desingularization $\widetilde{\cC}''^{\log}$ is isomorphic to  $\cC^{\log}$, the exceptional divisor $E''$ of $\widetilde{\cC}''\to\cC''$ is the sum of some elements of $\cE_{\leq -2}(\cC'')$ by Proposition \ref{logexcpetionaldivisor}.
By the universal property of contractions, it suffices to show that every connected reduced effective Cartier divisor which can be written as the sum of some elements of $\cE_{\leq -2}(\cC'')$ satisfies conditions (cExc) and (cStr), i.e., condition 1 in Proposition \ref{logexcpetionaldivisor}.
Fix such a connected reduced effective Cartier divisor $E$.

Suppose that $E$ does not intersect any irreducible component of $\cC_{k}\cup\cD$ not contained in $E$.
Since $\cC_{k}$ is connected, we have $\cC_{k}=\Supp E$.
Then, by Lemma \ref{genus01}.4, we have $(g,r)=(0,1)$, which contradicts the inequality $2g+r-2>0$.
Thus, $E$ intersects some irreducible components of $\cC_{k}$ not contained in $E$.
Then, by the definition of $\cE_{\leq -2}(\cC'')$, $E$ satisfies conditions (cExc) and (cStr).
This completes the proof of Theorem \ref{logreglem}.
\end{proof}

\begin{nota-dfn}[cf.\,Remark \ref{notationamb}]
\label{lregnota}
Suppose $2g+r-2>0$.
We write $\cC^{\log}_{\lreg}$ for the minimal log regular model of $C^{\log}$ which can be constructed by applying Theorem \ref{logreglem} to $\cC^{\log}_{\ncd}$.
\end{nota-dfn}

The next proposition says that $\cC^{\log}_{\lreg}$ is the ``smallest" minimal log regular model in some sense.

\begin{prop}
\label{mlregstr}
Suppose $2g+r-2>0$.
For any log regular model $\cC^{\log}$ of $C^{\log}$, we have a canonical morphism $\widetilde{\cC}^{\log}\to\cC^{\log}_{\ncd}$ and hence $\widetilde{\cC}^{\log}\to\cC^{\log}_{\lreg}$.
\end{prop}

\begin{proof}
This follows from Lemma \ref{uniqueNCD}.
\end{proof}

Finally, we see properties of $\cC_{\lreg}^{\log}$.
We start with the study of ``$\cE_{-}(-)$".

\begin{lem}
\label{Eproperty}
Let $\cC^{\log}$ be a log regular model of $C^{\log}$.
The strict transform of any element $\cE_{\bP,1}(\cC)$ (resp.\,$\cE_{0,1'}(\cC)$; $\cE_{\bP,2}(\cC)$; $\cE_{\node}(\cC)$) in $\widetilde{\cC}$ is contained in $\cE_{\bP,1}(\widetilde{\cC})$ (resp.\,$\cE_{0,1'}(\widetilde{\cC})$; $\cE_{\bP,2}(\widetilde{\cC})$; $\cE_{\node}(\widetilde{\cC})\cup\cE_{\bP,2}(\widetilde{\cC})$).
If $2g+r-2>0$, the image of any element of $\cE_{\bP,1}(\widetilde{\cC})$ (resp.\,$\cE_{0,1'}(\widetilde{\cC})$; $\cE_{\bP,2}(\widetilde{\cC})$; $\cE_{\node}(\widetilde{\cC})$) in $\cC$ is a closed point or defines an element of $\cE_{\bP,1}(\cC)$ (resp.\,$\cE_{0,1'}(\cC)$; $\cE_{\bP,2}(\cC)\cup\cE_{\node}(\cC)$; $\cE_{\node}(\cC)$).
\end{lem}

\begin{proof}
We only show the latter assertions.
Suppose $2g+r-2>0$.
Let $F'$ be an element of $\cE_{\bP,1}(\widetilde{\cC})$ (resp.\,$\cE_{0,1'}(\widetilde{\cC})$; $\cE_{\bP,2}(\widetilde{\cC})$; $\cE_{\node}(\widetilde{\cC})$) which does not satisfy the desired assertion.
Write $F$ for the scheme theoretic image of the morphism $F'\to\cC$.
Then the natural morphism $F'\to F$ is not isomorphic.
It follows  from this and Theorem \ref{logsummary}.4 that there exists an element $c'\in F'\cap F(\widetilde{\cC})^{(2)}$ such that $F'$ is normal at $c'$, the image $c$ of $c'$ in $F$ is contained $F(\cC)^{(2)}$, $F$ is not normal at $c$, and $c$ is not contained in any other irreducible component of $\cC_{k}\cup\cD$.
Therefore, since $F'$ does not satisfy the desired assertion, we have $\cC_{k,\red}=F$ and hence  $r=0$.
Then all prime divisors on $\widetilde{\cC}$ except $F'$ are contained in $\cE_{\leq-2}(\widetilde{\cC})$.
Since $\widetilde{\cC}\to\cC$ is not isomorphic, we have $(F'\cdot F')\leq -1$.
These contradict Remark \ref{sepcase}.3 and Lemmas \ref{genus01}.2, \ref{genus01}.3, and \ref{genus01}.4.
\end{proof}

\begin{prop}
\label{absoluetness}
Suppose $2g+r-2>0$.
\begin{enumerate}
\item
$\cE_{\bP,2}(\cC_{\ncd})\cup\cE_{0,1'}(\cC_{\ncd})\cup\cE_{\node}(\cC_{\ncd})=\cE_{\leq-2}(\cC_{\ncd})$.
\item
$\cE_{\bP,2}(\cC_{\lreg})=\cE_{0,1'}(\cC_{\lreg})=\cE_{\node}(\cC_{\lreg})=\emptyset$.
\item
Let $\cC^{\log}$ be a log regular model satisfying $\cE(\cC)=\emptyset$.
Then $\cC^{\log}\simeq\cC_{\lreg}^{\log}$ and $\cC^{\log}$ is an absolutely minimal log regular model.
\item
The following are equivalent:
(i) $\cC_{\lreg}^{\log}$ is an absolutely minimal log regular model.
(ii) $\cE_{\bP,1}(\cC_{\ncd})=\emptyset$.
(iii) $\cE_{\bP,1}(\cC_{\lreg})=\emptyset$.
\end{enumerate}
\end{prop}

\begin{proof}
Assertion 1 follows from the assumption $2g+r-2>0$, the definition of $\cC_{\ncd}$, and Lemma \ref{genus01}.1.
Assertion 2 follows from assertion 1, Lemma \ref{Eproperty}, and the definition of $\cC_{\lreg}^{\log}$.

Next, we show assertion 3.
Since the image of each element of $\cE_{=-1}(\widetilde{\cC})$ in $\cC$ is a point by Lemma \ref{Eproperty}, it follows that $\widetilde{\cC}^{\log}\simeq\cC_{\ncd}^{\log}$ by the universal property of contractions.
By a similar argument to that for $\cE_{\leq -2}(\widetilde{\cC})$ and $\cC_{\lreg}$, we have $\cC=\cC_{\lreg}$.
Let $\cC'^{\log}$ be another log regular model of $C^{\log}$.
To show the desired absoluteness, by the universal property of contractions, it suffices to show that the image of of each element $F'$ of $\cE_{\leq -2}(\widetilde{\cC'})$ in $\cC_{\lreg}$ (cf.\,Proposition \ref{mlregstr}) is a point.
Suppose that the image $F$ of $F'$ in $\cC(=\cC_{\lreg})$ is not a point and regard $F$ as an integral closed subscheme of $\cC_{k}$.
Then $F'\to F$ is not isomorphic by $\cE(\cC)=\emptyset$.
By the discussion in the proof of Lemma \ref{Eproperty}, we have $\cC_{k,\red}=F$.
In any case, by taking the minimal log desingularization of $\cC^{\log}$, or by blowing up $\cC$ at suitable points of $F(\cC)^{(2)}$, we obtain a regular n.c.d.\,model of $C^{\log}$ satisfying the assumption of one of Lemmas \ref{genus01}.2, \ref{genus01}.3, and \ref{genus01}.4, which contradicts the assumption that $2g+r-2>0$.

Finally, we show assertion 4.
The equivalence of (ii) and (iii) follows from the definition of $\cC_{\ncd}^{\log}$ and $\cC_{\lreg}^{\log}$.
Suppose $\cE_{\bP,1}(\cC_{\ncd})=\emptyset$.
Then, by assertion 1, we have $\cE(\cC_{\ncd})=\cE_{\bP,2}(\cC)\cup\cE_{0,1'}(\cC)\subset\cE_{\leq-2}(\cC_{\ncd})$.
Hence, by assertion 3, $\cC_{\lreg}^{\log}$ is an absolutely minimal log regular model.

Suppose that we have an element $F\in\cE_{\bP,1}(\cC_{\ncd})$.
Fix a $k(F)$-rational point $c\in F\setminus F(\cC)^{(2)}$.
Write $\cC'$ for the blow-up of $\cC_{\ncd}$ at $c$ and $F'$ for the strict transform of $F$ in $\cC'$.
Then we have a regular n.c.d.\,model $\cC'^{\log}$ of $C^{\log}$ and $F\in\cE_{\leq-2}(\cC')$.
The log regular model obtained by contracting the elements of $\cE_{\leq-2}(\cC')$ (and applying Theorem \ref{logreglem}) cannot dominate $\cC_{\lreg}^{\log}$ since the image of $F'$ in $\cC_{\lreg}$ is not a point.
\end{proof}

\section{Log smooth models}
\label{thirdsection}
In this section, we study properties of $\cC_{\lreg}^{\log}$.
In particular, we show that, if $C^{\log}$ has log smooth reduction (resp.\,stable reduction), $\cC_{\lreg}^{\log}$ is a log smooth model (resp.\,a stable model) of $C^{\log}$ in the case of $2g+r-2>0$.

In this section, we keep the notations of Section \ref{modelsection}.
We fix some more notations.
Let $J$ be the Jacobian variety of $C$, $\cJ$ the N\'eron model of $J$ over $O_{K}$, and $\cJ^{0}$ the identity component of $\cJ$.
Recall that $J$ is said to have {\em stable reduction} if the unipotent radical of $\cJ_{k}$ is trivial (cf.\,\cite[Definition 2.1]{DM}).
Let $K^{\sep}$ be a separable closure of $K^{sh}$, $G_{K}$ the absolute Galois group $\Gal(K^{\sep}/K)$, $I_{K}$ the inertia subgroup $\Gal(K^{sh}/K)$, $P_{K}$ the wild inertia subgroup of $I_{K}$, and $l$ a prime number \textbf{not} divisible by $p$.
Then we have the following exact sequence of \'etale cohomology groups:
\begin{align}
0&\to H^{1}_{\text{\'et}}(C_{K^{\sep}},\Z_{l})
\to H^{1}_{\text{\'et}}((C\setminus D)_{K^{\sep}},\Z_{l})\notag\\
&\to\Z_{l}\otimes_{\Z}\Z^{D(K^{\sep})}
\to H^{2}_{\text{\'et}}(C_{K^{\sep}},\Z_{l})\to 0.\tag{H}
\end{align}
Here, $I_{K}$ acts on $H^{2}_{\text{\'et}}(C_{K^{\sep}},\Z_{l})\simeq\Z_{l}(-1)$ trivially.

In this section, we only treat ``good" charts of morphisms of log schemes as follows:
\begin{sub}[cf.\,{\cite[Definition (2.9)]{Ka1}}]
\label{chardefinition}
Let $\cC^{\log}$ be a log regular model of $C^{\log}$, $c$ a point of $\cC_{k}$, and $\overline{c}$ a geometric point of $\cC$ over $c$.
In this section, we define a {\em chart of the morphism} $\cC^{\log}\to(\Spec O_{K})^{\log}$, for which we write $(\cU,\N\to Q\to\Gamma(\cU,O_{\cC}))$, to be a collection of an \'etale neighborhood $\cU\to\cC$ of $\overline{c}$, a chart $Q\to \Gamma(\cU,O_{\cC})$ (note that $Q$ is a fine monoid), and a monoid homomorphism $\N \to Q$ compatible with the ring homomorphism $O_{K}\to\Gamma(\cU,O_{\cC})$ satisfying the following:
$Q/Q^{\ast}\simeq M_{\cC,c}/O_{\cC,c}^{\ast}$, $Q$ is fine saturated, and $Q\to \Gamma(\cU,O_{\cC})$ is injective.
By elementary monoid theory, a chart of the morphism always exists.

Suppose that $\cC^{\log}$ is a log smooth model of $C^{\log}$.
By \cite[THEOREM (3.5)]{Ka1}, there exists a chart of the morphism $\cC^{\log}\to(\Spec O_{K})^{\log}$ satisfying the following:
\begin{itemize}
\item[(I)]
$\cU\to \Spec O_{K}\times_{\Spec \Z[\N]}\Spec \Z[Q]$ is \'etale at the image of $\overline{c}$.
\item[(II)]
The kernel and the torsion subgroup of the cokernel of the group homomorphism $\N^{\mathrm{gp}}\to Q^{\mathrm{gp}}$ are finite groups of order invertible on $\cU$.
\end{itemize}
\end{sub}

\subsection{Fundamental properties of log smooth morphisms}
\label{lsmsubsection}

In this subsection, we discuss the fundamental structures of log smooth models of log regular curves.
Some statements are generalizations and revisions of statements in \cite{Stix} (cf.\,Lemmas \ref{logsmoothchart} and \ref{logblowupet}.2).
Since we do \textbf{not} suppose $k$ is perfect, we need more delicate handling of log structures.

First, we see properties of $\cC_{k}$ of a log regular model $\cC^{\log}$.

\begin{lem}
\label{F1case}
Let $\cC^{\log}$ be a log regular model of $C^{\log}$.
\begin{enumerate}
\item
If $\cC^{\log}$ is a log smooth model, $(\cC_{k}\setminus F(\cC)^{(2)})_{\red}$ is smooth over $k$.
\item
Let $Z$ be an integral closed subscheme of $\cC_{k}$ of dimension $1$ whose multiplicity in $\cC_{k}$ is \textbf{not} divisible by $p$.
Suppose that $Z\setminus F(\cC)^{(2)}$ is smooth over $k$.
Then $\cC^{\log}\to(\Spec O_{K})^{\log}$ is log smooth at points of $Z\setminus F(\cC)^{(2)}$.
\end{enumerate}
\end{lem}

\begin{proof}
Suppose that $\cC^{\log}$ is a log smooth model.
Take a closed point $c\in\cC_{k}\setminus F(\cC)^{(2)}$ and geometric point $\overline{c}$ over $c$.
By considering a chart of the morphism $(\cU,\N\to Q\to\Gamma(\cU,O_{\cC}))$ satisfying conditions (I) and (II) in \ref{chardefinition} at $c$, it suffices to show that $(\Spec k[Q]/(Q\setminus Q^{\ast}))_{\red}$ is smooth over $k$.
Choose a lift of the generator of $Q/Q^{\ast}(\simeq M_{\cC,\overline{c}}/O^{\ast}_{\cC,\overline{c}}\simeq\N)$ in $Q$ and an isomorphism $Q^{\ast}\simeq (Q^{\ast})_{\tor}\times\Z^{\oplus r}$ for some $r\in\N_{\geq1}$.
Then $Q$ is naturally isomorphic to $\N\times (Q^{\ast})_{\tor}\times\Z^{\oplus r}$.
By the structure theorem of finitely generated abelian groups, $(\Spec k[Q]/(Q\setminus Q^{\ast}))_{\red}$ is isomorphic to the disjoint union of tori over finite separable extension fields over $k$.
Hence, assertion 1 holds.

Assertion 2 follows from the (relative) Abhyankar lemma and the Zariski-Nagata purity theorem.
\end{proof}

Next, we see properties of $F(\cC)^{(2)}$ of a log regular model $\cC^{\log}$.

\begin{lem}
\label{goodchart}
Let $\cC^{\log}$ be a log regular model, $c$ an element of $F(\cC)^{(2)}$, $\overline{c}$ a geometric point of $\cC$ over $c$, and $(\cU,\N\to Q\to\Gamma(\cU,O_{\cC}))$ a chart of the morphism $\cC^{\log}\to(\Spec O_{K})^{\log}$
\begin{enumerate}
\item
Suppose $Q^{\ast}=\{1\}$ (or equivalently, $Q\simeq M_{\cC,\overline{c}}/O_{\cC,\overline{c}}^{\ast}$).
Then the condition obtained by replacing ``\'etale" in condition (I) in \ref{chardefinition} with ``flat" is satisfied.
Moreover, if $k(c)$ is separable over $k$, condition (I) in \ref{chardefinition} is satisfied.
\item
If condition (I) in \ref{chardefinition} is satisfied, then we have $\sharp Q^{\ast}<\infty$ and $p\!\not{|}\,\sharp Q^{\ast}$.
\end{enumerate}
\end{lem}

\begin{proof}
Assertion 1 follows from \cite[Th\'eor\`eme 5.6]{SGA1}, \cite[(6.1) THEOREM]{Ka2}, and the fact that $(Q\setminus \{1\})O_{\cC,\overline{c}}$ is the maximal ideal of the local ring (, which follows from the log regularity).

Next, we show assertion 2.
Write $u$ for the image of $\overline{c}$ in $\cU$.
Then the homomorphism $O_{K}\otimes_{\Z[\N]}\Z[Q]\to O_{\cU,u}$ is flat.
By considering the quotient rings by the ideals generated by $Q\setminus Q^{\ast}$, we obtain a flat homomorphism $k[Q^{\ast}]\to k(u)$.
Hence, assertion 2 holds.
\end{proof}

\begin{lem}[cf.\,{\cite[Proposition 5.2]{Stix}}]
\label{logsmoothchart}
Let $\cC^{\log}$ be a log regular model, $c$ an element of $F(\cC)^{(2)}$, and $\overline{c}$ a geometric point of $\cC$ over $c$.
Consider the following conditions:
\begin{itemize}
\item[(a)]
$\cC^{\log}\to(\Spec O_{K})^{\log}$ is log smooth at $c$.
\item[(b)]
$\sharp((M^{\gp}_{\cC,\overline{c}}/O_{\cC,\overline{c}}^{\ast})/\varpi^{\Z})_{\tor}$ is not divisible by $p$.
\item[(b')]
Write $\{\xi_{1},\xi_{2}\}=\{\xi\in F(\cC)^{(1)}\mid c\in \overline{\{\xi\}}\}$.
(Note that $\xi_{1}=\xi_{2}$ can occur.)
Write $m_{i}$ for the valuation of $\varpi$ at $\xi_{i}$ for $i\in\{1,2\}$.
Then $\gcd (m_{1},m_{2})$ is \textbf{not} divisible by $p$.
\item[(c)]
The residue field $k(c)$ is separable over $k$.
\end{itemize}
Then condition (a) is satisfied if and only if conditions (b) and (c) are satisfied.
Moreover, (b')$\Rightarrow$(b) holds.
Furthermore, if $\cC$ is a regular at $c$, we have (b)$\Rightarrow$(b').
\end{lem}

\begin{proof}
First, suppose conditions (b) and (c) are satisfied.
It suffices to show that there exists a chart $(\cU, \N\to Q\to\Gamma(\cU,O_{\cC}))$ satisfying $Q^{\ast}=\{1\}$ by \cite[THEOREM (3.5)]{Ka1} and Lemma \ref{goodchart}.1.
Write $e$ for $\sharp((M^{\gp}_{\cC,\overline{c}}/O_{\cC,\overline{c}}^{\ast})/\varpi^{\Z})_{\tor}$.
Since there exists the $e$-th root of the image of $\varpi$ in the multiplicative monoid $P:=M_{\cC,\overline{c}}/O_{\cC,\overline{c}}^{\ast}$ by condition (b) and $O_{\cC,\overline{c}}$ is strictly henselian, we have $\varpi^{1/e}\in M_{\cC,\overline{c}}$.
Let $q$ be an element of $M_{\cC,\overline{c}}$ such that the image of $q$ and $\varpi^{1/e}$ in $P^{\gp}$ form a $\Z$-basis.
Then we obtain a section of the homomorphism $M_{\cC,\overline{c}}^{\gp}\to P^{\gp}$ by using $q$ and $\varpi^{1/e}$.
By using this section, we obtain a desired chart of the morphism \'etale locally on $\cC$.

Next, suppose that condition (a) is satisfied.
Take a chart $(\cU,\N\to Q\to\Gamma(\cU,O_{\cC}))$ satisfying conditions (I) and (II) in \ref{chardefinition}.
By using the isomorphism $(Q^{\gp}/Q^{\ast})\simeq M_{\cC,\overline{c}}/O_{\cC,\overline{c}}^{\ast}$, Lemma \ref{goodchart}.2, and condition (II) in \ref{chardefinition}, we can see that condition (b) is satisfied.

Next, we see the relation between (b) and (b').
By Example \ref{maximallogblowup}, we may assume $\xi_{1}\neq\xi_{2}$.
Then the valuations at $\xi_{1}$ and $\xi_{2}$ define an injection $M_{\cC,c}/O_{\cC,c}^{\ast}\hookrightarrow \N^{\oplus2}$ whose groupfication is of finite index.
Note that this injection is an isomorphism if $O_{\cC,c}$ is regular.
The composite of this homomorphism and $(O_{K}\setminus\{0\})/O_{K}^{\ast}\to M_{\cC,c}/O_{\cC,c}^{\ast}$ sends the image of $\varpi$ to $(m_{1},m_{2})$.
From these observations, the asserted relation between (b) and (b') holds.

Finally, we show (a)$\Rightarrow$(c).
Take a log blow-up $\cC'^{\log}\to\cC^{\log}$ such that there exists a prime divisor $\overline{\{\eta\}}$ on $\cC'$ whose image in $\cC$ is $\{c\}$.
Since $\cC'^{\log}$ is log \'etale over $\cC^{\log}$, $\cC'^{\log}$ is a log smooth model of $C^{\log}$.
Moreover, we have a field extension $k\subset k(c)\subset k(\eta)$, which shows that $k(c)$ is separable over $k$ if $k(\eta)$ is separable over $k$.
Hence, (a)$\Rightarrow$(c) follows from Lemma \ref{F1case}.1.
\end{proof}

Next, we see properties of $\cD$ of a log regular model.

\begin{lem}
\label{fundlogsm}
Let $\pi^{\log}:\cC^{\log}\to(\Spec O_{K})^{\log}$ be a log smooth model of $C^{\log}$.
$\cD$ is a normal scheme and each irreducible component of $\cD$ is at most tamely ramified over $\Spec O_{K}$.
\end{lem}

\begin{proof}
The desired normality follows from Lemma \ref{fundlog}.3.
By replacing $\cC^{\log}$ with its minimal log desingularization, we may assume that $\cC^{\log}$ is a regular n.c.d.\,model.
Then the ramification index (resp.\,the residue field extension) of each irreducible component of $\cD$ over $O_{K}$ is prime to $p$ by Lemma \ref{ramificationindex} and Lemma \ref{logsmoothchart} (resp.\,is separable by Lemma \ref{logsmoothchart}).
Thus, Lemma \ref{fundlogsm} holds.
\end{proof}

Finally, we see the properties of prime divisors on a log smooth model.

\begin{prop}[cf.\,{\cite[Theorem 1.3]{K}}]
\label{genK}
Let $\cC^{\log}$ be a log smooth model of $C^{\log}$.
Then $\cC_{k,\red}$ (resp.\,each irreducible components of $\cC_{k,\red}$) is a proper geometrically connected curve with only ordinary double points over $k$, whose singular locus coincides with $F(\cC)^{(2)}$ (resp.\,consists of the closed points of the curve where $\cC^{\log}$ is \textbf{not} Zariski log regular (cf.\,Theorem \ref{logsummary}.3)).
\end{prop}

\begin{proof}
This follows from Lemmas \ref{F1case}.1 and \ref{logsmoothchart}, Theorems \ref{logsummary}.3 and \ref{logsummary}.4, and the fact that the regular local ring essentially of finite finite type over a field is geometrically regular if its residue field is separable over the base field.
\end{proof}

\begin{lem}
\label{logblowupet}
Let $f^{\log}:\cC_{0}^{\log}\to\cC_{1}^{\log}$ be a morphism of log regular models of $C^{\log}$.
\begin{enumerate}
\item
Suppose that $f^{\log}$ is a composite morphism of log blow-ups of coherent ideals.
Then $\cC_{0}^{\log}$ is a log smooth model if and only if $\cC_{1}^{\log}$ is so.
\item
Suppose that $\cC_{0}^{\log}$ and $\cC_{1}^{\log}$ are regular n.c.d.\,models of $C^{\log}$ and $\cC_{0}^{\log}$ is a log smooth model.
Then $\cC_{1}^{\log}$ is a log smooth model.
\end{enumerate}
\end{lem}

\begin{proof}
First, we show assertion 1.
Since a log \'etale morphism is log smooth, $\cC_{0}^{\log}$ is a log smooth model if $\cC_{1}^{\log}$ is so.
Suppose $\cC_{0}^{\log}$ is a log smooth model.
It suffices to show that $\cC_{1}^{\log}\to(\Spec O_{K})^{\log}$ is log smooth at any point $c_{0}$ of $F(\cC_{1})^{(2)}$.
By using the surjectivity of $F(\cC_{0})^{(2)}\to F(\cC_{1})^{(2)}$, we take $c_{0}\in F(\cC_{0})^{(2)}$ over $c_{1}$ and a geometric point $\overline{c}$ over $c_{0}$.
Then we have extensions of fields $k\subset k(c_{1})\subset k(c_{0})$ and an isomorphism $M^{\gp}_{\cC_{1},\overline{c}}/O_{\cC_{1},\overline{c}}^{\ast}\simeq M^{\gp}_{\cC_{0},\overline{c}}/O_{\cC_{0},\overline{c}}^{\ast}$.
Therefore, assertion 1 follows from the first assertion of Lemma \ref{logsmoothchart}.

Next, we show assertion 2.
By \cite[Tag:0C5R]{Stacks}, we may assume that $\cC_{1}\to\cC_{0}$ is the blow-up at a closed point $c_{0}$ and it suffices to show that $\cC_{0}^{\log}$ is log smooth at $c_{0}$ over $(\Spec O_{K})^{\log}$.
If $c_{0}\in F(\cC_{0})^{(2)}$, assertion 2 follows from assertion 1 and Example \ref{maximallogblowup}.
We may assume $c_{0}\notin F(\cC_{1})^{(2)}$ and write $\xi$ (resp.\,$m$) for the point of $F(\cC_{0})^{(1)}$ satisfying $c_{0}\in\overline{\{\xi\}}$.
Since the multiplicity of $f^{-1}(c_{0})$ coincides with $m$, $m$ is not divisible by $p$ by Lemma \ref{logsmoothchart}.
From this and Lemma \ref{F1case}.2, it suffices show that the $(\cC_{1,k})_{\red}$ is smooth at $c_{0}$.
Since $(\cC_{1,k})_{\red}$ is regular at $c$ by the assumption that $\cC_{0}$ is a regular n.c.d.\,model and $k(c_{0})$ is separable over $k$ by Lemma \ref{logsmoothchart}, the local ring of $c\in (\cC_{1,k})_{\red}$ is geometrically regular over $k$.
This completes the proof of assertion 2.
\end{proof}

\begin{prop}
\label{logsmoothcase}
Suppose that $C^{\log}$ has log smooth reduction.
Then there exists a minimal log regular model of $C^{\log}$ which is a log smooth model.
In particular, if $2g+r-2>0$, $\cC^{\log}_{\lreg}$ and $\cC_{\ncd}^{\log}$ are log smooth models.
\end{prop}

\begin{proof}
This follows from Theorem \ref{logreglem}, Proposition \ref{mlregstr}, and Lemma \ref{logblowupet}.
\end{proof}

\subsection{Base change of models and the structure theorem}
\label{basechangesubsection}

In this subsection, we discuss base changes of $\cC^{\log}_{\ncd}$ and $\cC^{\log}_{\lreg}$.

Let $O_{K'}\supset O_{K}$ be an extension of discrete valuation rings, $K'$ the field of fractions of $O_{K'}$, and $k'$ the residue field of $O_{K'}$.
Suppose that the ramification index of the extension $O_{K'}\supset O_{K}$ is $1$.
We consider the log regular curve $C'^{\log}:=(C_{K'},D_{K'})$ and its models $\cC'^{\log}_{\lreg}$ and $\cC'^{\log}_{\ncd}$.

\begin{lem}
\label{forbasechange}
Let $\cC^{\log}$ be a regular n.c.d.\,model of $C^{\log}$.
Suppose that $((\cC_{k,\red}\setminus F(\cC)^{(2)})_{k'}:=)(\cC_{k,\red}\setminus F(\cC)^{(2)})\times_{\Spec k}\Spec k'$ is normal and $k(c)\otimes_{k}k'$ is reduced for any $c\in F(\cC)^{(2)}$.
\begin{enumerate}
\item
$(\cC',\cD'):=(\cC_{O_{K'}},\cD_{O_{K'}})$ is a regular n.c.d.\,model of $C'^{\log}$ and $(\cC_{k}\cup\cD)_{\red,k'}\simeq(\cC'_{k'}\cup\cD')_{\red}$.
\end{enumerate}
Let $F'$ be a prime divisor on $\cC'$ whose support is contained in $\cC'_{k'}$, $F$ the prime divisor on $\cC$ whose support is the image of $F'$, and $\widetilde{F}\to F$ the normalization.
\begin{enumerate}
\setcounter{enumi}{1}
\item
$(F_{k'}:=)F\times_{\Spec k}\Spec k'$ is reduced and $(\widetilde{F}_{k'}:=)\widetilde{F}\times_{\Spec k}\Spec k'\to F_{k'}$ is the normalization morphism.
\item
Suppose $(g,r)\neq(0,0), (0,2)$.
$F'\in\cE_{= -1}(\cC')$ if and only if $F\in\cE_{= -1}(\cC)$.
\item
Suppose $(g,r)\neq(1,0)$.
$F'\in\cE_{\leq -2}(\cC')$ if and only if $F\in\cE_{\leq -2}(\cC)$.
\end{enumerate}
\end{lem}

Because the proof of Lemma \ref{forbasechange} is long, we divide the proof into four parts.

\begin{proof}[Proofs of Lemmas \ref{forbasechange}.1 and \ref{forbasechange}.2]
First, note that $\cD'$ coincides with the scheme theoretic closure of $D_{K'}$ in $\cC'$ since $O_{K'}$ is flat over $O_{K}$.

We show assertion 1.
Write $F(\cC')^{(2)}$ for the set of points over $F(\cC)^{(2)}$.
(After we prove assertion 1, this notation is compatible with that defined in \ref{lognotationa}.)
Since $(\cC_{k,\red}\setminus F(\cC)^{(2)})_{k'}$ is a $1$-dimensional normal closed subscheme of $\cC'\setminus F(\cC')^{(2)}$ whose defining ideal is locally principal, $\cC'\setminus F(\cC')^{(2)}$ is regular.
Take a point $c\in F(\cC)^{(2)}$.
Since $k(c)\otimes_{k}k'$ is the product of the residue fields of the points of $\cC'$ over $c$, the maximal ideal of $O_{\cC,c}$ generates the maximal ideal of each such a point.
Hence, $\cC'$ is regular at points over $c$.
Moreover, if $c\in F(\cC)^{(2)}\cap\cD$, $\cD'$ is regular at points over $c$ by a similar argument.
Let $\Spec R\to\Spec O_{\cC,c}$ be a strict localization, $E$ for the normal crossing divisor on $\Spec R$ supported on the inverse image of $\cC_{k}\cup\cD$, and $\varpi_{1}$ and $\varpi_{2}$ prime elements of $R$ such that $\varpi_{1}\varpi_{2}$ defines $E$.
Let $c'$ be a point of $F(\cC')^{(2)}$ over $c$ and $\Spec R'\to (\Spec R\times_{\Spec O_{\cC,c}}\Spec O_{\cC',c'}\to)\Spec O_{\cC',c'}$ a strict localization.
Since the maximal ideal of $R$ generates the maximal ideals of $O_{\cC,c}$ and $R'$, $\varpi_{1}$ and $\varpi_{2}$ generate the maximal ideal of $R'$.
Hence, the base change of $E$ to $\Spec R'$ is a normal crossing divisor.
This completes the proof of assertion 1.

Next, we show assertion 2.
The first assertion follows from assertion 1.
To show the second assertion, by the assumption that $(\cC_{k,\red}\setminus F(\cC)^{(2)})_{k'}$ is normal, it suffices to show that $\widetilde{F}_{k'}$ is normal at points over $F(\cC')^{(2)}(\subset F_{k'})$.
Let $c$ be a point of $F(\cC)^{(2)}\cap F$, $\widetilde{c}'$ a point of $\widetilde{F}_{k'}$ over $c$, and $\widetilde{c}$ the image of $\widetilde{c}'$ in $\widetilde{F}$.
If we show $\fm_{\widetilde{c}}O_{\cC',\widetilde{c}'}=\fm_{\widetilde{c}'}$, the desired assertion holds.
Since $\widetilde{F}\times_{F}\Spec k(c)$ is reduced, it suffices to show that $\widetilde{F}_{k'}\times_{F}\Spec k(c)$ is reduced.
Since we have $\widetilde{F}_{k'}\times_{F}\Spec k(c)\simeq(\widetilde{F}\times_{F}\Spec k(c))\times_{\Spec k(c)}\Spec(k(c)\otimes_{k}k')$, the desired reducedness follows from the assumption that $k(c)\otimes_{k}k'$ is reduced and the fact that $\widetilde{F}\times_{F}\Spec k(c)\to\Spec k(c)$ is a finite \'etale morphism of rank $\leq 2$.
\end{proof}

\begin{proof}[Reduction of Lemmas \ref{forbasechange}.3 and \ref{forbasechange}.4 to a special case]
Here, we show that we may assume $O_{K}=O_{K}^{h}$ and $O_{K'}=O_{K}^{sh}$ to show Lemmas \ref{forbasechange}.3 and \ref{forbasechange}.4.

Let $O_{K'}^{h}$ be the henselization of $O_{K'}$ dominating $O_{K}^{h}$ over $O_{K}$ and $O_{K'}^{sh}$ the strict henselization of $O_{K'}^{h}$ dominating $O_{K}^{sh}$ over $O_{K}^{h}$.
Note that the base changes of $\cC$ to these discrete valuation rings are regular n.c.d.\,models by Lemma \ref{forbasechange}.1.
Take a prime divisor $F'^{sh}$ on $\cC_{O_{K'}^{sh}}$ over $F'$ and consider the prime divisor on each model whose support coincides with the image of $F'$ in the model.
By considering these models and divisors, we can reduce Lemmas \ref{forbasechange}.3 and \ref{forbasechange}.4 to the following three cases:
(i) $O_{K'}=O_{K}^{h}$.
(ii) $O_{K}=O_{K}^{sh}$ and $O_{K'}=O_{K}^{sh}$.
(iii) $O_{K}=O_{K}^{sh}$ and $O_{K'}=O_{K'}^{sh}$.
In the case of (i), Lemmas \ref{forbasechange}.3 and \ref{forbasechange}.4 follow from Remark \ref{sepcase}.2.
Here, we give the proofs of Lemmas \ref{forbasechange}.3 and \ref{forbasechange}.4 in the case of (iii).

Suppose that $k$ and $k'$ are separably closed.
Then $F_{k'}$ is irreducible and hence we have $F'=F_{k'}$ by Lemma \ref{forbasechange}.2.
From this and Lemma \ref{forbasechange}.2, it follows that $F$ is normal if and only if $F'$ is normal.
Moreover, we have $O_{\cC'}(F')=O_{\cC}(F)|_{\cC'}$ and $\deg_{k(F)}O_{\cC}(F)=\deg_{k(F')}O_{\cC'}(F')$.
By using these observations and Remark \ref{sepcase}.1, it suffices to show that $F\in\cE_{\bP,1}(\cC)$ (resp.\,$F\in\cE_{\bP,2}(\cC)$) if and only if $F'\in\cE_{\bP,1}(\cC')$ (resp.\,$F'\in\cE_{\bP,2}(\cC')$) to prove Lemmas \ref{forbasechange}.3 and \ref{forbasechange}.4.
These follow from the assumption that $k$ is separably closed and Lemma \ref{forbasechange}.1.
\end{proof}

In this paragraph, we assume $O_{K}=O_{K}^{h}$ and $O_{K'}=O_{K}^{sh}$.
Before we prove Lemmas \ref{forbasechange}.3 and \ref{forbasechange}.4, we make some observations.
Write $H^{0}(F_{k'},O_{F_{k'}})\simeq k(F)\otimes_{k}k'$ as the product of fields $\prod_{i=1}^{m}k'_{i}$.
There exists a canonical bijection between the set of connected component of $F_{k'}$ and that of $\Spec k(F)\times_{\Spec k}\Spec k'$.
Write $F'_{i}$ for the connected component of $F_{k'}$ satisfying $k(F'_{i})=k'_{i}$.
We may suppose $F'\leq F'_{1}$.
Then we have $F'_{i}=F\times_{\Spec k(F)}\Spec k'_{i}$ and $H^{1}(F_{i},\cL|_{F_{i}})\simeq H^{1}(F,\cL)\otimes_{k(F)}k'_{i}$ for each $1\leq i\leq m$ and every line bundle $\cL$ on $F$.
By using this observation, we have $\deg_{k(F)}O_{\cC}(F)|_{F}=\deg_{k'_{i}}O_{\cC}(F_{i})|_{F'_{i}}$.
Let $F''$ be the reduced effective Cartier divisor on $\cC'$ satisfying $F'+F''=F'_{1}$, $d$ the rank of the natural flat morphism from scheme theoretic intersection $F'\cap F''$ to $\Spec k(F')$, and $i_{1}$ the number of irreducible components of $F'_{1}$.
Then we have
\begin{align}
\deg_{k(F)}O_{\cC}(F)|_{F}
&=\deg_{k'_{1}}O_{\cC'}(F'_{1})|_{F'_{1}}
=i_{1}\deg_{k'_{1}}(O_{\cC'}(F'_{1})|_{F'})\tag{$\sharp$}\\
&=i_{1}[k(F'):k'_{1}](d+\deg_{k(F')}O_{\cC'}(F')|_{F'}).\notag
\end{align}
Here, the second equality follows from \cite[Proposition 9.1/5]{BLR}.

\begin{proof}[Proof of Lemma \ref{forbasechange}.3]
We may assume $O_{K}=O_{K}^{h}$ and $O_{K'}=O_{K}^{sh}$.
If $F\in\cE_{= -1}(\cC)$, it follows that $F'=F'_{1}$, $i_{1}=1$, $d=0$, and $F'_{1}\in\cE_{=-1}(\cC')$ from ($\sharp$).

Suppose $F'\in\cE_{=-1}(\cC')$.
(Note that, by Remark \ref{sepcase}.1, we have $F'\in\cE_{\bP,1}(\cC')\cup\cE_{\bP,2}(\cC')$.)
Then, for any prime divisor $F'''$ satisfying $F'''\leq F''$, $F'''\in\cE_{=-1}(\cC')$ holds.
Suppose $d\geq1$.
Then $F'$ intersects another element of $\cE_{=-1}(\cC')$.
By \cite[Tag:0C6A]{Stacks}, it holds that $g=0$, $F''$ is irreducible, and $\cC'_{k',\red}=F'+F''$.
If $F'$ (and hence $F''$) is an element of $\cE_{\bP,1}(\cC')$ (resp.\,$\cE_{\bP,2}(\cC')$), we have $(g,r)=(0,0)$ (resp.\,$(0,2)$) by Lemma \ref{ramificationindex}, which contradicts the assumption of Lemma \ref{forbasechange}.3.
Thus, we have $d=0$ and hence $F'=F'_{1}$, $i_{1}=1$, $k'_{1}=k(F')$.
By using these, ($\sharp$), and $\deg_{k(F')}O_{\cC'}(F')|_{F'}=-1$, we obtain $\deg_{k(F)}O_{\cC}(F)|_{F}=-1$.
Moreover, since we have $F'\in\cE_{\bP,1}(\cC')\cup\cE_{\bP,2}(\cC')$ and $F'=F'_{1}\simeq F\times_{\Spec k(F)}k'_{1}$, we have $F\in\cE_{\bP,1}(\cC')\cup\cE_{0,1'}(\cC')\cup\cE_{\bP,2}(\cC')$.
Thus, we have $F\in\cE_{=-1}(\cC)$.
\end{proof}

\begin{proof}[Proof of Lemma \ref{forbasechange}.4]
We may assume $O_{K}=O_{K}^{h}$ and $O_{K'}=O_{K}^{sh}$.
If we have $F\in(\cE_{\bP,2}(\cC)\cup\cE_{0,1'}(\cC))\cap\cE_{\leq -2}(\cC)$, it holds that $F'=F'_{1}$ and hence $d=0$, $i=1$, and $k(F')=k'_{1}$.
Hence, in this case, $F'_{1}\in\cE_{\leq-2}(\cC')$ follows from ($\sharp$).

Next, suppose $F\in\cE_{\node}(\cC)\cap\cE_{\leq -2}(\cC)$.
Then we have $F', F''\in\cE_{\bP,2}(\cC')$ and $d\geq 1$ (cf.\,\ref{badcurves}).
By using thes and ($\sharp$), we obtain $\deg_{k(F')}O_{\cC'}(F')|_{F'}\leq-2$.
Hence, we have $F'\in\cE_{\leq -2}(\cC')$.

In the rest of the proof, we suppose $F'\in\cE_{\leq -2}(\cC')$.
Note that any prime divisor $F'''$ satisfying $F'''\leq F'_{1}$ is also an element of $\cE_{\leq -2}(\cC')$.
Since $k'$ is separably closed, we have $F'\in\cE_{\bP,2}(\cC')$ by Remark \ref{sepcase}.1.

Suppose $d=2$.
Then $F'$ and hence also any prime divisor $F'''$ satisfying $F'''\leq F'_{1}$ do not intersect other irreducible components of $\cC'_{k'}$ not contained in $F'_{1}$.
It follows from this observation and the fact that $\cC_{k}$ is connected that $\cC_{k}=\Supp F$.
Then, by  Lemma \ref{genus01}.4, it holds that $(g,r)=(1,0)$, which contradicts the assumption of Lemma \ref{forbasechange}.4.

Next, suppose $d=1$.
Then since any prime divisor $F'''$ satisfying $F'''\leq F'_{1}$ intersects exactly one such other prime divisor, it holds that $F''$ is a prime divisor and $i_{1}=2$.
It follows from this and ($\sharp$) that $\deg_{k(F)}O_{\cC}(F)|_{F}\leq-2$.
Write $c'_{1}$ (resp.\,$c'_{2}$) for the unique point of $F'\cap F''$ (resp.\,$(F'\cap F(\cC')^{(2)})\setminus F''$) and $c_{1}$ (resp.\,$c_{2}$) for the image of $c'_{1}$ (resp.\,$c'_{2}$) in $\cC$.
Since $F(\cC')^{(2)}$ is the inverse image of $F(\cC)^{(2)}$ in $\cC'$, we have $\sharp(F\cap F(\cC')^{(2)})=2$.
Since $F$ intersects (resp.\,does not intersect) another irreducible component at $c_{2}$ (resp.\,$c_{1}$), $F$ is normal (resp.\,not normal) at $c_{2}$ (resp.\,$c_{1}$).
Since we have $F'+F''=F'_{1}\simeq F\times_{\Spec k(F)}\Spec k'_{1}$, $\widetilde{F}\to F$ is isomorphic over $F\setminus\{c_{1}\}$ and $c_{2}$ is a $k(\widetilde{F})$-rational point.
From these, the fact that $F'\simeq\bP^{1}_{k(F')}$, and Lemma \ref{forbasechange}.2, we have $\widetilde{F}\simeq\bP^{1}_{k(\widetilde{F})}$.
It follows from these observations and \ref{badcurves} that $F\in\cE_{\node}(\cC)$.
Hence, we have $F\in\cE_{\leq -2}(\cC)$ by Lemma \ref{genus01}.1.

Finally, we consider the case of $d=0$.
In this case, by using a similar argument to those given in the last of the proof of Lemma \ref{forbasechange}.3, we can show that $F\in(\cE_{0,1'}(\cC)\cup\cE_{\bP,2}(\cC))\cap\cE_{\leq -2}(\cC)$.
\end{proof}

\begin{rem}
If $(g,r)=(0,0)$, Lemma \ref{forbasechange}.3 does not hold in general.
For example, the regular n.c.d.\,model
$$\cC:=V_{+}(X^{2}+Y^{2}+tZ^{2})(\subset \Proj \R[[t]][X,Y,Z])\to \Spec \R[[t]]$$
of its generic fiber is minimal since the special fiber is irreducible.
Since $\cC_{\C((t))}$ is isomorphic to $\bP^{1}_{\C((t))}$ over $\C((t))$, any minimal regular model $\cC_{\C((t))}$ is isomorphic to $\bP^{1}_{\C[[t]]}$ (cf.\,\cite[Tag:0CDA]{Stacks}).
Hence, each prime divisor on $\cC_{\C[[t]]}$ whose support is contained in $\cC_{\C}$ is an element of $\cE_{\leq -1}(\cC_{\C[[t]]})$.
\end{rem}

\begin{prop}
\label{basechange}
Suppose $2g+r-2>0$.
\begin{enumerate}
\item
The underlying scheme of the f.s.\,log scheme
$$\cC^{\log}_{\lreg}\times^{\log}_{(\Spec O_{K})^{\log}}(\Spec O_{K'})^{\log}\quad (\text{resp.}\,\cC^{\log}_{\ncd}\times^{\log}_{(\Spec O_{K})^{\log}}(\Spec O_{K'})^{\log})$$
coincides with $\cC_{\lreg,O_{K'}}$ (resp.\,$\cC_{\ncd,O_{K'}}$).
\item
If $k'$ is separable over $k$, we have $\cC_{\ncd, O_{K'}}=\cC'_{\ncd}$ and $\cC_{\lreg,O_{K'}}=\cC'_{\lreg}$.
\item
If $C$ has log smooth reduction, we have $\cC_{\ncd, O_{K'}}=\cC'_{\ncd}$ and $\cC_{\lreg,O_{K'}}=\cC'_{\lreg}$.
\end{enumerate}
\end{prop}

\begin{proof}
Assertion 1 follows from the fact that $(\Spec O_{K'})^{\log}\to(\Spec O_{K})^{\log}$ is a strict morphism of log schemes. 
Assertions 2 follows from Lemmas \ref{forbasechange}.
Assertion 3 follows from Lemmas \ref{F1case}.1, \ref{fundlogsm}.3, and \ref{forbasechange} and Proposition \ref{logsmoothcase}.
\end{proof}

\subsection{Stable models of curves}
\label{stablesubsection}

In this subsection, we discuss stable models and give a generalization of \cite[Theorem (2.4)]{DM}.
Recall that, as explained in Notation-Definition \ref{models}.2, stable models can be regarded as log smooth models.

\begin{prop}
\label{stablecoin}
Suppose that $2g+r-2>0$ and $C^{\log}$ has stable reduction.
Then $\cC_{\lreg}^{\log}$ is a stable model of $C^{\log}$.
Moreover, $\cC_{\lreg}^{\log}$ is an absolutely minimal log regular model.
\end{prop}

\begin{proof}
As explained in Notation-Definition \ref{models}, the stable model is a log smooth model.
Hence, Proposition \ref{stablecoin} follows from Proposition \ref{absoluetness}.3.
\end{proof}

\begin{cor}
\label{stablebasechange}
Suppose $2g+r-2>0$.
\begin{enumerate}
\item
Let $K'$ be a discrete valuation field over $K$ satisfying the assumption of Proposition \ref{basechange}.2.
Then $C^{\log}$ has stable reduction if and only if $(C_{K'})^{\log}$ does.
\item
Suppose that $C^{\log}$ has log smooth reduction.
Let $K'$ be a discrete valuation field over $K$ satisfying the assumption of Proposition \ref{basechange}.3.
Then $C^{\log}$ has stable reduction if and only if $(C_{K'})^{\log}$ does.
\end{enumerate}
\end{cor}

\begin{proof}
Assertions follow from Propositions \ref{basechange}.1, \ref{basechange}.2, \ref{basechange}.3, and \ref{stablecoin}.
\end{proof}

\begin{prop}
\label{prime-to-p}
Suppose that $2g+r-2>0$.
If there exists a log smooth model $\cC^{\log}=(\cC,\cD)$ of $C^{\log}$ satisfying that $\cC_{k}$ is reduced, then $C^{\log}$ has stable reduction.
In particular, if $C^{\log}$ has log smooth reduction, $\cC_{\lreg,k}$ is reduced if and only if $C^{\log}$ has stable reduction.
\end{prop}

\begin{proof}
Note that the second assertion follows from the first assertion and Propositions \ref{logsmoothcase} and \ref{stablecoin}.
To show the first assertion, we may assume that $k$ is algebraically closed by Corollary \ref{stablebasechange}.2.
Since $\cC^{\log}$ is a log smooth model and $\cC_{k}$ is reduced, $\cC_{k}$ has only ordinary double points  by Proposition \ref{genK}.

Let $F$ be a prime divisor such that $F\in\cE_{\bP,1}(\cC)$ or $\Supp F\subset \cD$ is satisfied.
Write $F'$ for the (unique) prime divisor on $\cC$ which intersects $F$ whose support is contained in $\cC_{k}\cup\cD$.
Note that $\Supp F'\subset\cC_{k}$ holds by Lemma \ref{genus01}.1.
Write $c$ for the unique point of $F\cap F'$ and $E$ for the fiber of the morphism $\widetilde{\cC}\to\cC$ over $c$.
Since the multiplicity in $\widetilde{\cC}_{k}$ of the strict transform of $F'$ in $\widetilde{\cC}$ is $1$, we can show that the multiplicity of every irreducible component of $E$ is $1$ by using the same calculation as in (\ref{chainineq1}) and (\ref{chainineq2}).
It follows from this and by Lemmas \ref{ramificationindex} and \ref{fundlogsm} that $\cD$ is \'etale over $\Spec O_{K}$.
Suppose $F\in\cE_{\bP,1}(\cC)$ and write $\widetilde{F}$ for the strict transform of $F$.
Then we have $\widetilde{F}\in\cE_{\bP,1}(\cC)$ by Lemma \ref{Eproperty} and the multiplicity of $\widetilde{F}$ in $\widetilde{\cC}_{k}$ is also $1$.
Again by using the same calculation as in (\ref{chainineq1}), we obtain $(\widetilde{F}\cdot\widetilde{F})=-1$.
Then we can contract $\widetilde{\cC}$ along $\widetilde{F}\cup E$ to a regular n.c.d.\,model.
By using this contraction, we can contract $\cC$ along $F$ to a log smooth model whose special fiber is reduced.
Hence, we may assume $\cE_{\bP,1}(\cC)=\emptyset$.

Suppose that we have an element $F\in\cE_{\bP,2}(\cC)$.
Write $c_{1},c_{2}$ for the element of $F\cap F(\cC)^{(2)}$, $F_{1}$ and $F_{2}$ for the prime divisor on $\cC$ whose supports are contained in $\cC_{k}\cup\cD$ and intersecting $F$ at $c_{1}$ and $c_{2}$, respectively, and  $E_{1}$ and $E_{2}$ for the fiber of the morphism $\widetilde{\cC}\to\cC$ at $c_{1}$ and $c_{2}$, respectively.
Since we have $2g+r-2>0$, we may assume $\Supp F_{1}\subset\cC_{k}$ by Lemma \ref{genus01}.1.
Write $\widetilde{F}$ for the strict transform of $F$ in $\widetilde{\cC}$ (, which is contained in $\cE_{\bP,2}(\widetilde{\cC})$ by Lemma \ref{Eproperty}).
Note that the multiplicity of $\widetilde{F}$ in $\widetilde{\cC}_{k}$ is also $1$.
By using the same calculation as in (\ref{chainineq2}), we obtain $(\widetilde{F}\cdot\widetilde{F})\leq-2$.
Then we can contract $\widetilde{\cC}$ along $\widetilde{F}\cup E_{1}\cup E_{2}$ to a log regular model by Proposition \ref{logexcpetionaldivisor}.
By using this contraction, we can contract $\cC$ along $F$ to a log smooth model whose special fiber is reduced by Lemma \ref{logblowupet}.
Hence, we may assume $\cE_{\bP,2}(\cC)=\emptyset$.
It follows from these observations and Proposition \ref{genK} that $C^{\log}$ has a stable model.
\end{proof}

We give the proof of ``only if" part and a generalization of Corollary \ref{intromain1}.

\begin{thm}[cf.\,{\cite[Theorem (2.4)]{DM}}, Corollary \ref{intromain1}]
\label{generalmain}
Suppose that $2g+r-2>0$ and $D$ is \'etale over $\Spec K$.
$C^{\log}$ has stable reduction if and only if the following hold:
$J$ has stable reduction and the normalization of $\Spec O_{K}$ in each irreducible component of $D$ is unramified over $\Spec O_{K}$.
\end{thm}

\begin{proof}
First, we suppose that $J$ has stable reduction and the normalization of $\Spec O_{K}$ in each irreducible component of $D$ is unramified over $\Spec O_{K}$.
We show $C^{\log}$ has stable reduction.
By considering the discussion in the proof of ``if" of Corollary \ref{intromain1} (even in the case where $D\neq\emptyset$), we may assume that $k$ is algebraically closed, $O_{K}$ is strictly henselian, and $\cC_{\lreg}^{\log}$ is a log smooth model.
Then $D$ is the disjoint union of finite copies of $\Spec K$.
If $g=0$ (resp.\,$g=1$; $g\geq 2$), let $D_{0}$ be a reduced effective divisor of degree $3$ satisfying $D_{0}\leq D$ (resp.\,a prime divisor satisfying $D_{0}\leq D$; the divisor $0$) and $(\cC_{0},\cD_{0})$ the minimal regular n.c.d.\,model of $(C,D_{0})$.
Note that $\cC_{0,k}$ is reduced since $(\cC_{0},\cD_{0})$ is isomorphic to $(\bP^{1}_{O_{K}},[0]+[1]+[\infty])$ (resp.\,by, for example, \cite[THEOREM (3.8)]{Sai1} and Lemma \ref{ramificationindex}; by \cite[Proposition (2.3)]{DM} and \cite[Theorem (2.4)]{DM}).
Since sections of the structure morphisms of regular models over $O_{K}$ factor through the smooth locus of the models (cf.\,the proof of \cite[Tag:0CE8]{Stacks} or \cite[Proposition 3.1.2]{BLR}), we obtain a regular n.c.d.\,model of $C^{\log}$ as an iterate of a blow-up at a smooth closed point.
Therefore, there exists a regular n.c.d.\,model whose special fiber is reduced, which shows $\cC_{\ncd,k}$ is also reduced.
Since $k$ is algebraically closed, a prime divisor $F$ on $\cC_{\ncd}$ isomorphic to $\bP^{1}_{k}$ whose support is contained in $\cC_{\ncd,k}$ satisfies $\sharp(F\cap F(\cC_{\ncd})^{(2)})=-(F\cdot F)$.
From this observation and the construction of $\cC_{\lreg}$, $\cC_{\lreg}$ satisfies the third condition in the definition of stable curves in \ref{curvedfn}.
Moreover, since the structure morphism $\cC_{\lreg}^{\log}\to(\Spec O_{K})^{\log}$ is log smooth by Proposition \ref{logsmoothcase} and $\cC_{\lreg,k}$ is reduced, $\cC_{\lreg,k}$ has only ordinary double points by Proposition \ref{genK}.
Therefore, $\cC_{\lreg}^{\log}$ is a stable curve over $(\Spec O_{K})^{\log}$.

Next, suppose that $C^{\log}$ has stable reduction (and hence $(\cC_{\lreg})^{\log}$ is a stable model by Proposition \ref{stablecoin}).
The desired assertion for $D$ follows from the definition of stable curves.
To show that $J$ has stable reduction, we may assume that $O_{K}$ is strictly henselian.
If $g=0$, we have $J\simeq\Pic^{0}_{C/K}\simeq \Spec K$.
If $g=1$, by fixing an arbitrary $K$-rational point of $C$ (which exists since $D\neq\emptyset$ and $O_{K}$ is strictly henselian), we can show that $(C\simeq)J$ has stable reduction.
Thus, we may assume $g\geq 2$.

Note that $\cC_{\lreg,k}$ is reduced and smooth around $\cD\cap F(\cC_{\lreg})^{(2)}$.
Therefore, $(\cC_{\lreg},\emptyset)$ is a log smooth model of $(C,\emptyset)$.
By these observation, $C$ has stable reduction by Proposition \ref{prime-to-p}.
Write $\cC'^{\log}$ for the stable model of $(C,\emptyset)$.
By \cite[Proposition 4.3]{D} (cf.\,\cite[Theorem 9.4.1]{BLR} and Remark \ref{betterBLR}), $\Pic^{0}_{\cC'/S}$ is a semi-abelian scheme over $S$.
Since the generic fiber of $\Pic^{0}_{\cC'/S}$ is isomorphic to $J$, we have $\cJ^{0}\simeq \Pic^{0}_{\cC'/S}$ by \cite[Proposition 7.4.3]{BLR}, where $\cJ^{0}$ is the open subgroup scheme whose special fiber is the identity component of $\cJ_{k}$.
Therefore, $J$ has stable reduction.
\end{proof}

\begin{rem}
\label{betterBLR}
In the second paragraph of the proof of Theorem \ref{generalmain}, one might think applying \cite[Theorem 9.4/1]{BLR} would make the discussion simpler.
In \cite[Theorem 9.4/1]{BLR} and \cite[Proposition 9.4/4]{BLR}, (proper) semi-stable curves are treated.
(Note that, in \cite[Proposition 4.3]{D}, only proper stable curves are treated.)
Moreover, in \cite[Definition 9.4/1]{BLR}, the definition of a (proper) semi-stable curve does not require the third and fourth conditions in the definition of stable curves in \ref{curvedfn}.
However, in the proof of \cite[Theorem 9.4/1]{BLR} and \cite[Proposition 9.4/4]{BLR}, \cite[Corollary (1.7)]{DM} is used.
Because \cite[Corollary (1.7)]{DM} can be applied to only proper stable curves, the author of the present paper considers the discussion in our proof is needed.
\end{rem}

\begin{cor}[cf.\,{\cite[Proposition 2.3 and Theorem (2.4)]{DM}}, {\cite[Exp.IX, Proposition 3.9]{SGA7}}, {\cite[Theorem 4.8]{Sai1}}, {\cite[Tag:0E8D, 0CDH, and 0CDG]{Stacks}}, and Remark \ref{why(F)}.1]
\label{monodromy}
The following are equivalent:
\begin{enumerate}
\item
$C^{\log}$ has stable reduction.
\item
$\cC_{\lreg}^{\log}$ is a stable model of $C^{\log}$.
\item
$C^{\log}$ has a semi-stable model, i.e., a regular n.c.d.\,model whose special fiber has only ordinary double points.
(This condition is equivalent to the condition that $C^{\log}$ has semi-stable reduction in the usual sense (cf.\,\cite[Definition 1.6]{Sai2}.)
\item
$\cC_{\ncd}^{\log}$ is a semi-stable model.
\item
$C$ has a semi-stable model and the normalization of $\Spec O_{K}$ in each irreducible component of $D$ is unramified over $\Spec O_{K}$.
\item
The action of $I_{K}$ on $H^{1}_{\text{\'et}}((C\setminus D)_{K^{\sep}},\Z_{l})$ is unipotent and $D$ is unramified over $\Spec K$.
\item
$J$ has stable reduction and  the normalization of $\Spec O_{K}$ in each irreducible component of $D$ is unramified over $\Spec O_{K}$.
\end{enumerate}
\end{cor}

\begin{proof}
We have 4$\Rightarrow$3$\Rightarrow$5.
1$\Leftrightarrow$2 follows from Proposition \ref{stablecoin}.
6$\Leftrightarrow$7 follows from \cite[Exp.IX, Proposition 3.9]{SGA7} and the exact sequence (H).
1$\Leftrightarrow$7 follows from Theorem \ref{generalmain}.
1$\Leftrightarrow$4 follows from Remark \ref{why(F)}.1.
5$\Rightarrow$ 7 follows from \cite[Example 9.2/8]{BLR}, \cite[Theorem 9.5/4]{BLR}, and \cite[Remark 9.5/5]{BLR}.
(This also follows from the discussions in the second and third paragraphs of the proof of Theorem \ref{generalmain}.)
\end{proof}

\subsection{Multiplicities divisible by $p$}
\label{psubsection}

In this subsection, we discuss certain prime divisors whose supports are contained in the special fibers of log smooth models, that is, prime divisors with multiplicities divisible by $p$.
We also give a generalization of \cite[THEOREM (3.11)]{Sai1}, \cite[Theorem 4.2]{Sai2}.

First, we start by generalizing the discussions in the proof of \cite[Proposition 2.4]{MS} to our case.
(In \cite{MS}, treated log schemes are Zariski log regular, $k$ is assumed to be perfect, and $D$ is assumed to be $\emptyset$.
Note that, in \cite{MS}, higher dimensional varieties are treated.)

\begin{sub}[``$\omega_{\fp}$" in \cite{MS}]
\label{genMS}
Let $\cC^{\log}=(\cC,\cD)$ be a log regular model of $C^{\log}$ and $F$ an integral subscheme of $\cC_{k}$ of dimension $1$ whose multiplicity $m$ in $\cC_{k}$ is divisible by $p$.
Suppose that, around $F$, $\cC_{k}\cup\cD$ is a normal crossing divisor on $\cC$ and $\cC^{\log}$ is log smooth over $(\Spec O_{K})^{\log}$.

First, we see that $\cC^{\log}$ is Zariski log regular around $F$, or equivalently, $\cC_{k}\cup\cD$ is a simple normal crossing divisor on $\cC$ around $F$ (cf.\,Theorems \ref{logsummary}.3).
Let $c$ be a point of $F\cap F(\cC)^{(2)}$ such that $c$ is not contained in other irreducible components of $\cC_{k}\cup\cD$.
Write $\cC'^{\log}$ the log blow-up of $\cC^{\log}$ at $c$ (cf.\,Example \ref{maximallogblowup}), $E$ for the exceptional divisor, and $F'$ for the strict transform of $F$ in $\cC'$.
Then $\cC'^{\log}$ is log smooth over $(\Spec O_{K})^{\log}$ around $E$.
On the other hand, the multiplicity of $E$ in $\cC_{k}$ coincides with $2m$, which contradicts that $\cC'$ is log smooth over $(\Spec O_{K})^{\log}$ at $E\cap F'$ by Lemma \ref{logsmoothchart}.

Furthermore, it follows that $\cD\cap F=\emptyset$ from Lemmas \ref{ramificationindex} and \ref{fundlogsm}.
By Theorems \ref{logsummary}.3 and \ref{logsummary}.4, $F$ is normal at any point of $F\cap F(\cC)^{(2)}$.
Moreover, from this, we may assume that $F$ is smooth over $k$ by Lemma \ref{F1case}.1 and Lemma \ref{logsmoothchart}.
Then $(F,F\setminus F(\cC)^{(2)})$ is a toric pair and, again by Lemma \ref{logsmoothchart}, $F^{\log}:=(F,F\setminus F(\cC)^{(2)})$ is log smooth over $\Spec k$ (with the trivial log structure).
Therefore, we have $\Omega^{1}_{F^{\log}/k}=\Omega^{1}_{F/k(F)}(F\cap F(\cC)^{(2)})$.

By using the discussion of the proof of \cite[Proposition 2.4]{MS}, we define
$$\omega_{F}:=``\text{the image of }d\log\varpi(=\frac{d\varpi}{\varpi})"\in H^{0}(F,\Omega^{1}_{F^{\log}/k}(F\cap F(\cC)^{(2)})),$$
for which is written $\omega_{\fp}$ in \cite[Proposition 2.4]{MS}, in the following.
Let $x$ be a closed point of $F\setminus F(\cC)^{(2)}$ (resp.\,$F\cap F(\cC)^{(2)}$).
Take an affine open neighborhood $x\in U=\Spec A\subset \cC$ and elements $u,\varpi_{1}\in A$ (resp.\,elements $u,\varpi_{1},\varpi_{2}\in A$) satisfying the following:
(i) $U$ does not intersect any irreducible component of $\cC_{k}$ not containing $x$.
(ii) $F\simeq\Spec A/\varpi_{1}$.
(iii) If $x\in F\cap F(\cC)^{(2)}$, we have $F'\simeq A/\varpi_{2}$, where $F'(\neq F)$ is a prime divisor containing $x$.
(iv) $\varpi=u\varpi_{1}^{e_{1}}$ (resp.\,$\varpi=u\varpi_{1}^{e_{1}}\varpi_{2}^{e_{2}}$), where $e_{1}$ and $e_{2}$ are the multiplicities of $F$ and $F'$ in $\cC_{k}$, respectively.
(v) $u\in A^{\ast}$.
Then we have an equation obtained by eliminating the denominators of both sides of the equation
\begin{align}
\label{logcalcu}
\frac{d\varpi}{\varpi}=\frac{du}{u}+e_{1}\frac{d\varpi_{1}}{\varpi_{1}}\quad(\text{resp.\,}\frac{d\varpi}{\varpi}=\frac{du}{u}+e_{1}\frac{d\varpi_{1}}{\varpi_{1}}+e_{2}\frac{d\varpi_{2}}{\varpi_{2}})
\end{align}
by multiplying $\varpi(=u\varpi_{1}^{e_{1}})$ (resp.\,$\varpi(=u\varpi_{1}^{e_{1}}\varpi_{2}^{e_{2}})$).
We define
$$\omega_{F}|_{F\cap U}:=\frac{d\overline{u}}{\overline{u}}
\quad(\text{resp.\,}\omega_{F}|_{F\cap U}:=\frac{d\overline{u}}{\overline{u}}+e_{2}\frac{d\overline{\varpi_{2}}}{\overline{\varpi_{2}}})\in\Gamma(F\cap U,\Omega^{1}_{F^{\log}/k}(F\cap F(\cC)^{(2)})).$$
The well-definedness of the definition of $\omega_{F}|_{F\cap U}$ follows from (\ref{logcalcu}) and the fact that $e_{1}$ (resp.\,$e_{2}$) is divisible (resp.\,\textbf{not} divisible) by $p$ (which follows from Lemma \ref{logsmoothchart}).
By using the same discussion, we can glue each $\omega_{F}|_{F\cap U}$ and obtain the desired $\omega_{F}$.

Finally, we see that $\omega_{F}$ generates $\Omega^{1}_{F^{\log}/k}$.
By construction, $\omega_{F}$ generates $\Omega^{1}_{F^{\log}/k}$ at points in $F\cap F(\cC)^{(2)}$.
We show that $\omega_{F}$ generates $\Omega^{1}_{F^{\log}/k}$ at an arbitrary closed point $x\in F\setminus F(\cC)^{(2)}$ \'etale locally over $\cC$.
Let $(\cU,\N\to Q\to \Gamma(\cU,O_{\cU}))$ be a chart of the morphism $\cC^{\log}\to(\Spec O_{K})^{\log}$ satisfying (I) and (II) in \ref{chardefinition} and $y$ a point of $\cU$ over $x$.
Fix a lift $\varpi'_{1}$ of a generator of $Q/Q^{\ast}\simeq (M_{\cU}/O_{\cU}^{\ast})_{y}(\simeq \N)$ in $Q$ and shrink $\cU$ so that $\cU(=:\Spec A')$ is affine and $A'/\varpi'_{1}\simeq F|_{\cU}$ is satisfied.
Then we have a decomposition $Q=Q^{\ast}\langle\varpi'_{1}\rangle$.
Write $u'$ for the element of $Q^{\ast}$ satisfying $\varpi=u'(\varpi'_{1})^{e_{1}}$.
Since $\cU_{k}\to \Spec k[Q\setminus Q\varpi](=\Spec k[Q\setminus Q\varpi'_{1}])$ is \'etale by (I) in \ref{chardefinition}, the induced morphism
$$(\cU_{k,\red}\simeq)F|_{\cU}\to \Spec k[Q^{\ast}/Q^{\ast}_{\tor}](\simeq (\Spec k[Q\setminus Q\varpi'_{1}])_{\red})$$
is also \'etale.
From this, we have $Q^{\ast}/Q^{\ast}_{\tor}\simeq\Z$.
Write $\overline{u'}$ for the image of $u'$ in $Q^{\ast}/Q^{\ast}_{\tor}$.
Since $\sharp(Q^{\gp}/\langle u'(\varpi'_{1})^{e_{1}}\rangle)_{\tor}$ is not divisible by $p$ by (II) in \ref{chardefinition} and we have an identification $Q=Q^{\ast}_{\tor}\times(Q^{\ast}/Q^{\ast}_{\tor})\times\langle\varpi'_{1}\rangle$ and $p|e_{1}$, it holds that $\sharp(Q^{\ast}/Q^{\ast}_{\tor})/\langle \overline{u'}\rangle$ is not divisible by $p$.
Hence,
$$\frac{d\overline{u'}}{\overline{u'}}\in H^{0}(\Spec k[Q^{\ast}/Q^{\ast}_{\tor}],\Omega^{1}_{\Spec k[Q^{\ast}/Q^{\ast}_{\tor}]/k})$$
generates $\Omega^{1}_{\Spec k[Q^{\ast}/Q^{\ast}_{\tor}]/k}$.
Since $F|_{\cU}\to \Spec k[Q^{\ast}/Q^{\ast}_{\tor}]$ is \'etale, $\frac{d\overline{u'}}{\overline{u'}}$ also generates $\Omega^{1}_{F|_{\cU}/k}$.
By using this and the last discussion in the previous paragraph, we can show that $\omega_{F}$ generates $\Omega^{1}_{F^{\log}/k}$.
\end{sub}

\begin{prop}
\label{pstr}
Let $\cC^{\log}$ and $F$ be those in \ref{genMS} without the assumption that $\cC_{k}\cup\cD$ is a normal crossing divisor on $\cC$ around $F$.
\begin{enumerate}
\item
$\cD$ does not intersect $F$ and one of the following conditions is satisfied:
\begin{itemize}
\item[($(0,2)$)]
$F\in\cE_{\bP,2}(\cC)\cup\cE_{0,1'}(\cC)$ and $k(F)$ is separable over $k$.
\item[($\node$)]
$F=\cC_{k,\red}$, $\sharp F(\cC)^{(2)}=1$, the normalization of $F$ is a smooth curve over $k$ of type $(0,0)$, $F$ is smooth over $F\setminus F(\cC)^{(2)}$, and $F$ has an ordinary double point at the point of $F(\cC)^{(2)}$.
\item[($(1,0)$)]
$F=\cC_{k,\red}$, $k=k(F)$, and $F$ is a smooth curve of type $(1,0)$ over $k$ (in particular, we have $\sharp F(\cC)^{(2)}=0$).
\end{itemize}
If one of conditions (node) or ((1,0)) is satisfied, we have $(g,r)=(1,0)$.
\item
Suppose that $\cC_{k}\cup\cD$ is a normal crossing divisor on $\cC$ around $F$.
Then $\cD\cap F=\emptyset$ and one of conditions ($(0,2)$) or ($(1,0)$) is satisfied.
In particular, $F$ is smooth over $k$ and, at each point of $F\cap F(\cC)^{(2)}$, $F$ intersects another irreducible component of $\cC_{k}$ whose multiplicity is \textbf{not} divisible by $p$.
\end{enumerate}
\end{prop}

\begin{proof}
First, we assume that $\cC_{k}\cup\cD$ is a normal crossing divisor on $\cC$ around $F$ and show assertion 2.
Note that the second assertion of assertion 2 follows from the first assertion of assertion 2 and Lemma \ref{logsmoothchart}.
By the discussions in \ref{genMS}, $F$ is smooth over $k$, the residue fields of points of $F\cap F(\cC)^{(2)}$ are separable over $k$, $\cD\cap F=\emptyset$, and $\Omega^{1}_{F/k}(F\cap F(\cC)^{(2)})$ is isomorphic to $O_{F}$.
Therefore, the log regular curve $F^{\log}$ over $k(F)$ is of type (0,2) or (1,0).
Assertion 2 follows from these observations and the fact that $\cC_{k}$ is geometrically connected over $k$.

Next, we show assertion 1.
(We do not assume that $\cC^{\log}$ is a regular n.c.d.\,model).
Let $\widetilde{\cC}^{\log}$ be the minimal desingularization  of $\cC^{\log}$ and write $\widetilde{F}$ for the strict transform of $F$ in $\widetilde{\cC}$.
Then $\widetilde{\cC}^{\log}$ is log smooth over $(\Spec O_{K})^{\log}$ around $\widetilde{F}$ by Lemma \ref{logblowupet}.1.
By assertion 2, $\widetilde{F}$ satisfies one of conditions ((0,2)) or ((1,0)).
In particular, $\widetilde{F}\to F$ is the normalization morphism.
Therefore, if $\widetilde{F}$ is isomorphic to $F$, assertion 1 holds.
Suppose that the morphism $\widetilde{F}\to F$ is not isomorphic.
Then we have $\sharp(\widetilde{F}\cap F(\widetilde{\cC})^{(2)})\geq1$, which shows that $\widetilde{F}$ satisfies condition ((0,2)).
Since $\widetilde{F}\to F$ is not isomorphic, $F\cap F(\cC)^{(2)}$ consists of a unique point $c$ where $\widetilde{F}\to F$ is not isomorphic.
From this, the geometrical connectedness of $\cC_{k}$ over $k$, and condition ((0,2)), it follows that $F=\cC_{k,\red}$, $\sharp F(\cC)^{(2)}=1$, and $k(F)=k$.
From these observations, \ref{node-like}, and Proposition \ref{genK}, it follows that condition (node) is satisfied.
Moreover, in this case, we have $(g,r)=(1,0)$ by Lemma \ref{genus01}.3.
This completes the proof of assertion 1.
\end{proof}

\begin{cor}[cf.\,{\cite[THEOREM (3.11)]{Sai1}}, {\cite[Theorem 4.2]{Sai2}}, and {\cite[Theorem 1]{Lo}}]
\label{wildmonodromy}
Suppose $2g+r-2>0$.
The following are equivalent:
\begin{enumerate}
\item
$C^{\log}$ has log smooth reduction.
\item
$\cC_{\lreg}^{\log}$ is a log smooth model.
\item
$\cC_{\ncd}^{\log}$ is a log smooth model.
\item
Let $F$ be a prime divisor on $\cC_{\ncd}$ whose support is contained in $\cC_{\ncd,k}\cup\cD_{\ncd}$.
If $\Supp F\subset \cD_{\ncd}$, $F$ is at most tamely ramified over $\Spec O_{K}$.
If $\Supp F\subset \cC_{\ncd,k}$ and the multiplicity of $F$ is \textbf{not} divisible by $p$, $F\setminus F(\cC_{\ncd})^{(2)}$ is smooth over $k$ and the residue field of any point of $F\cap F(\cC_{\ncd})^{(2)}$ is separable over $k$.
If $\Supp F\subset \cC_{\ncd,k}$ and the multiplicity of $F$ is divisible by $p$, the following are satisfied:
$F$ is an element of $\cE_{0,1'}(\cC_{\ncd})\cup\cE_{\bP,2}(\cC_{\ncd})$.
(Note that, from these conditions, any elements of $F\cap F(\cC_{\ncd})^{(2)}$ is an intersection point of $F$ and another irreducible component of $\cC_{k}$.)
Moreover, $k(F)$ is separable over $k$.
Furthermore, the multiplicity of any irreducible component of $\cC_{k}$ intersecting such $F$ is \textbf{not} divisible by $p$.
\item
The curve $(C,\emptyset)$ (of type $(g,0)$) has a log smooth model and each connected component of the normalization of $\Spec O_{K}$ in $D$ is at most tamely ramified over $\Spec O_{K}$.
\item
The action of $P_{K}$ on $H^{1}_{\text{\'et}}((C\setminus D)_{K^{\sep}},\Z_{l})$ is trivial and $D$ is \'etale over $\Spec K$.
\item
The action of $P_{K}$ on $H^{1}_{\text{\'et}}(J,\Z_{l})$ is trivial and each connected component of the normalization of $\Spec O_{K}$ in $D$ is at most tamely ramified over $\Spec O_{K}$.
\item
There exists a finite (at most) tamely ramified extension $K'/K$ such that $(C_{K'},D_{K'})$ has stable reduction.
\end{enumerate}
\end{cor}

\begin{proof}
3$\Rightarrow$5 follows from Lemmas \ref{F1case}.2 and \ref{fundlogsm} and Proposition \ref{genK}.
1$\Leftrightarrow$7$\Leftrightarrow$8 follows from \cite[Theorem 4.2]{Sai2} (which we can use because \cite[Proposition 1.13]{Sai2} is valid by Corollary \ref{monodromy}).
We can obtain 6$\Leftrightarrow$7 by using (H).
1$\Leftrightarrow$2$\Leftrightarrow$3 follows from Proposition \ref{logsmoothcase}.
3$\Rightarrow$4 follows from Lemmas \ref{F1case}.2, \ref{logsmoothchart}, \ref{fundlogsm} and Proposition \ref{pstr}.2.

Next, we show 4$\Rightarrow$2.
Suppose that condition 4 is satisfied.
Then, by Lemmas \ref{F1case}.2 and \ref{logsmoothchart}, $\cC^{\log}_{\ncd}$ is log smooth over $(\Spec O_{K})^{\log}$ at points on irreducible divisors whose multiplicity in $\cC_{k}$ is not divisible by $p$.
By the proof of Lemma \ref{logblowupet}.1, it holds that condition 2 is satisfied.

Finally, we show 5$\Rightarrow$8.
By replacing $K$ by a suitable finite tamely ramified extension field $K'$, we may assume that $D$ is isomorphic to the disjoint union of finite copies of $\Spec K$.
By Lemma \ref{logblowupet}.1, we obtain a regular n.c.d.\,model $\cC^{\log}$ of the log regular curve $(C,O_{C}^{\ast})$ which is log smooth over $(\Spec O_{K})^{\log}$.
Since sections of the structure morphisms of regular models over $O_{K}$ factor through the smooth locus of the models (cf.\,the proof of \cite[Tag:0CE8]{Stacks} or \cite[Proposition 3.1.2]{BLR}), we obtain a regular n.c.d.\,model of $C^{\log}$ as an iterate of a blow-up at a smooth closed point.
Since the resulting model is log smooth over $(\Spec O_{K})^{\log}$ by Lemma \ref{logblowupet}.2, condition 8 is satisfied.
\end{proof}

\begin{rem}
In the proof of Corollary \ref{wildmonodromy} , one can prove 5$\Rightarrow$6 by applying the proper base change theorem in the theory of log \'etale cohomology (\cite{Nak}).
\end{rem}

\begin{prop}[cf.\,{\cite[Theorem 1.8]{Sai2}}, {\cite[Theorem 4.2]{Sai2}}, and {\cite[Section 7]{H}}]
\label{genH}
Suppose that $2g+r-2>0$ and $C^{\log}$ has log smooth reduction.
Then the least common multiple $e$ of the multiplicities of irreducible components of $\cC_{\lreg,k}$ is prime to $p$.
Moreover, for any extension of discrete valuation fields $L/K$ whose ramification index is $e'$, $e'$ is divisible by $e$ if and only if $C_{L}$ has stable reduction.
\end{prop}

\begin{proof}
The fact that $e$ is not divisible by $p$ follows from Theorem \ref{logreglem}, Notation-Definition \ref{lregnota}, and the equivalence 1$\Leftrightarrow$4 in Corollary \ref{wildmonodromy}.
If $e$ divides $e'$, the special fiber of $\cC_{\lreg}^{\log}\times^{\log}_{(\Spec O_{K})^{\log}}(\Spec O_{L})^{\log}$ is reduced by Abhyankars's lemma, which shows that $(C_{L})^{\log}$ has stable reduction by Proposition \ref{prime-to-p}.

Suppose that $e$ does not divide $e'$ and $(C_{L})^{\log}$ has stable reduction.
We show a contradiction.
We may assume that $K$ is complete and $k$ is algebraically closed by Proposition \ref{basechange}.3.
In this case, Proposition \ref{genH} follows from \cite[Theorem 7.5]{H}.
\end{proof}

{\em Email address}: \texttt{nagachi@kurims.kyoto-u.ac.jp}

\end{document}